\documentclass[english]{article}

\usepackage{amsthm,amsmath,amssymb}
\usepackage{url} 
\usepackage[authoryear]{natbib}
\usepackage[colorlinks=true,linkcolor=black,citecolor=black,urlcolor=black,breaklinks]{hyperref}
\usepackage{babel}
\usepackage{caption}
\usepackage{subcaption}
\usepackage{graphicx} 
\usepackage{epstopdf}
\usepackage{color}
\usepackage{framed}
\usepackage{lscape}
\usepackage{rotating}
\usepackage{algorithm}
\usepackage[noend]{algpseudocode}
\usepackage{custom_tex}
\usepackage{grffile}
\usepackage{multirow}
\usepackage{xcolor}
\usepackage{setspace}
\usepackage{comment}

\newcommand{\ep}{\varepsilon}

\newcommand{\E}{\mathbb{E}}
\newcommand{\R}{\mathbb{R}}

\newcommand{\V}{\textnormal{Var}}

\newcommand{\op}{\textnormal{op}}
\newcommand{\supp}{\textnormal{supp}}
\newcommand{\bitem}{\begin{itemize}}
\newcommand{\eitem}{\end{itemize}}
\newcommand{\benum}{\begin{enumerate}}
\newcommand{\eenum}{\end{enumerate}}
\newcommand{\beq}{\begin{equation}}
\newcommand{\eeq}{\end{equation}}
\newcommand{\beqs}{\begin{equation*}}
\newcommand{\eeqs}{\end{equation*}}

\newtheorem{thms}{Theorem}[section]

\newtheorem{lem}[thms]{Lemma}

\newtheorem{cor}[thms]{Corollary}
\newtheorem{prop}[thms]{Proposition}
\theoremstyle{thms}

\usepackage{geometry} 
\begin{document}
\title{PCA from noisy, linearly reduced data: the diagonal case}

\author{Edgar Dobriban\footnote{Department of Statistics, Stanford University. E-mail: \texttt{dobriban@stanford.edu}. Supported in part by NSF grant DMS-1407813, and by an HHMI International Student Research Fellowship.},
William Leeb\footnote{Program in Applied and Computational Mathematics, Princeton University.  E-mail: \texttt{wleeb@math.princeton.edu}. Supported by the Simons Collaborations on Algorithms and Geometry.},
and Amit Singer\footnote{Department of Mathematics, and Program in Applied and Computational Mathematics, Princeton University. E-mail: \texttt{amits@math.princeton.edu}. Partially supported by Award Number
R01GM090200 from the NIGMS, FA9550-12-1-0317 from AFOSR, Simons Foundation Investigator Award and Simons Collaborations on Algorithms and Geometry, and the Moore Foundation Data-Driven Discovery Investigator Award.}}
\date{}
\maketitle

\begin{abstract}
Suppose we observe data of the form $Y_i = D_i (S_i + \varepsilon_i)   \in \mathbb{R}^p$ or $Y_i = D_i S_i + \varepsilon_i   \in \mathbb{R}^p$, $i=1,\ldots,n$, where $D_i  \in \mathbb{R}^{p\times p}$ are known diagonal matrices, $\ep_i$ are noise, and we wish to perform principal component analysis (PCA) on the unobserved signals $S_i \in \mathbb{R}^p$. The first model arises in missing data problems, where the $D_i$ are binary. The second model captures noisy  deconvolution problems, where the $D_i$ are the Fourier transforms of the convolution kernels. It is often reasonable to assume the $S_i$ lie on an unknown low-dimensional linear space; however, because many coordinates can be suppressed by the $D_i$, this low-dimensional structure can be obscured.

We introduce \emph{diagonally reduced spiked covariance models} to capture this setting. We characterize the behavior of the singular vectors and singular values of the data matrix under high-dimensional asymptotics where $n,p\to\infty$ such that $p/n\to\gamma>0$. Our results have the most general assumptions to date \emph{even without diagonal reduction}.  Using them, we develop optimal eigenvalue shrinkage methods for covariance matrix estimation and optimal singular value shrinkage methods for data denoising.

Finally, we characterize the error rates of the empirical Best Linear Predictor (EBLP) denoisers. We show that, perhaps surprisingly, their optimal tuning depends on whether we denoise in-sample or out-of-sample, but the optimally tuned mean squared error is the same in the two cases. 

\end{abstract}
 \section{Introduction}

 Principal component analysis (PCA) is a classical statistical method that decomposes a collection of datapoints $s_1,\dots,s_n \in \mathbb{R}^p$ as a linear combination of vectors that account for the most variability \citep[e.g.,][]{jolliffe2002principal,anderson1958introduction}. More formally, if $s_1,\dots,s_n$ are drawn from a probability distribution with mean zero and covariance matrix $\Sigma_S$, then the principal components (PCs) of the distribution are the eigenvectors $u_1,\dots, u_p$ of $\Sigma_S$. 
Typically, we approximate the distribution by projecting it onto the PCs with the largest eigenvalues.

 A more specific model arising in many applications is the \emph{spiked covariance model} \citep{johnstone2001distribution}. First, the signal $S_i$ is a linear combination of $r$ fixed but unobserved orthonormal PCs $u_k$: 
        \begin{align}
        \label{eq:S_i}
        S_i = \sum_{k=1}^r \ell_k^{1/2} z_{ik} u_k,
        \end{align}
 where $r$ is a fixed parameter (independent of $n$ and $p$) and $z_{ik}$ are iid standardized random variables. Here $\ell_k$ are the eigenvalues, or equivalently the variances along the PCs $u_k$.  Second, the observations are $X_i = S_i + \ep_i$, where $\ep_i$ is noise with iid standardized entries. The spiked covariance model has been widely studied in probability and statistics \cite[e.g.,][etc]{baik2005phase, baik2006eigenvalues, paul2007asymptotics, benaych2012singular}; see also \cite{paul2014random,yao2015large}. 

 This paper considers the setting when the vectors $S_i$ are not only corrupted by noise, but also \textit{linearly reduced}. This means that for given matrices $D_i \in \mathbb{R}^{q_i \times p}$, we observe either
        \begin{align}
        \label{po_def}
        Y_i = D_i X_i = D_i S_i + D_i \ep_i.
        \end{align}
 or
        \begin{align}
        \label{po_def2}
        Y_i = D_i S_i + \ep_i.
        \end{align}
 We think of the \textit{reduction matrix} $D_i$ as either a projection matrix or a linear filter reducing the information that we observe. In general, it will not be possible to reconstruct a vector $v$ from $D_i v$.

 We refer to model $\eqref{po_def}$ as the \textit{reduced-noise model}, and to model \eqref{po_def2} as the \textit{unreduced-noise model}. In the reduced-noise model, both the signal and noise are reduced, while in the unreduced-noise case, only the signal is. These models generalize the spiked covariance model, and arise naturally in several settings. For instance:

 \begin{enumerate}
 \item
 {\bf Missing data:} For diagonal matrices $D_i$ with zeros or ones the reduced-noise model from \eqref{po_def} corresponds to missing data problems widely encountered in statistics \cite[e.g.,][]{schafer1997analysis,little2014statistical}. For random $D_i$ independent of other variables, we are under the assumption of missing completely at random (MCAR).

 \item
 {\bf Deconvolution and image restoration:} In image processing, an image might be corrupted by ``blurring''---convolution with a linear filter---followed by noise.  After taking the Fourier transform, this can be modeled as a coordinate-wise multiplication by a diagonal matrix $D_i$, followed by adding noise. This corresponds to the unreduced-noise model from \eqref{po_def2}.  
For example, the image formation model in cryo-electron microscopy (cryo-EM) under the linear, weak phase approximation leads to such a model \citep{frank1996three}. A closely related model was recently used by  \cite{bhamre2016denoising}, where $S_i$ are Fourier transforms of projection images of molecules, and $D_i$ are contrast transfer functions.

 \item
 {\bf Structural variability in cryo-EM:} In cryo-EM, $S_i$ is the three-dimensional structure of a molecule, and $D_i$ is a tomographic projection of this volume onto a randomly selected plane. Since the molecule typically has only a few degrees of freedom, such as different conformations or states, it is reasonable to model $S_i$ to lie on some unknown low-dimensional space \citep[e.g.,][]{katsevich2015covariance,anden2015covariance}. This corresponds to the unreduced-noise model from \eqref{po_def2}: the data $Y_i = D_i S_i+\ep_i$ are tomographic projections $D_i S_i$ with added noise $\ep_i$.

\item
 {\bf Signal acquisition and compressed sensing:} In some signal acquisition tasks such as hyperspectral imaging, due to resource constraints it is convenient to acquire reduced or compressed measurements of signals. The reductions are often taken to be random projections. It is of interest to reconstruct the PCs of the original measurements \citep[e.g.,][]{chang2003hyperspectral, fowler2009compressive}. This falls under the reduced-noise model from \eqref{po_def}.
 \end{enumerate}

 There may certainly be many other applications fitting this framework. In the above examples it is natural to posit that the distribution of the signal vectors $S_i$ is of low effective dimensionality. In this paper we will assume that the distribution lies on some unknown linear space of small dimension $r$, as in equation \eqref{eq:S_i}. Given observations of the form \eqref{po_def} or \eqref{po_def2}, we address several natural statistical questions:
 \begin{enumerate}
 \item {\bf Covariance estimation:}  How to estimate the covariance matrix of the signals $S_i$? This is both a fundamental statistical problem, and has numerous applications, including classification and denoising.

 \item {\bf PCA:} How should we estimate the principal components of $S_i$? This question is of special interest due to the importance of PCA for exploratory data analysis and visualization.

 \item {\bf Denoising:} How can we denoise---or predict---the individual signal vectors $S_i$? This is a central question both in the missing data problems, where it corresponds to imputation, as well as in the image processing problems, where it amounts to noise reduction.
 \end{enumerate}

 In this paper, we develop new methods for a special class of models, where the matrices $D_i \in \mathbb{R}^{p \times p}$ are \emph{diagonal}. We will call the observations $Y_i$ from \eqref{po_def} or \eqref{po_def2} \textit{diagonally reduced}. The missing data and deconvolution problems belong to this class. In the high-dimensional asymptotic regime where $p,n$ both grow to infinity and $p/n \rightarrow \gamma>0$, we develop methods that provide clear, quantitative answers to all questions posed above, under quite weak assumptions. 


Related work by \cite{katsevich2015covariance} and \cite{anden2015covariance} develops methods for covariance estimation when the $D_i$'s are projection matrices mapping a 3-D electron density to its integral on a randomly chosen plane. In this data acquisition model for cryo-EM, the authors propose consistent estimators of the covariance of the electron density. However, their observation models are different from our diagonally reduced models. In \cite{bhamre2016denoising}, the questions of covariance estimation and denoising are studied empirically when the $D_i$'s come from the contrast transfer function of a microscope and the $S_i$'s  are Fourier transforms of clean tomographic projections.

 \cite{nadakuditi2014optshrink} develops methods for low-rank matrix estimation with missing data. Our results are more general, and also include methods for \emph{covariance estimation} and  \emph{denoising individual datapoints}, see Sec.\ \ref{proba_rel_work} for more details. \cite{lounici2014high} develops eigenvalue soft thresholding methods for covariance estimation with missing data. In our somewhat more specialized models, we instead find the \emph{optimal eigenvalue shrinkers}, and they are different from soft thresholding. \cite{cai2016minimax} develop minimax rate-optimal covariance matrix estimators for missing data, focusing on bandable and sparse models. We instead focus on the spiked covariance model. 

 We next give a brief overview of our results.

 \subsection{Probabilistic Results}
 A lot of work in random matrix theory studies the asymptotic spectral theory of the spiked covariance model and its variants, see e.g., \cite{paul2014random,yao2015large}. In Sec.\ \ref{proba}, we introduce two general \emph{diagonally reduced} spiked covariance models corresponding to \eqref{po_def} and \eqref{po_def2}. We characterize the limiting eigenvalues of the data matrix $Y$ (with rows $Y_i^\top$), and the limiting angles of its singular vectors with the population singular vectors (of the matrix $S$ with rows $S_i^\top$).

More specifically, we show in Thm.\ \ref{spike_proj_multi} that the eigenvalue distribution of $n^{-1}Y^\top Y$ converges to a general Marchenko-Pastur distribution \citep{marchenko1967distribution}, while the top few eigenvalues have well-defined almost sure limits. This mirrors the behavior known in unreduced spiked models, \cite[e.g.,][]{baik2005phase, baik2006eigenvalues, benaych2012singular};  however, our assumptions in Sec.\ \ref{proba} are very general, and in fact lead to the \emph{most general results to date even in the unreduced case} when $D_i  = I_p$ (see Cor. \ref{standard_spike_gen} and Sec.\ \ref{proba_rel_work} for discussion). 

In the special case where the entries of $D_i$ are iid, the limiting spectrum and the angles between the population and empirical singular vectors are given by explicit formulas related to the standard Marchenko-Pastur law, as described in Cor.\ \ref{standard_spike_proj}. For general $D_i$, we can compute numerically the quantities specified by Thm.\ \ref{spike_proj_multi} with the {\sc Spectrode} method \citep{dobriban2015efficient}; see Sec.\ \ref{proj_simu}.  All computational results of this paper are reproducible with software publicly available at \url{github.com/dobriban/diagonally_reduced/}.


 \subsection{Covariance estimation}
 In Sec.\ \ref{sec:cov_est}, we apply the probabilistic results from Sec.\ \ref{proba} to develop methods for covariance estimation from diagonally reduced data, under the additional assumption that the diagonal entries of the reduction matrices $D_i$ are iid. The prototypical example is the missing data problem with uniform missingness, where in the reduced-noise model \eqref{po_def} the entries are Bernoulli($\delta$).

Already for unreduced data from the spiked covariance model, the sample covariance matrix is a poor estimator of the population covariance, since neither the empirical PCs nor the empirical spectrum converge to their population counterparts. Though little can be done about correcting the PCs, one can develop optimal shrinkage estimators of the eigenvalues. In the unreduced case, \cite{donoho2013optimal} consider estimators of the form $U \eta(\Lambda) U^\top$, where $U$ is the orthogonal matrix of PCs, $\Lambda$ is the matrix of eigenvalues of the sample covariance, $\eta:\mathbb{R}\to\mathbb{R}$ is a shrinkage function, and $\eta(\Lambda)$ replaces every diagonal element $\lambda$ of $\Lambda$ by $\eta(\lambda)$.

 In Sec.\ \ref{sec:cov_est}, we take a similar approach to the problem of covariance estimation in reduced models. We first define an unbiased estimator of the population covariance of the unreduced signals (Eq.\ \ref{sig_hat}). Building on the results of Sec.\ \ref{proba}, we describe the asymptotic spectral theory of this estimator, and finally derive shrinkers of the spectrum that are asymptotically optimal for certain loss functions. 
 We explicitly derive the optimal shrinkers in the case of operator norm loss (Eq.\ \eqref{eta_op} for reduced-noise and Eq.\ \eqref{eta_op_unre} for unreduced-noise) and Frobenius norm loss (Eq.\ \eqref{eta_fr} for reduced-noise and Eq.\ \eqref{eta_fr_unre} for unreduced-noise). We also derive the asymptotic errors of the optimal shrinkers (Eq.\ \eqref{op_loss} and \eqref{fr_loss}), and give a recipe for deriving the optimal shrinkers for a broad class of loss functions.

 \subsection{Denoising}
In Sec.\ \ref{sec:denoise}, we consider denoising, the task of predicting the signal vectors $S_i$ based on the observations $Y_i$. For this we study the Best Linear Predictor (BLP) well-known from random effects models \citep[e.g.,][]{searle2009variance}. The general form of a BLP is $\hat S_i = \Sigma_S (\Sigma_S + \Sigma_\ep)^{-1}Y_i$, where $\Sigma_\ep$ is the covariance matrix of the noise. This is also known as a ``linear Bayesian'' method \citep{hartigan1969linear}. In other areas such as electrical engineering and signal processing, it is known as the ``Wiener filter'', ``(linear) Minimum Mean Squared Estimator (MMSE)''  \citep[][Ch. 12]{kay1993fundamentals}, ``linear Wiener estimator'' \cite[][p. 538]{mallat2008wavelet}, or ``optimal linear filter'' \citep[][p. 550-551]{mackay2003information}.

The BLP is an ``oracle'' method, because it depends on unknown population parameters. In practice we can use the empirical BLP (EBLP), $\hat \Sigma_S (\hat \Sigma_S + \hat\Sigma_\ep)^{-1}Y_i$, where the unknown parameters are estimated using the data. Due to the inconsistency of PCA in high dimensions, this is suboptimal to the BLP. However, we can find the asymptotically optimal method for estimating the covariance matrix $\Sigma_S$ using the empirical PCs. This estimator holds several surprises. In particular, the optimal EBLP coefficients are different for \emph{in-sample} and \emph{out-of-sample} denoising---but the optimal mean squared error ends up identical! See Thms.\ \ref{den_thm}, \ref{oos-prop} for reduced-noise and Sec.\ \ref{den_po_def2} for unreduced-noise. It also turns out that the formula for in-sample EBLP, applied to all $Y_i$, is identical to optimal singular value shrinkage estimators (Sec.\ \ref{den_sv_shr}).

 This analysis involves characterizing random quantities with an intricate dependence structure, such $ Y_{i}^\top D_i \hat u_{k}$, where $\hat u_{k}$ is the $k$-th PC of the sample covariance matrix of $Y_i$. For this we extend significantly the technique introduced by \cite{benaych2012singular} to study the angles between $u_k$ and $\hat u_k$. We call this approach the \emph{outlier equation} method
 (see Sec.\ \ref{sec:denoise}).

 \section{Probabilistic results}
 \label{proba}
 \subsection{Main probabilistic results}
 This section presents a new result in random matrix theory, which will be the key tool for our work on covariance estimation. Recall that we have diagonally reduced observations $Y_i = D_i (S_i+\ep_i)$ or $Y_i = D_i S_i + \ep_i , \,\, i=1,\ldots,n$, where $S_i \in \mathbb{R}^p$ are unobserved signals and $D_i \in \mathbb{R}^{p \times p}$ are diagonal matrices. The signals have the form $S_i = \sum_{k=1}^r \ell_k^{1/2} z_{ik} u_k$.

 Here $u_k$ are deterministic signal directions with $\|u_k\|=1$. We will assume that $u_k$ are delocalized, so that $|u_{k}|_{\infty} \le C\log(p)^B/p^{1/2}$ for some constants $B,C>0$.  The scalars $z_{ik}$ are standardized independent random variables, specifying the variation in signal strength from sample to sample. For simplicity we assume that the deterministic spike strengths are different and sorted: $\ell_1>\ell_2>\ldots>\ell_r>0$. 
Finally $\ep_i = \Gamma^{1/2} \alpha_i$ is sampling noise, where $\Gamma= \diag(g_1^2,\ldots,g_p^2)$ is diagonal and deterministic, and $\alpha_i=(\alpha_{i1},\ldots,\alpha_{ip})^\top$ has independent standardized entries.

 \begin{table}[]
\centering
\caption{Definitions}
\label{def_tab}
{\renewcommand{\arraystretch}{1.4}
\begin{tabular}{|l|l|l|}
\hline
Name  & Definition & Defined in  \\ \hline
Reduced Observation   & $Y_i = D_i (S_i+\ep_i)$ or $Y_i = D_i S_i+\ep_i$ & \eqref{po_def},\, \eqref{po_def2}     \\ 
Signal & $S_i = \sum_{k=1}^r \ell_k^{1/2} z_{ik} u_k$ & \eqref{eq:S_i} \\ 
Reduction Matrix       & $D_i = \mu + \Sigma^{1/2} E_i$  & \eqref{p_def} \\ 
& $\mu  = \diag(\mu_1,\ldots,\mu_p)$, \, $\Sigma = \diag(\sigma^2_1,\ldots,\sigma^2_p) $ & \\
Noise & $\ep_i = \Gamma^{1/2} \alpha_i$, where $\Gamma= \diag(g_1^2,\ldots,g_p^2)$ &\\ 
Noise Variances & $H_p = p^{-1}\sum_{j=1}^p\delta_{g_j^2 (\mu_j^2 +\sigma_j^2)}$,\ $H_p \Rightarrow H$ &\\ 
& $G_p = p^{-1}\sum_{j=1}^p\delta_{g_j^2}$,\ $G_p \Rightarrow G$ &\\ 
General MP Law & $F_{\gamma,H}$, \, $\underline F_{\gamma,H}(x)  = \gamma F_{\gamma,H}(x)  + (1-\gamma)\delta_0$ &\\ 
Stieltjes Transform & $m_{\gamma,H}(z)  = \int \frac{d F_{\gamma,H}(x)}{x-z} $, \, $\underline m_{\gamma,H}(z)  = \int \frac{d \underline F_{\gamma,H}(x)}{x-z} $ &\\ 
 & $m_{H}(z)  = \int \frac{d H(x)}{x-z} $ &\\ 
D-Transform & $D_{\gamma,H}(x) = x\cdot m_{\gamma,H}(x) \cdot \underline m_{\gamma,H}(x)$&\\ 
Upper Edge & $b_H^2 = \sup \supp(F_{\gamma,H})$ & \\
\hline   
\end{tabular}
}
 \end{table}

 The diagonal matrices $D_i=\diag(D_{i1},\ldots,D_{ip})$ have the form
        \begin{align}
        \label{p_def}
        D_i = \mu + \Sigma^{1/2} E_i,
        \end{align}
where $E_i$ have independent standardized entries. Here $\mu  = \diag(\mu_1,\ldots,\mu_p) \in \mathbb{R}^{p\times p}$ is the deterministic diagonal mean and $\Sigma = \diag(\sigma^2_1,\ldots,\sigma^2_p) \in \mathbb{R}^{p\times p}$ is the deterministic diagonal covariance matrix of the entries of the reduction matrices. Let $H_p$ be the uniform distribution on the $p$ scalars $g_j^2 \cdot \E D_{ij}^2 = g_j^2 \cdot (\mu_j^2 +\sigma_j^2)$, $j=1,\ldots,p$, and let $G_p$ be the analogous object for $g_j^2$, $j=1,\ldots,p$. It turns out that these are the distributions of noise variances relevant for our models.  For dealing with model \eqref{po_def}, we assume that as $p\to \infty$, $H_p$ converges to a compactly supported limit distribution $H$: $H_p \Rightarrow H$. For dealing with model \eqref{po_def2}, we assume that $G_p$ converges to a compactly supported limit distribution $G$: $G_p \Rightarrow G$. 
 
 We will consider the high-dimensional regime where $n,p \to \infty$ such that $p/n\to \gamma>0$. In this setup, our answers will depend on the general Marchenko-Pastur distribution \citep{marchenko1967distribution}.  This law  describes the behavior of empirical eigenvalue distributions of sample covariance matrices: If $N$ is an $n\times p$ matrix with iid standardized entries, and $T$ is a $p\times p$ positive semidefinite matrix with eigenvalue distribution converging to $H$, then the eigenvalue distribution of the $p\times p$ matrix $n^{-1} T^{1/2} N^\top N T^{1/2}$ converges almost surely (a.s.) to the Marchenko-Pastur distribution $F_{\gamma,H}$ \citep[see e.g.,][for a reference]{bai2009spectral}. 

When $T = I_p$ is the identity, $F_{\gamma,H}$ is known as the \emph{standard Marchenko-Pastur distribution}, and has density (if $\gamma \in (0,1)$):
        \begin{align*}
        f_{\gamma}(x) = \frac{\sqrt{(g_{+} - x) (x - g_{-})}}
                           {2 \pi x}
                      I(x\in [g_{-},g_{+}])
        \end{align*}
 where $g_\pm = (1 \pm \sqrt{\gamma})^2$. For general $H$, $F_{\gamma,H}$ does not have a closed form, but it can be studied numerically \citep[see e.g.,][]{dobriban2015efficient}.

 Closely related to $F_{\gamma,H}$ is the distribution $\underline F_{\gamma,H}(x)  = \gamma F_{\gamma,H}(x)  + (1-\gamma)\delta_0$. This is the limit of the eigenvalue distribution of the $n\times n$ matrix $\smash{n^{-1}  N^\top T N}$. We will also need the Stieltjes transform $m_{\gamma,H}$ of $F_{\gamma,H}$, $m_{\gamma,H}(z)  = \int (x-z)^{-1} d F_{\gamma,H}(x)$, and the Stieltjes transform $\underline m_{\gamma,H}$ of $\underline F_{\gamma,H}$.  Based on these, one can define the D-transform of $F_{\gamma,H}$ by 
\[D_{\gamma,H}(x) = x\cdot m_{\gamma,H}(x) \cdot \underline m_{\gamma,H}(x).\] 
Up to the change of variables $x = y^2$, this agrees with the D-transform defined in \cite{benaych2012singular}.  Let $b_H^2$ be the supremum of the support of $F_{\gamma,H}$, and $D_{\gamma,H}(b_H^2) = \lim_{t\downarrow b} D_{\gamma,H}(t^2)$. It is easy to see that this limit is well defined, and is either finite or $+\infty$. Let us denote the support of a distribution $H$ on $\R$ by $\supp(H)$.

 Denote the normalized data matrix $\tilde Y = n^{-1/2}Y$, with the $n\times p$ matrix $Y$ having rows $Y_i^\top$.  Our main probability result, proved later in Sec.\ \ref{main_pf_proba}, is the following.

\begin{thms}[Diagonally reduced spiked models]
\label{spike_proj_multi}
Consider the observation models \eqref{po_def} and \eqref{po_def2}, under the above assumptions. Suppose that  
\benum 
\item $\E{\alpha_{ij}^4}<C$, $\smash{\E E_{ij}^4<C}$, and $\smash{\E |z_i|^{4+\phi}<C}$ for some $\phi>0$ and $C<\infty$. 
\item Under model \eqref{po_def}, $\sup \supp(H_p) \to \sup \supp(H)$. Under model \eqref{po_def2}, $\sup \supp(G_p) \to \sup \supp(G)$.
\item The squared norms $\|\mu  u_k\|^2$ converge to $\tau_k >0$. 
\item Under model \eqref{po_def}, $\mu u_k$ are \emph{generic} with respect to $M=\Gamma(\Sigma+\mu^2)$ in the sense that 
\[u_j^\top \mu (M -zI_p)^{-1} \mu u_k \to I(j=k) \cdot \tau_k \cdot m_H(z)\] 
for all $z\in \mathbb{C}^+$, where $m_H$ is the Stieltjes transform of $H$. 

Under model \eqref{po_def2}, $\mu u_k$ are generic with respect to $M=\Gamma$, i.e., $u_j^\top \mu (M -zI_p)^{-1} \mu u_k \to I(j=k) \cdot \tau_k \cdot m_G(z)$. 
\eenum

Then 
\benum
\item Under model \eqref{po_def}, the eigenvalue distribution of $\tilde Y^\top \tilde Y$ converges to the general Marchenko-Pastur law $F_{\gamma,H}$ a.s. In addition, the $k$-th largest singular value of $\tilde Y$ converges, $\sigma_k(\tilde Y) \to t_k>0$ a.s., where
\begin{equation}
\label{sv_eq}
t_k^2=
\left\{
	\begin{array}{ll}
		D_{\gamma,H}^{-1}(\frac{1}{\tau_k\ell_k}) & \mbox{\, if \, }  \ell_k>1/[\tau_k D_{\gamma,H}(b_H^2)], \\
		b_H^2 & \mbox{\, otherwise.}
	\end{array}
\right.
\end{equation}
Moreover,  let $\nu_j=\mu u_j/\|\mu u_j\|$ be the \emph{normalized reduced signals} and let $\hat u_k$ be the right singular vector of $\tilde Y$ corresponding to $\sigma_k(\tilde Y)$. Then $(\nu_j^\top \hat u_k)^2 \to c_{jk}^2$ a.s., where
\begin{equation*}
c_{jk}^2=
\left\{
	\begin{array}{ll}
		\frac{m_{\gamma,H}(t_k^2)}{D_{\gamma,H}'(t_k^2)\tau_k \ell_k} & \mbox{\, if \, } j=k \mbox{\, and \, }  \ell_k>1/[\tau_k D_{\gamma,H}(b_H^2)], \\
		0 &  \mbox{\, otherwise.}
	\end{array}
\right.
\end{equation*}
\item Under model \eqref{po_def2}, the analogous results hold with $G$ replacing $H$ everywhere.

\eenum
\end{thms}

Assumption 4 needs explanation. This assumption generalizes the existing conditions for spiked models. In particular, it is easy to see that it holds when the vectors $u_k$ are random with independent coordinates.  Specifically, let $x,y$ are two independent random vectors with iid zero-mean entries with variance $1/p$. Then $\E x^\top \mu (M - zI_p)^{-1} \mu x =p^{-1}\tr  \mu (M - zI_p)^{-1}\mu$. Assumption 4 requires that this converges to  $\tau \cdot m_H(z)$, which happens for instance when the vector $\mu$ itself has random independent coordinates with variance $\tau/p$, or when it equals a multiple of the identity. However, Assumption 4 is more general, as it does not require any kind of randomness in $u_k$. 

Thm \ref{spike_proj_multi} gives the limiting angles of the empirical eigenvectors $\hat u_k$ with the \emph{reduced} population eigenvectors $\nu_j=\mu u_j/\|\mu u_j\|$. These are in general different from the true eigenvectors $u_j$. However, in our main application (Cor.\ \ref{standard_spike_proj} and the following sections) they are the same, because $\mu$ is a multiple of identity.

One can gain some insight into the result in the simpler case where the noise is uncolored, so that $\Gamma = I_p$. In that case, before reduction we have a spike strength $\ell$ and an average noise level of unity. After reduction under model \eqref{po_def}, we have a spike strength $\ell \|\mu u \|^2$, and an average noise level $p^{-1}\sum_j \E D_{ij}^2 = p^{-1}(\|\mu\|^2 + \|\sigma\|^2)$, where $\sigma = (\sigma_1,\dots,\sigma_p)$. For a delocalized $u$, we expect $\|\mu u \|^2\approx p^{-1}\|\mu\|^2 $. Therefore, \emph{reduction in model \eqref{po_def} typically decreases the signal strength} by a factor of
$$\frac{\|\mu\|^2}{\|\mu\|^2 + \|\sigma\|^2}.$$
However, the spike strength after reduction---$\ell \|\mu u \|^2$---depends on the correlation between $\mu$ and $u$. In particular, the reduction can vary among the different  PCs. In contrast it is easy to see in a similar way that reduction in model \eqref{po_def2} may increase or decrease the signal strength.

The key strength of Thm.\ \ref{spike_proj_multi} is its generality. Specifically, there is essentially only one previous result on reduced spiked models, appearing in \cite{nadakuditi2014optshrink}. However, that only studies iid Bernoulli projections under restrictive conditions on the noise, whereas we allow for (1) a general diagonal covariance structure $\Sigma$ in the reduction matrices, as well as (2) a general diagonal noise structure $\Gamma$, and (3) more general moment conditions. Moreover, even in unreduced spiked models, our results are already the most general results to date (see Sec.\ \ref{proba_rel_work}).

We think that the generality is important for several reasons: first, for practical reasons it is good to have results that require as few assumptions as possible, especially unverifiable conditions like ``randomness'' in $u_k$ or ``orthogonal invariance'' of the noise. As a consequence of these general results, existing tools like singular value shrinkage are shown to apply more generally, so this is a direct improvement. Second, from a theoretical perspective it is good to understand the reason for the ``spiking'' behavior; our results clarify for instance that ``orthogonal invariance'' of the noise is not needed.

\subsubsection{Comments on the proof}

The broad outline of the proof is inspired by the argument presented in \cite{benaych2012singular} for unreduced spiked models. However, there are several new steps. First, the proof in \cite{benaych2012singular} concerns only the unreduced case, and the dependence introduced by the random reduction matrices $D_i$ is a new challenge. The observations in model \eqref{po_def} are $D_iX_i = D_iS_i + D_i\ep_i$, so the ``signal'' $D_iS_i$ and ``noise'' $D_i\ep_i$ are dependent. However, we show that the dependence is asymptotically negligible.  For non-diagonal reduction matrices $D_i$, the depencence may be asymptotically non-negligible; this explains why we currently need the diagonal assumption.

As a second novelty, the proof involves finding the limits of certain quadratic forms $u_k^\top R(z) u_j$, where $R(z)$ is a specific resolvent matrix with complex argument $z$. Since the $u_k$ are deterministic, the concentration arguments of \cite{benaych2012singular} are not available. Instead, we adapt the ``deterministic equivalents'' approach of \cite{bai2007asymptotics}. For this we need to take the imaginary part of the complex argument $z$ to zero, which appears to be a new argument in this context. 

\subsection{Reduced standard spiked models}
\label{red_stand_spiked}

We will later use the following corollary for the \emph{reduced standard spiked model} where the reduction coefficients and the noise entries are iid random variables. Suppose that the noise variances are equal to unity, so $\Gamma  = I_p$ and thus $\ep_i$ have independent standardized entries. Moreover assume that the reduction matrices $D_i$ have iid random diagonal entries $D_{ij}$ with mean $\mu = \E D_{ij}$ and variance $\sigma^2 = \V [D_{ij}]$. Note that previously $\mu$ was a matrix, but from now on it will be a scalar, and there will be no possibility for confusion. Let $m = \mu^2+\sigma^2$ and $\delta= \mu^2/m$. In this case it is easy to see that the reduced eigenvectors are the same as the unreduced ones, i.e., $\nu_k = u_k$. Moreover, Assumption 4 from Thm.\ \ref{spike_proj_multi} reduces to $u_k^\top u_l \to 0$ as $n\to \infty$, if $k\neq l$. 

Our answers can be expressed in terms of the well-known characteristics of the standard spiked model. The asymptotic location of the top singular values will depend on \emph{the spike forward map} \cite[e.g.,][]{baik2005phase}: 
\begin{equation*}
\lambda(\ell)=
\lambda(\ell;\gamma)=
\left\{
	\begin{array}{ll}
		(1+\ell) \left(1+ \frac{\gamma}{\ell}\right) & \mbox{\, if \, } \ell>\gamma^{1/2}, \\
		(1+\gamma^{1/2})^2 & \mbox{\, otherwise.}
	\end{array}
\right.
\end{equation*}
The asymptotic angle between singular vectors will depend on  \emph{the cosine forward map} $c(\ell;\gamma) \ge 0$ given by \cite[e.g.,][etc]{paul2007asymptotics, benaych2011eigenvalues}:
\begin{equation}
\label{cos_map}
c(\ell)^2=
c(\ell;\gamma)^2=
\left\{
	\begin{array}{ll}
		\frac{1-\gamma/\ell^2}{1+\gamma/\ell} & \mbox{\, if \, } \ell>\gamma^{1/2}, \\
		0 &  \mbox{\, otherwise.}
	\end{array}
\right.
\end{equation}

\begin{cor}[Reduced standard spiked models]
\label{standard_spike_proj}
Under observation model \eqref{po_def}, in the above setting, the eigenvalue distribution of $m^{-1}\tilde Y^\top \tilde Y$ converges to the standard Marchenko-Pastur law with aspect ratio $\gamma$, a.s.  Moreover, $m^{-1/2}\sigma_k(\tilde Y) \to t_k>0$ a.s., where
\begin{equation}
\label{spike_eq_white}
t_k^2 = \lambda(\delta \ell_k)
\end{equation}
Finally, let $\hat u_k$ be the right singular vector of $\tilde Y$ corresponding to $\sigma_k(\tilde Y)$. Then $(u_j^\top \hat u_k)^2 \to c_{jk}^2$ a.s., where
\begin{equation}
\label{cos_eq_white}
c_{jk}^2=
\left\{
	\begin{array}{ll}
		c^2(\delta \ell_k)  & \mbox{\, if \, } j=k \\
		0 &  \mbox{\, otherwise.}
	\end{array}
\right.
\end{equation}
Under observation model \eqref{po_def2} in the above setting, we have $\sigma_k(\tilde Y)^2 \to t_k^2 = \lambda(\mu^2\ell)$, while $(u_j^\top \hat u_k)^2 \to c_{jk}^2$, where $c_{jk}^2 = c^2(\mu^2 \ell_k)$ if $j=k$ and 0 otherwise.
\end{cor}
For the proof, see Sec.\ \ref{pf_standard_spike_proj}. While in the current paper we  only use this corollary of Thm.\ \ref{spike_proj_multi}, the proof in the special case is essentially as involved as in the general case. For this reason, and for potential future applications, we prefer to state Thm.\ \ref{spike_proj_multi} as well.

In this special case, reduction in model \eqref{po_def} lowers the spike strength from $\ell$ to $\delta\ell$. This result is related to Thm 2.4 of \cite{nadakuditi2014optshrink} on missing data, but we have the following advantages: (1) our result works under a 4-th order moment condition instead of requiring all bounded moments; (2) our result admits arbitrary diagonal reductions, not just missing data (see Sec.\ \ref{proba_rel_work} for details). 

Finally, when there is no reduction, i.e., when $D_i = I_p$ for all $i$, then it turns out we do not need the delocalization of $u_k$. Indeed that is only needed to show that the diagonal reductions introduce a negligible amount of dependence, but we do not need this when $D_i=I_p$. Hence, we can state the following corollary for unreduced spiked models:

\begin{cor}[Standard spiked models]
\label{standard_spike_gen}
Suppose we observe unreduced signals $Y_i = S_i + \ep_i$, and we do \emph{not} assume the delocalization of the PCs $u_k$. Suppose that the other assumptions of Cor. \ref{standard_spike_proj} hold: $\smash{S_i = \sum_{k=1}^r \ell_k^{1/2} z_{ik} u_k}$, where $u_k^\top u_j \to \delta_{kj}$, while $z_{ik}$, $\ep_{ij}$ are iid standardized with $\smash{\E |z_i|^{4+\phi}<C}$, $\E{\ep_{ij}^4}<C$ for some $\phi>0$ and $C<\infty$. 

Then the conclusions of Cor. \ref{standard_spike_proj} hold. Specifically, the spectrum of $\smash{n^{-1} Y^\top Y}$ converges to a standard MP law, its spikes converge a.s. to $\smash{\lambda(\ell_k)}$ and the squared cosines between population and sample eigenvectors converge a.s. to $\smash{c^2(\ell_k)}$.
\end{cor}

Again, the key strength of this corollary is its generality, specifically that it only has 4-th moment assumptions, not orthogonal invariance.

\subsection{Related work}
\label{proba_rel_work}

There is substantial earlier work on unreduced spiked models, and even in this case our result leads to an improvement. We refer to \cite{paul2014random,yao2015large} for general overviews of the area. 
The paper of \cite{benaych2012singular} is closely related to our approach, and we essentially follow their novel technique, relying on controlling certain bilinear forms. When the signal direction $u$ is fixed, their results require the distribution of the noise matrix to be bi-unitarily invariant, which essentially reduces to Gaussian distributions. Our model is more general since it  only requires a fourth-moment condition on the noise. 

 The technique introduced by \cite{benaych2011eigenvalues,benaych2012singular}  was adapted to non-white Gaussian signal-plus noise matrices $X_i = \ell^{1/2} z_i u + \Gamma^{1/2} \ep_i$ in \cite{chapon2012outliers}. They rely on an integration
by parts formula for functionals of Gaussian vectors and the Poincar\'e-Nash
inequality. A Poincar\'e inequality was also assumed in  \cite{capitaine2013additive}. Our result is stronger, since we only require fourth moment conditions.

 Previous extensions to the setting of missing data are found in \cite{nadakuditi2014optshrink}, which describes the OptShrink method for singular value shrinkage matrix denoising. The method is extended to data missing at random, and in particular the limit spectrum of the data matrix with zeroed-out missing data is found (his Thm.\ 2.4). This is related to Thm.\ \ref{spike_proj_multi}, but we have the following advantages: (1) our result works under the optimal 4-th order moment condition instead of requiring all moments to be bounded; (2) our result admits arbitrary reductions, not just binary projection matrices, (3) we extend to reduction matrices $D_i$ that have non-iid entries, (4) we allow heteroskedastic diagonal noise $\ep_i = \Gamma^{1/2} \alpha_i$; and (5) we also consider the unreduced-noise model from Eq.\ \eqref{po_def2} (whereas \cite{nadakuditi2014optshrink} considers the reduced-noise model from Eq.\ \eqref{po_def}).

\subsection{A numerical study}
\label{proj_simu}

\begin{figure}
\centering
\begin{subfigure}{.5\textwidth}
  \centering
  \includegraphics[scale=0.4]{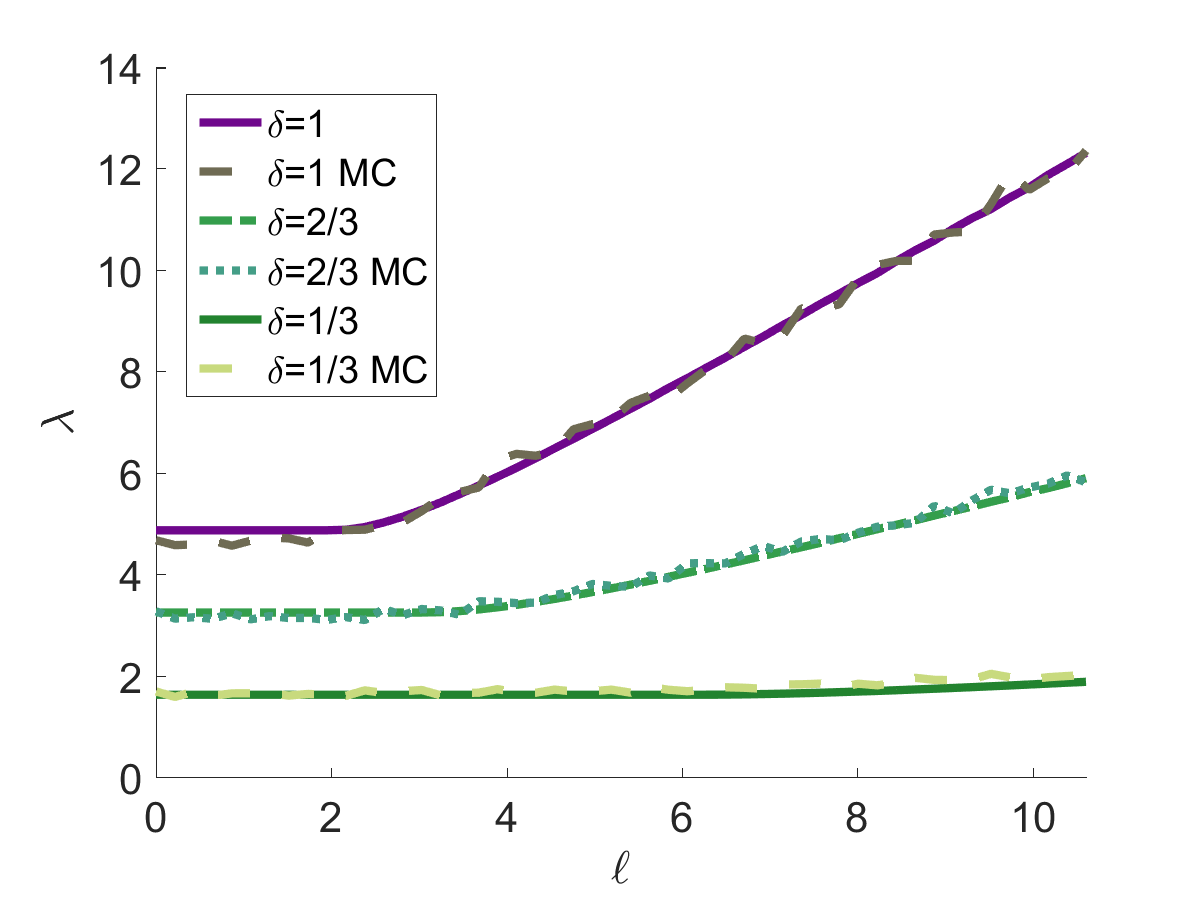}
\end{subfigure}%
\begin{subfigure}{.5\textwidth}
  \centering
  \includegraphics[scale=0.4]
{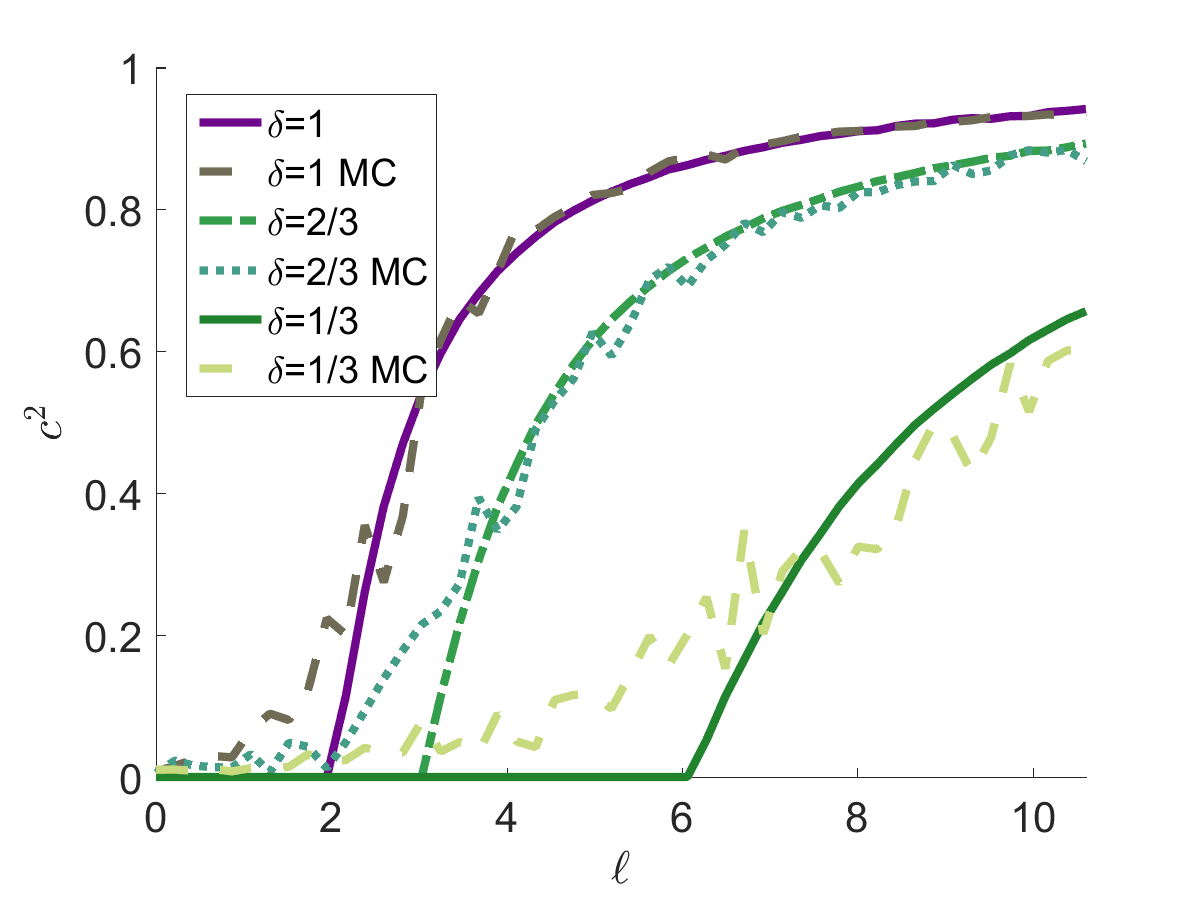}
\end{subfigure}
\caption{The effect of reduction on sample spikes and correlations between PCs. Formulas from Thm.\ \ref{spike_proj_multi} computed with {\sc Spectrode} \citep{dobriban2015efficient} overlaid with simulations. See Sec.\ \ref{proj_simu}. }
\label{spectrode_proj}
\end{figure}

We report the results of a numerical study to gain insight into our theoretical results.   
We consider model \eqref{po_def}, where $X_i = \ell^{1/2} z_i u + \ep_i$, where the noise is heteroskedastic, $\ep_i \sim \mathcal{N}(0,\Gamma)$, and $\Gamma$ is the diagonal matrix of eigenvalues of a $p\times p$ autoregressive (AR) covariance matrix of order 1, with entries $\Sigma_{ij} = \rho^{|i-j|}$. We set $Y_i = D_i X_i$, where $D_i$ have iid Bernoulli($\delta$) entries. We set the missingness parameter $\delta$ to the values $1/3,2/3$ and $1$. We choose the AR autocorrelation coefficient $\rho=0.5$, and vary the spikes $\ell$ from 0 to 3.5. 

We compute numerically the formulas in Thm.\ \ref{spike_proj_multi}, using the recent {\sc Spectrode} method \citep{dobriban2015efficient}, see Sec.\ \ref{spectrode} for the details.  We compare this with a Monte Carlo simulation with $n=200$, $\gamma=1/2$, $z_i$, $u$ generated as Gaussian random variables, and the results averaged over $10$ Monte Carlo trials. The results---displayed on Fig. \ref{spectrode_proj}---allow us to study the effect of reduction on spiked models. In particular, we observe the following phenomena: 
\bitem
\item The theory and simulations show good agreement. For eigenvalues, the results are very accurate. For the cosine, the results are more variable, and especially so for small $\delta$.

\item In the left plot of Fig. \ref{spectrode_proj}, we see that the empirical spike is an increasing function of the population spike $\ell$. Moreover, the location of the phase transition (PT) decreases with $\delta$, i.e., reduction degrades the critical signal strength.  
\item Similarly, in the right plot of Fig. \ref{spectrode_proj}, we see that the cosine between population and empirical PCs increases with the population spike $\ell$. For a given $\ell$, the cosine decreases as $\delta \to 0$. 
\eitem
It is not hard, but beyond our scope, to formalize the last two observations into theorems.

 \section{Covariance matrix estimation}
 \label{sec:cov_est}
In this section, we develop methods for covariance estimation in the reduced-noise model $Y_i = D_i X_i + D_i \varepsilon_i$ (Secs.\ \ref{sec:spec_redu}, \ref{opti_fun}) and in the unreduced-noise model $Y_i = D_i X_i + \varepsilon_i$ (Sec.\ \ref{opti_fun_unre}). We also discuss some related work in Sec.\ \ref{alt_est_cov}. Finally, we present numerical experiments illustrating the results in Sec.\ \ref{cov_numer}.
 
We restrict our attention to a special case of the diagonally reduced model we considered in Sec.\ \ref{proba}. We assume as in Sec.\ \ref{red_stand_spiked} that the entries of the reduction matrices $D_i$ are independently and identically distributed. We also suppose that the noise $\varepsilon_i$ is white, with variance 1 on each coordinate; that is,
$
        \text{Cov}(\varepsilon_i) = I_p.
$
 We will also require that the coordinates of $\varepsilon_i$ and the diagonal entries of $D_i$ both have finite eighth moments. Recall that in the setting of Cor.\ \ref{standard_spike_proj}, we have
$   \mu = \mathbb{E} D_{ij}
$
 and also that
$
 \sigma^2 = \text{Var}(D_{ij}).
$
The second moment is $m = \mathbb{E}D_{ij}^2 = \mu^2 + \sigma^2$, and 
$
        \delta = \mu^2/m.
$
For data missing uniformly at random, $m = \mu = \delta$ is the probability that each entry is observed. 

\subsection{The reduced-noise model}
\label{sec:redu_noise_mod}
 
 In the reduced-noise model, we observe $n$ samples of the random vector $Y = D (S + \varepsilon)$. It is then easy to see that we have the following formulas relating the covariance matrix of the signal $\Sigma_S$ and the covariance matrix of the observation $\Sigma_Y$:
        \begin{align}
        \label{eq:sig_Y}
        \Sigma_{Y} &= \mu^2 \Sigma_S 
                            + \sigma^2 \diag(\Sigma_S) 
                            +m I_p
        \\
        \nonumber
        \Sigma_S &= \frac{1}{\mu^2} \Sigma_{Y}
                    - \frac{\sigma^2}{m \mu^2} \diag(\Sigma_{Y})
                    -  I_p 
        \end{align}

These equations make it clear that the sample covariance matrix $\hat{\Sigma}_Y=n^{-1} \sum_{i=1}^n Y_i Y_i^\top$ is a biased estimator of the signal covariance matrix $\Sigma_S$.  Based on the second equation, we consider the following \emph{debiased} estimator of $\Sigma_S$:
        \begin{align}
        \label{sig_hat}
        \hat{\Sigma}_S = \frac{1}{\mu^2} \hat{\Sigma}_{Y}
                    - \frac{\sigma^2}{m \mu^2} \diag(\hat{\Sigma}_{Y}) 
                    - I_p.
        \end{align}
Here we assume for simplicity that $\mu,\sigma^2$ are known; but these scalar parameters are  straightforward to estimate from the observed $D_i$. In the special case of data missing completely at random, i.e., of iid sampling of entries with probability $\delta$, we have $\mu^2 = \delta^2$ and $m = \delta$, so this formula becomes
        \begin{align}
        \label{sig_hat_delta}
        \hat{\Sigma}_S &= \frac{1}{\delta^2} \hat{\Sigma}_{Y}
                    + \bigg(\frac{1}{\delta} - \frac{1}{\delta^2}\bigg)
                         \diag(\hat{\Sigma}_{Y}) 
                    - I_p.
        \end{align}
If our goal is to estimate $\Sigma_X = \Sigma_S+I_p$ instead of $\Sigma_S$, the corresponding unbiased estimator is
$       \hat{\Sigma}_X = \hat{\Sigma}_{Y}/\delta^2
                    + (\delta - 1)
                         \diag(\hat{\Sigma}_{Y})/\delta^2.$
 This recovers the unbiased estimator of $\Sigma_X$ proposed by \cite{lounici2014high}. That paper proposes to estimate $\Sigma_X$ by applying the soft-thresholding function 
$
        \eta_{\tau}(\lambda) = (\lambda - \tau)_+
$
 to the empirical eigenvalues $\lambda$ of the covariance $\hat{\Sigma}_X$. 
 \cite{lounici2014high} proves error bounds for this estimator in both operator and Frobenius norm losses, for covariance matrices $\Sigma_X$ of small effective rank $r_{eff}(\Sigma) = \tr(\Sigma)/\|\Sigma\|_{op}$. In contrast, we want to estimate the covariance matrix $\Sigma_S$ of the signal. For this different task, in the spiked covariance model, the function $\eta_\tau$ is not optimal, as we will show in Section \ref{opti_fun}.

In the next section, we employ the probabilistic results from Sec.\ \ref{proba} to determine the asymptotic spectral theory of $\hat{\Sigma}_Y$. In particular, we find asymptotic formulas for the eigenvalues, and the angles between its PCs and those of the population covariance $\Sigma_S$. Next, we show how to use these results in conjunction with the theory of \cite{donoho2013optimal} to derive optimal non-linearities $\eta$ of the spectrum of $\hat{\Sigma}_Y$ to estimate $\Sigma_S$ for a variety of loss functions.

 \subsection{The asymptotic spectral theory of $\hat{\Sigma}_S$ in reduced-noise}
 \label{sec:spec_redu}
 In this section, we will analyze the asymptotic spectral theory of the debiased estimator $\hat{\Sigma}_S$; that is, the limiting eigenvalue distribution, spikes, and limiting angles of its top eigenvectors with those of $\Sigma_S$.  We will rely on Corollary  \ref{standard_spike_proj} from Section \ref{proba} and an argument controlling the diagonal terms in the proof in Sec.\ \ref{diag_control_pf}.

 \begin{cor}
 \label{cor:s_hat}
 Let $1 \le k \le r$. Suppose that $\ell_k$ satisfies 
$    \ell_k > \sqrt{\gamma}/\delta.
$
 Then in the limit $p,n \rightarrow \infty$ and $p/n \rightarrow \gamma$, the $k^{th}$ largest eigenvalue of $\hat{\Sigma}_S$, $k=1,\dots,r$, converges almost surely to
        \begin{align}
        \label{lim_s_hat}
        \frac{1}{\delta} (\delta \ell_k + 1)
                \bigg( 1 + \frac{\gamma}{\delta \ell_k} \bigg) 
                - \frac{1}{\delta}.
        \end{align}
 The distribution of the bottom $p-r$ eigenvalues of $\hat{\Sigma}_S$  converges to a shifted and scaled Marchenko-Pastur distribution $(\mu_{MP}-1)/\delta$ supported on the interval
$[(1-\sqrt{\gamma})^2 - 1, (1+\sqrt{\gamma})^2- 1]/\delta$. 

 If $\hat{u}_k^\prime$ is the $k^{th}$ eigenvector of $\hat{\Sigma}_{Y}$ and $\hat{u}_k$ is the $k^{th}$ eigenvector of $\hat{\Sigma}_S$, then almost surely we have
$ \lim_{n \rightarrow \infty} \langle \hat{u}_k^\prime,\hat{u}_k \rangle^2 = 1$
 and 
\[
        \lim_{n \rightarrow \infty} \langle \hat{u}_k,u_k \rangle^2
        = \frac{1 - \gamma/(\delta \ell_k)^2}{1 + \gamma/(\delta \ell_k)}.
\]
If $\ell_k \le \sqrt{\gamma}/\delta$, then the top eigenvalue converges to the upper edge of the shifted MP distribution, and the cosine converges to 0.
 \end{cor}

 The shifted Marchenko-Pastur distribution arises as the limiting empirical spectral distribution of the eigenvalues corresponding to noise. This is also the case for the available case sample covariance of pure noise \citep{jurczak2015spectral}. We will discuss the available-case estimator in Secs.\ \ref{alt_est_cov} and \ref{cov_numer}.

 \subsection{Optimal shrinkage of the spectrum of $\hat{\Sigma}_S$ in reduced-noise}
 \label{opti_fun}
 \begin{table}[]
\centering
\caption{Optimal covariance shrinkage in unreduced-noise model. $\tilde{\ell}$ is the function defined by equation \eqref{tilde_ell}, $\delta = \mu^2 / m$, and $s_k = \sqrt{1 - c_k^2}$ is the asymptotic sine of the angle between the empirical and population PCs.}

{\renewcommand{\arraystretch}{1.8}
\begin{tabular}{|l|l|l|l|l|l|}
\hline
 Loss function & Eigenvalue & Asymptotic loss  & References \\ \hline
 Operator & $ \frac{ \tilde{\ell}(\delta\lambda + 1)}{\delta} $   & $\ell_1 s_1$ & 
                                              \eqref{eta_op}, \eqref{op_loss} \\ \hline
 Squared Frobenius    & $\frac{\tilde{\ell}(\delta\lambda + 1)
                          \cdot c^2(\tilde{\ell}(\delta\lambda + 1))}{\delta}$   
                                             & $\sum_{k=1}^r (1-c_k^4)\ell_k^2$ 
                                             & \eqref{eta_fr}, \eqref{fr_loss} \\ \hline
\end{tabular}
}
\end{table}

 Having characterized the asymptotic spectrum of the debiased estimator $\hat{\Sigma}_S$, we can apply the technique of \cite{donoho2013optimal} to derive optimal shrinkers of the eigenvalues of $\hat{\Sigma}_S$ to minimize various loss functions. Any of the 26 loss functions found in \cite{donoho2013optimal} can be adapted to the setting of diagonally reduced data. In Sec.\ \ref{proof_eta_2d}, we carefully check the details of this program.
%

 Write the eigendecomposition of $\Sigma_S$ as $\Sigma_S = U \Lambda U^\top$ and the eigendecomposition of the debiased estimator $\hat{\Sigma}_S$ as $\hat{\Sigma}_S = \hat{U} \hat{\Lambda} \hat{U}$. For a given function $\eta : \mathbb{R} \rightarrow [0,\infty)$, define the matrix $\hat{\Sigma}^\eta_S$ by
        \begin{align*}
        \hat{\Sigma}^\eta_S = \hat{U} \eta(\hat{\Lambda}) \hat{U}^\top
        \end{align*}
 where $\eta(\hat\Lambda)$ is the diagonal matrix that replaces the $k^{th}$ diagonal element $\hat \lambda_k$ of $\hat{\Lambda}$ with $\eta(\hat \lambda_k)$.

 For any value of $p$, let $L_p(A,B)$ denote a loss function between two $p$-by-$p$ symmetric matrices $A$ and $B$. We consider loss functions with two key properties: first, they must be \textit{orthogonally invariant}; that is, $L_p(A,B) = L_p(UAV,UBV)$ for any orthogonal matrices $U$ and $V$. Second, they must \textit{decompose over blocks}. This means that if $A_1,B_1 \in \mathbb{R}^{p_1 \times p_1}$ and $A_2,B_2 \in \mathbb{R}^{p_2 \times p_2}$ where $p = p_1 + p_2$, then either $L_p(A_1 \oplus A_2 , B_1 \oplus B_2) = \max\{L_{p_1}(A_1 , B_1), L_{p_2}(A_2 , B_2) \}$, in which case we say $L_p$ is \textit{max-decomposable}; or $L_p(A_1 \oplus A_2 , B_1 \oplus B_2) = L_{p_1}(A_1 , B_1) + L_{p_2}(A_2 , B_2)$, in which case we say $L_p$ is \textit{sum-decomposable}. Operator norm loss $L_p(A,B) = \|A-B\|_{op}$ is max-decomposable, whereas squared Frobenius norm loss $L_p(A,B) = \|A-B\|_F^2$ is sum-decomposable.

 Our goal is to find the function $\eta$ that minimizes the asymptotic loss over certain classes; that is, we seek:
        \begin{align*}
        \eta^* = \operatorname*{arg\,min}_{\eta} L_{\infty}(\Sigma_S,\hat{\Sigma}^\eta_S)
        \end{align*}
 where $L_{\infty}(\Sigma_S,\hat{\Sigma}^\eta_S)$ is the almost sure limit of $L_p(\Sigma_S,\hat{\Sigma}^\eta_S)$ as $n,p \to \infty$ and $p/n \to \gamma$. We will show from first principles that this limit is well defined. As in \cite{donoho2013optimal}, we will consider only those functions $\eta$ that \emph{collapse the vicinity of the bulk} to 0; that is, for which there is an $\varepsilon > 0$ such that $\eta(\lambda) = 0$ whenever $\lambda \le (1+\sqrt{\gamma})^2/\delta - 1/\delta +\ep$ (this is the value of the upper bulk edge, as given in Cor.\ \ref{cor:s_hat}).

 In Sec.\ \ref{proof_eta_2d}, we will show:
        \begin{align}
        \label{loss_infty}
        L_\infty(\Sigma_S,\hat{\Sigma}^\eta_S)
        = L_{2r}\bigg( \bigoplus_{k=1}^r A_2(\ell_k),
                       \bigoplus_{k=1}^r B_2(\eta(\lambda_{k}),c_{k},s_{k})
                \bigg)
        \end{align}
 where 
        \begin{math}
        A_2(\ell) =         
        \left(
        \begin{array}{c c}
         \ell & 0     \\
         0    & 0        
        \end{array}
        \right),
        \end{math}
 and $c_k = c(\delta \ell_k)$, $s_k = s(\delta \ell_k)$, where $s(\ell) = \sqrt{1 - c^2(\ell)}\ge 0$ the asymptotic sine of the angle between the empirical and population PCs, 
        \begin{align*}
        B_2(\eta(\lambda),c,s) =
        \left(
        \begin{array}{c c}
         \eta(\lambda) c^2(\delta \ell)       & \eta(\lambda)  c(\delta \ell) s(\delta \ell)     \\
         \eta(\lambda) c(\delta \ell) s(\delta \ell) & \eta(\lambda) s^2(\delta \ell)        
        \end{array}
        \right).
        \end{align*}

 Since the loss function is either max-decomposable or sum-decomposable, for $\eta$ to minimize the right side, it is sufficient that it minimize every individual term $L_{2}( A_2(\ell_k), B_2(\eta(\lambda_{k}),c_{k},s_{k}))$. That is, the asymptotically optimal $\eta$ minimizes the two-dimensional loss:
        \begin{align}
        \label{eta_2d}
        \eta^* = \operatorname*{argmin}_{\eta} L_2 (A_2(\ell), B_2(\eta(\lambda),c,s)).
        \end{align}

 This dramatically simplifies the problem, as this minimization can often be done explicitly. Deriving the optimal $\eta$ now depends on the particular choice of loss function. We consider two representative cases where a simple closed formula is easily found: operator norm loss, and squared Frobenius norm loss. The same recipe of explicitly solving the problem \eqref{eta_2d} can be used for any orthogonally-invariant and max- or sum-decomposable loss function, including those found in \cite{donoho2013optimal}.

 \subsubsection{Operator norm loss/max-decomposable losses}
 Since operator norm loss $L_p(A,B) = \| A - B\|_{op}$ is max-decomposable, equation \eqref{loss_infty} implies that
        \begin{align*}
        L_\infty(\Sigma_S,\hat{\Sigma}^\eta_S)
            = \max_{1\le k \le r} L_2(A(\ell_k),B(\eta(\lambda_k),c_k,s_k)).
        \end{align*}
 Consequently, the asymptotically optimal $\eta$ is the one that minimizes the two-dimensional loss function $L_2(A(\ell_k),B(\eta(\lambda_k),c_k,s_k)) = \| A(\ell_k) - B(\eta(\lambda_k),c_k,s_k) \|_{op}$. Repeating the derivation in \cite{donoho2013optimal}, the optimal $\eta(\lambda_k)$ sends $\lambda_k$ back to its population value, $\ell_k$. From formula~\eqref{lim_s_hat} in Cor.~\ref{cor:s_hat}, 
        \begin{align}
        \label{eta_op}
        \eta^*(\lambda) = \frac{ \tilde{\ell}(\delta\lambda + 1)}{\delta}
        \end{align}
 where $\tilde{\ell}$ inverts the spike forward map $\ell \mapsto \lambda(\ell;\gamma)$ defined in Sec.\ \ref{red_stand_spiked},
        \begin{align}
        \label{tilde_ell}
        \tilde{\ell}(y) = \frac{y - 1 - \gamma 
                          + \sqrt{(y - 1 - \gamma)^2 - 4\gamma}}{2}.
        \end{align}
 Direct computation shows that $L_2(A(\ell_k),B(\eta^*(\lambda_{k}),c_{k},s_{k})) = \ell_k s_k$; consequently, the asymptotic loss is given by the formula:
        \begin{align}
        \label{op_loss}
        L_\infty(\Sigma_S,\hat{\Sigma}^{\eta^*}_s) = \ell_1 s_1.
        \end{align}

 \subsubsection{Frobenius norm loss/sum-decomposable losses}
 Since the squared Frobenius loss $L_p(A,B) = \| A - B\|_F^2$ is sum-decomposable, equation \eqref{loss_infty} implies that
        \begin{align*}
        L_\infty(\Sigma_S,\hat{\Sigma}^\eta_S)
            = \sum_{k=1}^r L_2(A(\ell_k),B(\eta(\lambda_k),c_k,s_k)).
        \end{align*}
 Consequently, the asymptotically optimal $\eta$ is the one that minimizes the two-dimensional loss function $L_2(A(\ell_k),B(\eta(\lambda_k),c_k,s_k)) = \| A(\ell_k) - B(\eta(\lambda_k),c_k,s_k) \|_F^2$. As derived in \cite{donoho2013optimal}, the value of $\eta(\lambda_k)$ that minimizes this is $\ell_k c_k^2$. We have already seen that $\ell_k = \tilde{\ell}(\delta \lambda_k + 1) / \delta$, where $\tilde{\ell}$ is the function defined by \eqref{tilde_ell}. Consequently, with $c^2(\ell) = c^2(\ell;\gamma)$ being the cosine forward map, the formula for $\eta^*(\lambda)$ is
        \begin{align}
        \label{eta_fr}
        \eta^*(\lambda) = \frac{\tilde{\ell}(\delta\lambda + 1)
                          \cdot c^2(\tilde{\ell}(\delta\lambda + 1))}{\delta}.
        \end{align}

 A straightforward computation shows that $L_2(A(\ell_k),B(\ell_k c_k^2,c_k,s_k)) = (1-c_k^4)\ell_k^2$, and consequently, the asymptotic loss is given by the formula:
        \begin{align}
        \label{fr_loss}
        L_\infty(\Sigma_S,\hat{\Sigma}^{\eta^*}_s) = \sum_{k=1}^r (1-c_k^4)\ell_k^2.
        \end{align}

 \subsection{The unreduced-noise model}
 \label{opti_fun_unre}
 \begin{table}[]
\centering
\caption{Optimal covariance shrinkage in unreduced-noise model. $\tilde{\ell}$ is the function defined by equation \eqref{tilde_ell}, and $s_k = \sqrt{1 - c_k^2}$ is the asymptotic sine of the angle between the empirical and population PCs.}

{\renewcommand{\arraystretch}{1.8}
\begin{tabular}{|l|l|l|l|l|l|}
\hline
 Loss function & Eigenvalue & Asymptotic loss  & References \\ \hline
 Operator & $ \frac{\tilde{\ell}( \mu^2 \lambda + 1 )}{\mu^2} $   & $\ell_1 s_1$ 
                                          & \eqref{eta_op_unre}, \eqref{op_loss} \\ \hline
 Squared Frobenius    & $\frac{\tilde{\ell}( \mu^2 \lambda + 1 )
                          \cdot c^2(\tilde{\ell}(\mu^2 \lambda + 1))}{\mu^2}$   
                                        & $\sum_{k=1}^r (1-c_k^4)\ell_k^2$ 
                                        & \eqref{eta_fr_unre}, \eqref{fr_loss} \\ \hline
\end{tabular}
}
\end{table}

 In the unreduced-noise model  $Y_i = D_i s_i + \varepsilon_i$ we can develop similar methods for covariance estimation. The formulas relating $\Sigma_S$ to $\Sigma_Y$ are
        \begin{align}
        \label{eq:sig_Y_unre}
        \Sigma_{Y} &= \mu^2 \Sigma_S 
                            + \sigma^2 \diag(\Sigma_S) 
                            +  I_p
        \\
        \nonumber
        \Sigma_S &= \frac{1}{\mu^2} \Sigma_Y 
                   - \frac{\sigma^2}{m \mu^2} \text{diag}(\Sigma_Y)
                   - \frac{1}{m} I_p.
        \end{align}

 Consequently, the analogous \emph{debiased} covariance estimator is:
        \begin{align}
        \label{sig_hat_unre}
        \hat{\Sigma}_S = \frac{1}{\mu^2} \hat{\Sigma}_Y 
                   - \frac{\sigma^2}{m \mu^2} \text{diag}(\hat{\Sigma}_Y)
                   - \frac{1}{m} I_p.
        \end{align}

 In this section, we will derive optimal shrinkers of the spectrum of $\hat{\Sigma}_S$ in the unreduced-noise model using the same technique as for the reduced-noise model in Sec.\ \ref{opti_fun}. Our analysis rests on the following result, proved in Sec. \ref{diag_control_pf}:

 \begin{cor}
 \label{cor:s_hat_unred}
 Let $1 \le k \le r$. Suppose that $\ell_k$ satisfies $\ell_k > \sqrt{\gamma}/\mu^2$.
 Then in the limit $p,n \rightarrow \infty$ and $p/n \rightarrow \gamma$, the $k^{th}$ largest eigenvalue of $\hat{\Sigma}_S$ converges almost surely to
        \begin{align}
        \label{lim_s_hat2}
        \frac{1}{\mu^2} (\mu^2 \ell_k + 1)
                \bigg( 1 + \frac{\gamma}{\mu^2 \ell_k} \bigg) 
                - \frac{1}{\mu^2}.
        \end{align}
 The distribution of the bottom $p-r$ eigenvalues of $\hat{\Sigma}_S$  converges to a shifted Marchenko-Pastur distribution supported on the interval $[(1-\sqrt{\gamma})^2 - 1, (1+\sqrt{\gamma})^2- 1]/\mu^2$. 
 
 If $\hat{u}_k^\prime$ is the $k^{th}$ eigenvector of $\hat{\Sigma}_{Y}$ and $\hat{u}_k$ is the $k^{th}$ eigenvector of $\hat{\Sigma}_S$, then almost surely we have
$        \lim_{n \rightarrow \infty} \langle \hat{u}_k^\prime,\hat{u}_k \rangle^2 = 1
$ and 
        \begin{align*}
        \lim_{n \rightarrow \infty} \langle \hat{u}_k,u_k \rangle^2
        = \frac{1 - \gamma/(\mu^2 \ell_k)^2}{1 + \gamma/(\mu^2 \ell_k)}.
        \end{align*}
If $\ell_k \le \sqrt{\gamma}/\mu^2$, then the top eigenvalue converges to the upper edge of the shifted MP distribution, and the cosine converges to 0.
 \end{cor}

 Given an empirical eigenvalue $\hat \lambda_k$ of $\hat{\Sigma}_S$, we estimate the population eigenvalue $\ell_k$ by
$  \tilde{\ell}( \mu^2 \hat \lambda_k + 1 )/\mu^2
$
 where $\tilde{\ell}$ is the function given by formula~\eqref{tilde_ell}. This estimator converges almost surely to the true value $\ell_k$ if $\ell_k$ exceeds the threshold $\sqrt{\gamma} / \mu^2$. This also gives us an estimator of the squared cosine, by the formula
$     c(\tilde{\ell}(\mu^2 \hat \lambda_k + 1))$.
 We can now derive the optimal non-linear functions on the spectrum. For operator norm loss, we have
        \begin{align}
        \label{eta_op_unre}
        \eta^*(\lambda) = \frac{\tilde{\ell}( \mu^2 \lambda + 1 )}{\mu^2},
        \end{align}
 which incurs an asymptotic loss of $L_\infty(\Sigma_S,\hat{\Sigma}^{\eta^*}_s) = \ell_1s_1$. For squared Frobenius norm loss, the optimal non-linearity is
        \begin{align}
        \label{eta_fr_unre}
        \eta^*(\lambda) = \frac{\tilde{\ell}( \mu^2 \lambda + 1 )
                          \cdot c(\tilde{\ell}(\mu^2 \lambda + 1))}{\mu^2} 
        \end{align}
 and the asymptotic loss is $L_\infty(\Sigma_S,\hat{\Sigma}^{\eta^*}_s) = \sum_{k=1}^r (1-c_k^4)\ell_k^2$.

 \subsection{Alternative linear systems for estimating $\Sigma_S$}
 \label{alt_est_cov}
 The optimal shrinkers derived in Secs.\ \ref{opti_fun} and \ref{opti_fun_unre} for estimating the covariance $\Sigma_S$ from the reduced-noise observations start from the debiased estimators \eqref{sig_hat} for reduced-noise and \eqref{sig_hat_unre} for unreduced noise. Another way of viewing these estimators is as the solution to a linear system: for reduced-noise, this system is given by equation \eqref{eq:sig_Y},  and for unreduced-noise by equation \eqref{eq:sig_Y_unre}.

 Of course, there are other linear systems yielding unbiased estimators whose spectrum we could shrink. The papers \cite{katsevich2015covariance, anden2015covariance,bhamre2016denoising} consider such an estimator, which we will briefly discuss here.

 By definition, for any $k = 1,\dots,n$, $\Sigma_S = \mathbb{E}_{S,\varepsilon}[S_k S_k^\top]$, where $\mathbb{E}_{S,\varepsilon}$ denotes the expectation with respect to the random $S_k$ and $\varepsilon_k$, but not the $D_k$. We can therefore write in the unreduced-noise model $Y_k = D_k S_k+\ep_k$:
        \begin{align}
        \label{eqs_for_Sigma}
        D_k \Sigma D_k = \mathbb{E}_{S,\varepsilon} [(D_k S_k) (D_k S_k)^\top]
              = \mathbb{E}_{S,\varepsilon} [(Y_k - \varepsilon_k) (Y_k - \varepsilon_k)^\top] 
              = \mathbb{E}_{S,\varepsilon} [Y_k Y_k^\top] + I_p.
        \end{align}
 If we knew the values of $\mathbb{E}_{S,\varepsilon} [Y_k Y_k^\top]$ for every $k$, we could derive an unbiased estimator of $\Sigma_S$ by solving the $n$ equations given by \eqref{eqs_for_Sigma}. The papers \cite{katsevich2015covariance, anden2015covariance, bhamre2016denoising} instead substitute the observed value $Y_k Y_k^\top$ for its expected value, and derive an unbiased estimator of $\Sigma_S$ by the minimization problem
        \begin{align*}
        \hat{\Sigma}_S 
            = \operatorname*{arg\,min}_{\Sigma} 
              \frac{1}{n} \sum_{k=1}^n \| D_k\Sigma D_k^\top - Y_k Y_k^\top - I_p\|_F^2.
        \end{align*}
 Differentiating in $\Sigma$, we see that $\hat{\Sigma}_S$ must satisfy the linear system:
        \begin{align*}
        \frac{1}{n} \sum_{k=1}^n D_k^\top D_k \Sigma D_k^\top D_k
            = \frac{1}{n} \sum_{k=1}^n (D_k^\top Y_k) (D_k^\top Y_k)^\top 
              - \frac{1}{n} \sum_{k=1}^n D_k^\top D_k.
        \end{align*}
%

 We now consider another linear system, defined by averaging the $n$ equations \eqref{eqs_for_Sigma}:
        \begin{align*}
        \frac{1}{n}\sum_{k=1}^n D_k \Sigma_S D_k 
            = \frac{1}{n}\sum_{k=1}^n\mathbb{E}_{S,\varepsilon} [Y_k Y_k^\top] - I_p.
        \end{align*}
 Note that the matrices $D_k$ are fixed in this equation. By replacing $\mathbb{E}_{S,\varepsilon} [Y_k Y_k^\top]$ with the estimate $Y_k Y_k^\top$, we define a new estimator of $\Sigma_S$ as the solution to the equation
        \begin{align}
        \label{sys2_unre}
        \frac{1}{n}\sum_{k=1}^n D_k \Sigma D_k 
            = \frac{1}{n}\sum_{k=1}^n Y_k Y_k^\top - I_p.
        \end{align}

 In the case of reduced-noise, we can repeat the same derivation and arrive at an unbiased estimator that satisfies the system
        \begin{align}
        \label{sys2_redu}
        \frac{1}{n}\sum_{k=1}^n D_k \Sigma D_k 
            = \frac{1}{n}\sum_{k=1}^n Y_k Y_k^\top - \frac{1}{n}\sum_{k=1}^n D_k^2.
        \end{align}
%


 We will denote the estimator solving \eqref{sys2_unre} (in the unreduced-noise case) and \eqref{sys2_redu} (in the reduced noise case) by $\hat{\Sigma}_{s}^\prime$. Taking the expectation of each side of \eqref{sys2_redu}, we arrive at the linear system \eqref{eq:sig_Y_unre}, which defines our estimator $\hat{\Sigma}_S$ in the unreduced-noise model. Similarly, taking the expectation of each side of \eqref{sys2_unre}, we arrive at the linear system \eqref{eq:sig_Y}, which defines our estimator $\hat{\Sigma}_S$ in the reduced-noise model. Consequently, we expect that, in the limit $n,p \to \infty$, $\hat{\Sigma}_S$ and $\hat{\Sigma}_S^\prime$ will be close. In fact, we can show that the relative error of the estimators converges to 0, as stated in the folowing proposition (proved in Sec.\ \ref{rel_diff_prop_pf}):

 \begin{prop}
 \label{rel_diff_prop}
 In both the reduced-noise and unreduced-noise models, the relative difference $\|\hat{\Sigma}_S - \hat{\Sigma}_S^\prime\|_{F} / \| \hat{\Sigma}_S\|_F \to 0$ almost surely as $n,p\to\infty$ and $p/n\to\gamma$.
 \end{prop}

%

%
%
%

 \subsection{Numerical experiments}
 \label{cov_numer}
 We perform experiments with missing data, where the diagonal entries of the reduction matrices $D_i$ are independent Bernoulli($\delta$) random variables. In this case, $\mu = m = \delta$, and $\sigma^2 = \delta(1 - \delta)$. The unbiased estimator to which we apply shrinkage is given by the formula~\eqref{sig_hat_delta}. In all experiments, both the signal and the noise are drawn from Gaussian distributions.

\begin{figure}[p]
\centering
\begin{tabular}{c}
\includegraphics[scale=.55]{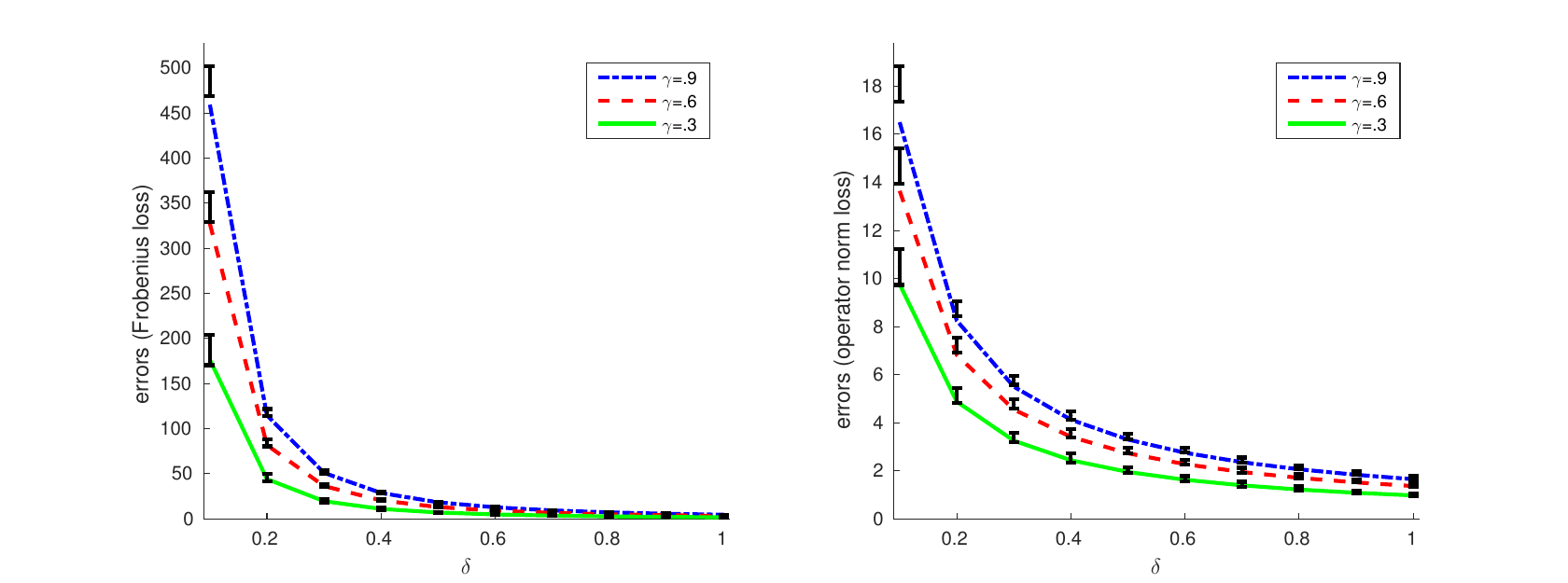}\\
\includegraphics[scale=.55]{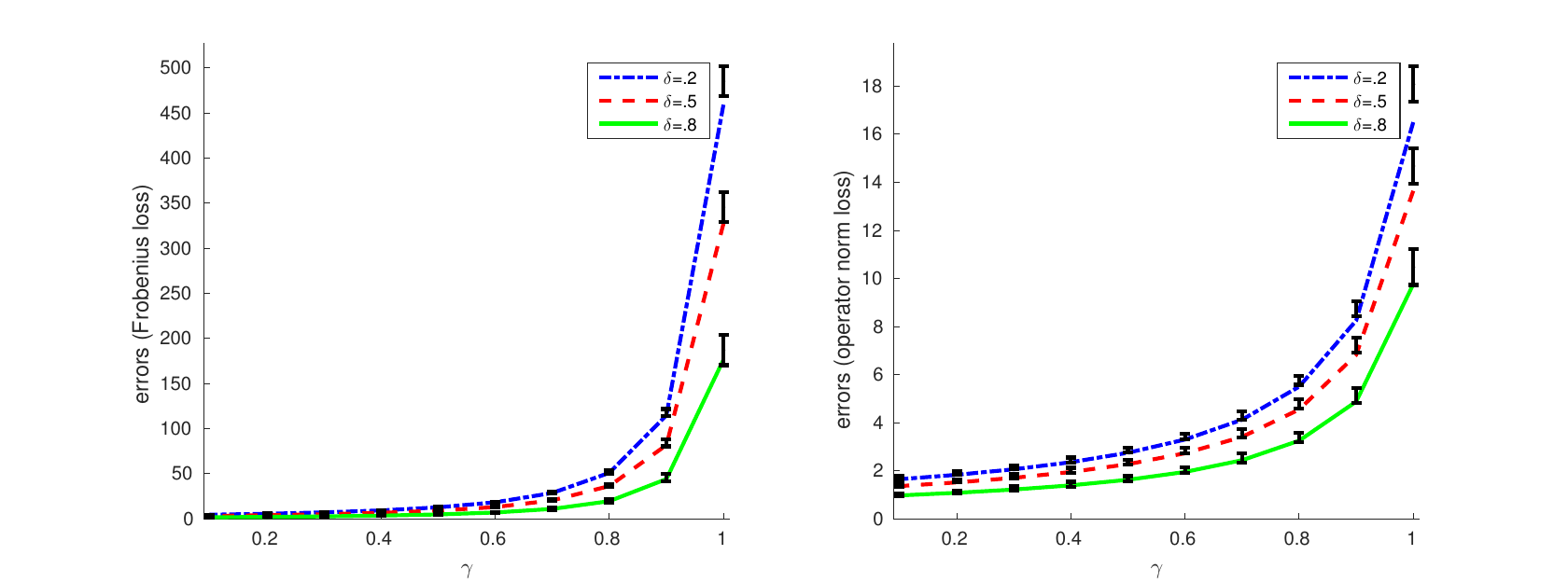}\\
\end{tabular}
\caption{\emph{Top row}: Estimation error of the covariance as a function of $\delta$. Left: Frobenius shrinker. Right: operator norm shrinker. For each $\delta$, 200 Monte Carlo tests were averaged. The error bars cover two standard deviations; the lines go through the predicted errors. \emph{Bottom row}: Same plot as a function of $\gamma$. 
}
\label{fig:MC}
\end{figure}

 \subsubsection{The errors in estimating the covariance}
 In the first experiment, we illustrate the dependence of the asymptotic errors on the parameters $\delta$ and $\gamma$. The clean signal vectors $S_i$ are drawn from a rank 1 Gaussian. The noise is also Gaussian noise, of unit variance; the ambient dimension $p$ is fixed at $p=1200$ in this experiment, while the number of samples $n$ varies with $\gamma$. 
 The top two rows of Fig.\ \ref{fig:MC} shows the errors in estimation for Frobenius loss and operator norm loss, as functions of the parameters $\gamma$ and $\delta$. The error bars cover the empirical mean error, plus/minus two standard deviations over 200 runs of the experiment.

 Several phenomena are apparent in these plots. First, the empirical mean of the errors is well-approximated by the asymptotic error formulas~\eqref{op_loss} and \eqref{fr_loss}, especially as the number of samples grows (corresponding to smaller $\gamma$, as $n = p/\gamma = 1200/\gamma$). Second, the errors decay as $\delta$ approaches 1, which is expected as larger $\delta$ increases the effective signal strength. Third, the errors grow as $\gamma$ approaches 1; this is also expected, since large $\gamma$ leads to higher dimensional problems.

%

%
%
%

 \section{Denoising}
 \label{sec:denoise}
 \subsection{Setup}
  \label{sec:denoise_setup}
 It is often of interest to denoise the observations $Y_i$ and predict the signal components $S_i$. We envision a scenario where the data $(Y_i,D_i)$, $i=1,\ldots,n$ is already collected, and we construct the denoisers using this dataset. With \emph{in-sample denoising} we denoise  $Y_i$ to predict the signal components $S_i$. This makes sense in many applications where we want to use the entire dataset to construct the denoiser.

 A closely related scenario is \emph{out-of-sample denoising}, where we want to denoise a new datapoint $(Y_0,D_0)$. This arises in applications where new samples are made available after an initial pre-processing of $Y_1,\dots,Y_n$ is performed, and it is not desired or not feasible to repeat this processing on the augmented dataset $(Y_0,D_0),(Y_1,D_1),\dots,(Y_n,D_n)$ for every new data point.

 While these two settings are very closely related, it turns out, perhaps surprisingly, that the optimal way to construct the denoisers differs substantially between the two. The reason turns out to be closely related to the observation that in high dimensions, the addition of a single datapoint changes the direction of the PCs. We will explain this phenomenon in detail below.

We first study the reduced-noise model from \eqref{po_def} in the setting of Cor.\ \ref{standard_spike_proj}, where the diagonal entries of $D_i$ are drawn iid from a distribution with mean $\mu$ and variance $\sigma^2$, and the observations are $Y_i = D_i(S_i+\ep_i)$. This is a special type of \emph{random effects} model, as the ``effects'' $z_{ik}$ of the ``factors'' $u_k$ in $\smash{S_i = \sum_{k=1}^r \ell_k^{1/2} z_{ik} u_k}$ are random from sample to sample.  As usual in random effects models, the optimal way to predict $S_i$ from a mean squared error (MSE) perspective is to use the Best Linear Predictor---or BLP---\cite[e.g.,][Sec.\ 7.4]{searle2009variance}. The BLP of $S_i$ is the predictor $\hat S_i = MY_i$ that minimizes $\E\|\hat S_i - S_i\|^2$.

It is well known that the BLP is $\hat S_i^{BLP}  = \Cov{S_i,Y_i}  \Cov{Y_i,Y_i}^{-1} Y_i$.  Under the assumptions of Cor.\ \ref{standard_spike_proj}, we can show  (see Sec.\ \ref{sv_shr_equiv}) that the BLP has the same asymptotic MSE properties as a denoiser of the following simpler form:
 \beq\label{blp_multi}
\hat S_i^{\tau,B} = \sum_{k=1}^{r} \tau_k u_k u_k^\top Y_{i}.
 \eeq
 The denoisers $\hat S_i^{\tau,B}$ are indexed by $\tau = (\tau_1,\ldots, \tau_r)$, and our argument shows that with the choice $\tau_k = \mu\ell_k/(\mu^2 \ell_k + m)$ they are in fact asymptotically equivalent to BLP denoisers.

 In practice, the true PCs $u_k$ are not known, so we use the Empirical BLP (EBLP), where we estimate the unknown parameters $ u_k$ using the entire dataset. Here we will use the $k$-th top right singular vector $\hat u_k$ of the $n \times p$ matrix $Y$ with rows $Y_i^\top$ as an estimator of $u_k$. In analogy with the simplified form of the denoisers in \eqref{blp_multi}, we will consider EBLPs scaled by $\eta = (\eta_1,\ldots,\eta_r)$ having the form:
 \begin{equation}
\label{eblp}
 \hat S_i^{\eta} = \sum_{k=1}^r \eta_k \hat u_k \hat u_k^\top Y_{i}.
\end{equation}
 Our goal is to find the optimal scalars $\eta_k$, and characterize their MSE.

 \subsection{In-sample denoising}

 First, we will study \emph{in-sample} denoising, where the data to be denoised are also used to construct the denoisers. Our main findings for the asymptotic MSE (AMSE) are summarized in Table \ref{denoise_tab} (in the single-spiked case) and in the following theorem, proved in Sec.\ \ref{den_pfs}.

\begin{table}[]
\centering
\caption{Denoising in the single-spiked case. BLP: population Best Linear Predictor using $u$. EBLP: in-sample empirical Best Linear Predictor using $\hat u$. EBLP-OOS: out-of-sample empirical Best Linear Predictor using $\hat u$. Here we abbreviate $\lambda  = \lambda(\delta\ell;\gamma)$, $c^2 = c^2(\delta\ell;\gamma)$, $s^2 = 1-c^2$, $\beta = 1 +\gamma/(\delta\ell)$.}
\label{denoise_tab}
{\renewcommand{\arraystretch}{1.8}
\begin{tabular}{|l|l|l|l|l|l|}
\hline
Name & Definition & Asy MSE & Asy Opt $\eta$  & Asy Opt MSE & Ref\\ \hline
BLP & $\eta \cdot u u^\top Y_i $& 
$(1-\eta\mu)^2\ell + \eta^2 m$                           
& $\frac{\mu \ell}{\mu^2 \ell+m}$ 
& $\frac{m\ell}{\mu^2 \ell+m}$ 
& Thm.\ \ref{den_thm} \\ \hline
EBLP & $\eta \cdot \hat u \hat u^\top Y_i$ 
& $\ell +\eta^2 \cdot m \lambda - 2\eta \cdot \mu  \ell c^2 \cdot \beta$ 
& $\frac{\mu \ell c^2}{\mu^2 \ell+m}$ 
& $ \ell\cdot\frac{\mu^2 \ell c^2 s^2 + m}{\mu^2\ell c^2 + m}$ 
& Thm.\ \ref{den_thm} \\ \hline
EBLP-OOS & $\eta \cdot \hat u \hat u^\top Y_i$ 
& $\ell + \eta^2 \cdot (\mu^2 \ell c^2 + m) -  2\eta\cdot\mu\ell c^2$ 
& $\frac{\mu\ell c^2}{\mu^2 \ell c^2+m}$ 
& $ \ell\cdot\frac{\mu^2 \ell c^2 s^2 + m}{\mu^2\ell c^2 + m}$ 
& Prop. \ref{oos-prop} \\ \hline
\end{tabular}
}
\end{table}

\begin{thms}[In-sample denoising]
\label{den_thm}
In the setting of Cor.\ \ref{standard_spike_proj} consider in-sample best linear predictors (BLP) of the signals $S_i$ based on the observations $Y_i$. 
\benum
\item The BLP denoisers $\hat S_i^{\tau,B} = \sum_{k=1}^r \tau_k \cdot u_k u_k^\top Y_{i}$ based on the population singular vectors have an AMSE  $\lim_{n,p\to\infty} \E \|S_i - \hat S_i^{\tau,B}\|^2$ of 
$$AMSE^{B} (\tau_1,\ldots,\tau_r; \ell_1,\ldots,\ell_r, \gamma) = \sum_{k=1}^r AMSE^{B}(\tau_r;\ell_r, \gamma),$$
where $AMSE^{B}(\tau;\ell, \gamma) =  (1-\tau\mu)^2\ell + \tau^2 m$ is the AMSE of the BLP $\hat S_i^{\tau,B} = \tau \cdot u u^\top Y_{i}$ in a single-spiked model with spike strength $\ell$  under the assumptions of Cor.\ \ref{standard_spike_proj}. The asymptotically optimal coefficients are
$$\tau_k^* = \frac{\mu \ell_k}{\mu^2\ell_k+m}.$$ 
\item The EBLP denoisers $\hat S_i^{\eta} = \sum_{k=1}^r \eta_k \cdot \hat u_k \hat u_k^\top Y_{i}$, based on the empirical singular vectors have an AMSE   $\lim_{n,p\to\infty} \E \|S_i - \hat S_i^{\eta}\|^2$ of 
        \begin{align}
        \label{err_insample}
        AMSE^{E} (\eta_1,\ldots,\eta_r; \ell_1,\ldots,\ell_r, \gamma) 
            = \sum_{k=1}^r AMSE^{E}(\eta_r;\ell_r, \gamma),
        \end{align}
 where $AMSE^{E}(\eta;\ell, \gamma) = \ell +\eta^2 \cdot m \cdot \lambda(\delta\ell;\gamma) - 2\eta \cdot \mu  \ell\cdot  c^2(\delta\ell;\gamma) \cdot \beta$ is the AMSE of EBLP $\hat S_i^{\eta}= \eta \cdot \hat u \hat u^\top Y_{i}$ in a single-spiked model with spike strength $\ell$  under the assumptions of Cor.\ \ref{standard_spike_proj}. Here  $ \lambda(\delta\ell;\gamma)$ is the limit empirical spike, while $c^2 = c^2(\delta\ell;\gamma)$ is the squared cosine, both corresponding to spike strength $\delta\ell$, defined in Cor.\ \ref{standard_spike_proj}. The asymptotically optimal coefficients are
        \begin{align}
        \label{eta_in}
        \eta_k^{*}  =\frac{\mu \ell_kc_k^2}{\mu^2 \ell_k+m},
        \end{align}
 where $c_k^2 =  c^2(\delta\ell;\gamma)$.
\eenum
\end{thms}

The basic discovery is that \emph{the optimal coefficients for empirical PCs are different from those for population PCs}. The coefficients using empirical PCs are reduced by a squared cosine compared to the coefficients using population PCs: $\eta_k^* = c_k^2 \tau_k^*$. 


 Note that we chose the optimal coefficients to minimize the limiting MSE. However, the limiting MSE, and thus the optimal coefficients, depend on the unknown parameters $\delta,\ell_k,c_k^2$. To make this a practical method, we can estimate the unknown parameters. The missingness parameter $\delta$ can be estimated by plug-in. Based on Cor.\ \ref{standard_spike_proj}, the estimation of $\ell_k$ and $c_k^2$ can be done by inverting the spike forward map $\ell_k \to \lambda(\ell_k)$, see e.g.,  \cite{bai2012estimation,donoho2013optimal}. 

 It is worth pointing out that the squared error $\|S_i - \hat S_i\|^2$ for each individual column $S_i$ of BLP and EBLP does not converge in probability or a.s. In fact, its variability is of unit order, and does not decrease as $n,p\to\infty$. However, as shown in Thm \ref{den_thm}, its expectation---the MSE---does converge. Furthermore, we will show in Sec.\ \ref{den_sv_shr} that the average error of EBLP over the \textit{entire data matrix} converges \textit{almost surely} (to the AMSE for a single column). As we will show, this is because EBLP, when applied to all columns of the data matrix, is a \textit{singular value shrinkage} estimator, defined by modifying the singular values of the data matrix while leaving the empirical singular vectors fixed.

 As a consequence, in the case of missing data the optimal AMSE agrees with that achieved by optimal singular value shrinkage described in \cite{nadakuditi2014optshrink} and \cite{gavish2014optimal}. We emphasize, however, that Thm.\ \ref{den_thm} applies to \emph{individual} data points, not just to the entire matrix.

\subsubsection{Comments on the proof}

Part 2 of Thm.\ \ref{den_thm} is a nontrivial result, because the AMSE is determined by stochastically dependent random quantities such as $u_{k}^\top D_i \hat u_{k}$. These are challenging to study, because $D_i$ and $\hat u_k$ are dependent random variables. We will analyze these quantities from first principles. 

 While we will explain our method in detail later, briefly, we use the \emph{outlier equation} approach (see Lemma \ref{outlier_eq_lem}), which reduces studying the inner products $w^\top \hat u_{k}$ to certain inner products $w^\top r$, where $r=r(Y)$ are vectors that depend on the entire dataset but not directly on the singular vector $\hat u_k$. As we will see, these inner products are more convenient to study. The basic method was introduced by \cite{benaych2012singular}, who used it to study the angles between $u_k$ and $\hat u_k$. We extend their approach to other angles, which are more challenging to study.

 \subsubsection{Singular value shrinkage and the almost sure convergence of the error in the reduced-noise model}
 \label{den_sv_shr}
 A well-studied approach to matrix denoising is known as \textit{singular value shrinkage}. Here, the singular values of the data matrix $Y$ are replaced with shrunken versions, analogous to the eigenvalue shrinkage of covariance matrices studied in Sec.\ \ref{sec:cov_est}. When applied to every row of the data matrix $Y$, EBLP is a singular value shrinkage algorithm: indeed, since $\hat S_i^{\eta} = \sum_{k=1}^r \eta_k \cdot \hat u_k \hat u_k^\top Y_{i}$, we can write the entire denoised matrix in the form
        \begin{align}
        \label{S_eta}
        \hat{S}^\eta = \sum_{k=1}^r \eta_k \cdot \hat u_k \hat u_k^\top Y
                     = \sum_{k=1}^r \eta_k \sigma_k(Y) \cdot \hat u_k \hat v_k^\top.
        \end{align}
 In other words, the denoised matrix $\hat{S}^\eta$ has the same singular vectors as the data matrix $Y$, where the singular values have been moved from $\sigma_k(Y)$ to $\eta_k \sigma_k(Y)$ for $k=1,\dots,r$ (and the remaining ones set to 0). It turns out that for the value of $\eta_k^*$ given by equation~\eqref{eta_in}, $\eta_k \sigma_k(Y)$ is the optimal singular value of the denoised matrix if we seek to minimize the asymptotic Frobenius loss $\| S - \hat{S}^\eta\|_F^2$.

 The proof of this fact follows easily from the analysis of the optimal singular value shrinkers given by \cite{gavish2014optimal}. This paper derives optimal shrinkers only when there is no missing data, and the task is to recover a low rank matrix from noisy observations of its entries. The $k^{th}$ singular value of the asymptotically optimal shrunken matrix (with respect to Frobenius loss) is $\ell_k^{1/2} c_k \tilde{c}_k$, where $c_k$ is the asymptotic cosine of the angle between the right population singular vector and the right empirical singular vector, and $\tilde{c}_k$ is the asymptotic cosine of the angle between the left population singular vector and the left empirical singular vector. In the models considered in \cite{gavish2014optimal}, these values are $c_k = c(\ell_k;\gamma)$, where $c$ is the cosine forward map from \eqref{cos_map}, and $\tilde{c}_k = \tilde c(\ell_k;\gamma)\ge 0$, where  $\tilde c^2(\ell;\gamma)=(1 - \gamma/\ell^2)/(1 + 1/\ell) $.

 As we can see from Cor.\ \ref{standard_spike_proj}, in our data model the reduced model behaves like the original model, but with spike strengths reduced from $\ell_k$ to $\delta \ell_k$. In particular, the value of $c_k^2$ is $c_k^2 = c^2(\delta \ell;\gamma)$. In fact, it is easy to see from the proof of that this is true for the left singular vectors as well; that is, $\tilde{c}^2 = \tilde{c}^2(\delta \ell_k;\gamma)$. 

 We now quote the result of \cite{gavish2014optimal} that the optimal singular value is equal to $\ell_k^{1/2}c_k\tilde{c}_k$; since formula \eqref{S_eta} shows that this is equal to $\eta^*_k \sigma_k(Y)$, the optimal coefficient $\eta_k^*$ is:
        \begin{align*}
        \eta^*_k = \frac{\ell_k^{1/2} c(\delta \ell_k  ; \gamma) 
                       \tilde{c}(\delta \ell_k  ; \gamma) }
                      {\sigma_k(Y)}
               = \ell_k^{1/2} c^2(\delta \ell_k  ; \gamma) 
                 \frac{\tilde{c}(\delta \ell_k  ; \gamma) }
                      {\ell_k^{1/2} \sigma_k(Y) c(\delta \ell_k  ; \gamma) }
        \end{align*}
 Substituting the asymptotic value for $\sigma_k(Y)$ given from \eqref{spike_eq_white}, a straightforward algebraic manipulation shows that
        \begin{align*}
        \eta^*_k           = \ell_k^{1/2} c^2(\delta \ell_k  ; \gamma) 
                 \frac{\ell_k^{1/2}}{\delta \ell_k + 1} 
               = \frac{\ell_k c_k^2}{\delta \ell_k + 1}.
        \end{align*}

 This agrees with the value for the optimal coefficient $\eta_k$ for EBLP derived in Thm.\ \ref{den_thm}. In particular, the EBLP estimator we derive for the entire data matrix and the optimal singular value shrinkage estimator are identical. Furthermore, from the analysis of \cite{gavish2014optimal}, the error of the \textit{entire} denoised matrix converges \textit{almost surely} to the expression given in equation \eqref{err_insample}.

 We have proved the following theorem:

 \begin{thms}
 If $Y = [Y_1,\ldots,Y_n]^\top$ is the $n$-by-$p$ matrix of observations in the reduced-noise model, then the asymptotically optimal singular value shrinkage estimator of the $n$-by-$p$ signal matrix $S = [S_1,\ldots, S_n]^\top$ is equal to the EBLP estimator defined in Thm.\ \ref{den_thm}. Furthermore, the squared Frobenius error of this estimator converges a.s.\ to the formula \eqref{err_insample}.
 \end{thms}

 \subsubsection{Comparison with matrix completion algorithms}
 In the case when the reductions matrices are binary, the task of denoising the data matrix $Y$ to approximate $S$ is a \textit{matrix completion} problem -- see, for instance, \cite{candes2009exact}, \cite{recht2011simpler}, \cite{keshevan2009optspace} and \cite{keshevan2010noisy}.  A typical model for matrix completion is a low-rank matrix with eigenvectors satisfying an incoherence condition, with order $O(n \cdot \text{poly}(\log n))$ entries revealed uniformly at random. In the setting we study in this paper, the number of observed entries of the matrix $Y$ will be $O(n^2)$, almost an order of magnitude more. When the noise level is very small compared to the magnitude of the entries in the clean matrix (or put differently, when $\ell_r \gg 1$), then the smaller number of samples is sufficient to recover the low-rank matrix to high accuracy. The references listed above provide several methods with these guarantees.

 We compare the in-sample EBLP denoiser to the OptSpace method for matrix completion found in \cite{keshevan2009optspace} and \cite{keshevan2010noisy}. Briefly, their method removes rows/columns with too many observations, truncates the singular values of the data matrix $Y$, and then cleans up the resulting matrix by an iterative algorithm. In Figure~\ref{fig201b} we plot the squared Frobenius errors in reconstructing a rank 1 matrix, for different choices of spike size $\ell$ and missingness parameter $\delta$. In this experiment, both the signal and noise are Gaussian, $\gamma=.8$ and the dimension $p=400$. Each data point plotted is the average error over 100 Monte Carlo runs of the experiment.

 We observe that the in-sample EBLP outperforms OptSpace when $\ell$ is small relative to the noise level. As the size of $\ell$ grows, OptSpace's performance improves, and for small $\delta$ and large $\ell$ it outperforms EBLP. This is consistent with the guarantees provided for OptSpace; it does well in the low-noise regime, with a small number of samples.

 \begin{figure}[h]
 \centerline{
 \includegraphics[scale=.7]{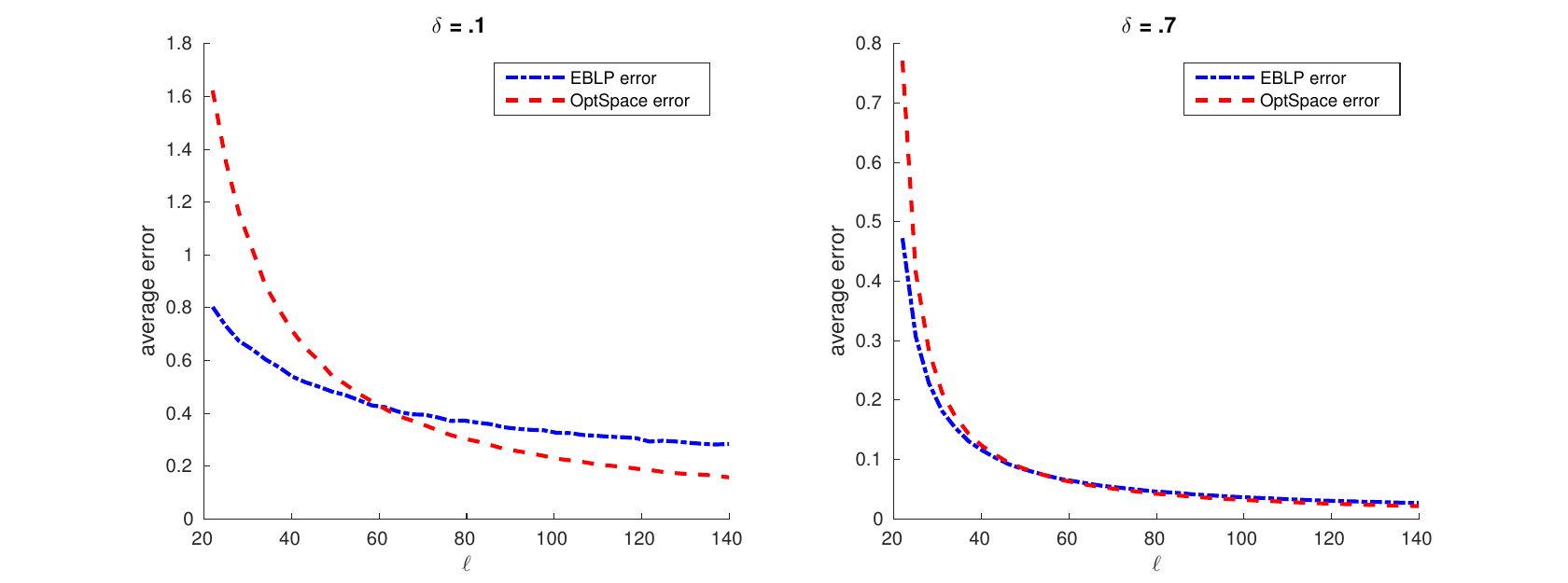}
}
 \caption{Comparison of EBLP and OptSpace for completing/denoising the matrix, for different signal strengths. Left; $\delta = 1/10$. Right: $\delta = 7/10$. Errors are measured as the squared Frobenius norm between the estimator and the full, clean, rank 1 matrix. Every experiment was averaged over 100 runs, with $\gamma=.8$ and $p=400$.}
 \label{fig201b}
 \end{figure}

 \subsection{Out-of-sample denoising}
 \label{oos}

We now study \emph{out-of-sample} denoising in the reduced-noise model, where we denoise new datapoints from the same distribution using a denoiser constructed on an existing dataset. This is typically faster than recomputing the denoiser on the entire dataset. For the oracle BLP denoiser, which assumes knowledge of $u$, this is the same as in-sample denoising. For the EBLP, however, it turns out that the optimal shrinkage coefficients in this case are different.  

To analyze this case, let $Y_0 = D_0 X_0$ be the new sample from the same distribution. We evaluate the limit of the out-of-sample mean squared prediction error $\E \|S_0 -\hat S_0^\eta\|^2$ of the EBLP $\hat S_0^{\eta} =\sum_{k=1}^r \eta_k \hat u_k \hat u_{k}^\top Y_0$, where $\hat u_k$ were formed based on $Y_i$, $i\ge 1$. 

\begin{thms}
\label{oos-prop}
In the setting of Cor.\ \ref{standard_spike_proj} consider out-of-sample denoising of a new sample $Y_0$ using the empirical BLP $\hat S_0^{\eta}$ based on the observations $Y_i$, $i=1,\ldots,n$. Then the limit of the out-of-sample prediction error is 
$$E^{\eta,o}  (\eta_1,\ldots,\eta_r; \ell_1,\ldots,\ell_r, \gamma) = \sum_{k=1}^r E^{\eta_k,o}(\eta_r;\ell_r, \gamma),$$ 
where $E^{\eta,o}(\eta;\ell, \gamma) =  \ell + \eta^2 \cdot (m +  \mu^2 \ell c^2) -  2\eta\cdot\mu\ell c^2$ is the out-of-sample prediction error of EBLP in a single-spiked model with spike $\ell$ under the assumptions of Cor.\ \ref{standard_spike_proj}. The asymptotically optimal shrinkage coefficients are
$$\eta_k^* = \frac{\mu\ell_k c_k^2}{m +\mu^2 \ell_k c_k^2}.$$
\end{thms}

See Sec.\ \ref{oos-prop-pf} for the proof. The key point is that the optimal shrinkage for out-of-sample prediction is \emph{different} from \emph{both} of the shrinkers from in-sample denoising.  We also mention that in the special case when $D_i = I_p$ for all $i$, the optimal shrinkage formula matches the one obtained by \cite{singer2013two}, under slightly more restrictive assumptions.

 It may be counterintuitive that the optimal coefficient changes when a single data point $Y_0$ is added. However, this can be understood because in high-dimensions, a single data point can drastically change the empirical eigenvectors.  For an illustration in a simpler setting, consider the sample covariance matrix based on $n$ samples, $\smash{\hat{\Sigma}_{Y,n} = n^{-1} \sum_{k=1}^n Y_k Y_k^\top}$ and the corresponding sample covariance $\hat{\Sigma}_{Y,n+1}$ based on the $n+1$ samples $Y_0,Y_1,\dots,Y_n$. Their difference is
        \begin{align*}
        \hat{\Sigma}_{Y,n} - \hat{\Sigma}_{Y,n+1}
            = \frac{1}{n} Y_0 Y_0^\top + \frac{1}{n(n+1)} \sum_{k=1}^{n+1} Y_k Y_k^\top
            = \frac{1}{n} Y_0 Y_0^\top + \frac{1}{n} \hat{\Sigma}_{Y,n+1}.
        \end{align*}
 Since the operator norm of $\hat{\Sigma}_{Y,n+1}$ converges a.s.\ to a finite quantity, the term $\hat{\Sigma}_{Y,n+1} / n$ is asymptotically negligible. However, the operator norm of $Y_0 Y_0^\top / n$ of size $\| Y_0\|^2 / n$, which converges a.s.\ to $\delta$. Since the spectral distributions of $\hat{\Sigma}_{Y,n}$ and $\hat{\Sigma}_{Y,n+1}$ converge to the same value, the fact that $\| \hat{\Sigma}_{Y,n} -  \hat{\Sigma}_{Y,n+1}\|_{op}$ is of order 1 is due to the fact that the addition of a single data point $Y_0$ completely changes the direction of the empirical PCs of the data.

To summarize and better understand our findings, we plot the optimal shrinkage coefficients and MSE for the three scenarios (BLP, in-sample optimal empirical BLP, and out-of-sample optimal empirical BLP) in a single-spiked model with $\gamma = 1/2$ and $D_i=I_p$ for all $i$, on Fig. \ref{blp_eblp_oos}. The optimal shrinkage for out-of-sample EBLP is intermediate between the stronger in-sample EBLP and the weaker out-of-sample BLP shrinkage coefficients. However, perhaps unexpectedly the MSE for the two EBLP scenarios agrees exactly! This prompts us to state the following result, proved in Sec.\ \ref{mse_equal}.

\begin{prop}[In-sample vs out-of-sample EBLP]
\label{is-vs-oos-prop}

The asymptotically optimal shrinkage coefficient for in-sample EBLP is smaller than the asymptotically optimal shrinkage coefficient for the out-of-sample EBLP. However, the asymptotically optimal MSEs are equal in the two cases.
\end{prop}

This result is interesting, because it shows that same MSE can be achieved out-of-sample as in-sample. Out-of-sample denoising should be harder, because it involves a new datapoint never seen before. The "hardness" of out-of-sample denoising should be observed in the leading order finite sample ($n$) correction to the asymptotic MSE.  However, the above result shows that this correction vanishes as $n,p\to\infty$. Using the right amount of shrinkage, the same MSE can be achieved asymptotically even out of sample.

\begin{figure}
\centering
\begin{subfigure}{.5\textwidth}
  \centering
  \includegraphics[scale=0.4]{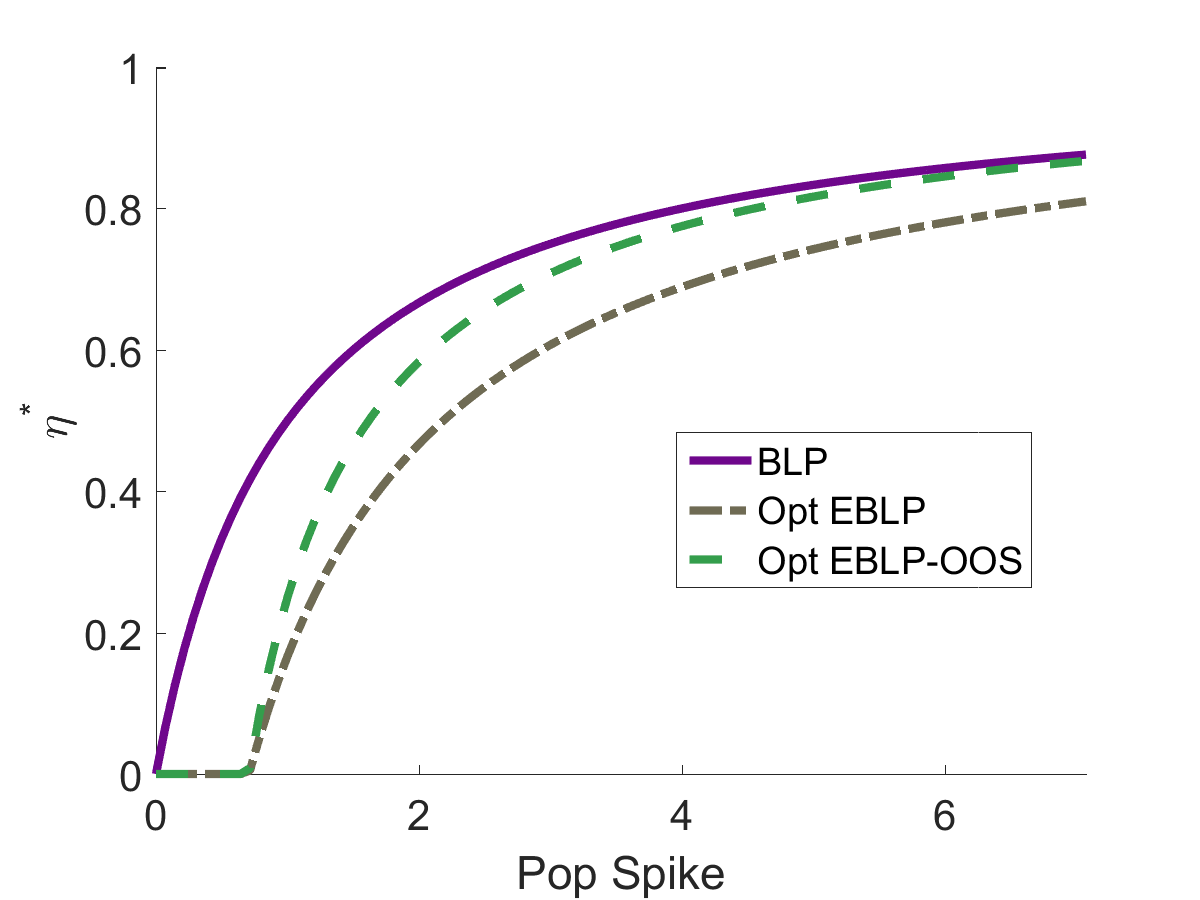}
\end{subfigure}%
\begin{subfigure}{.5\textwidth}
  \centering
  \includegraphics[scale=0.4]{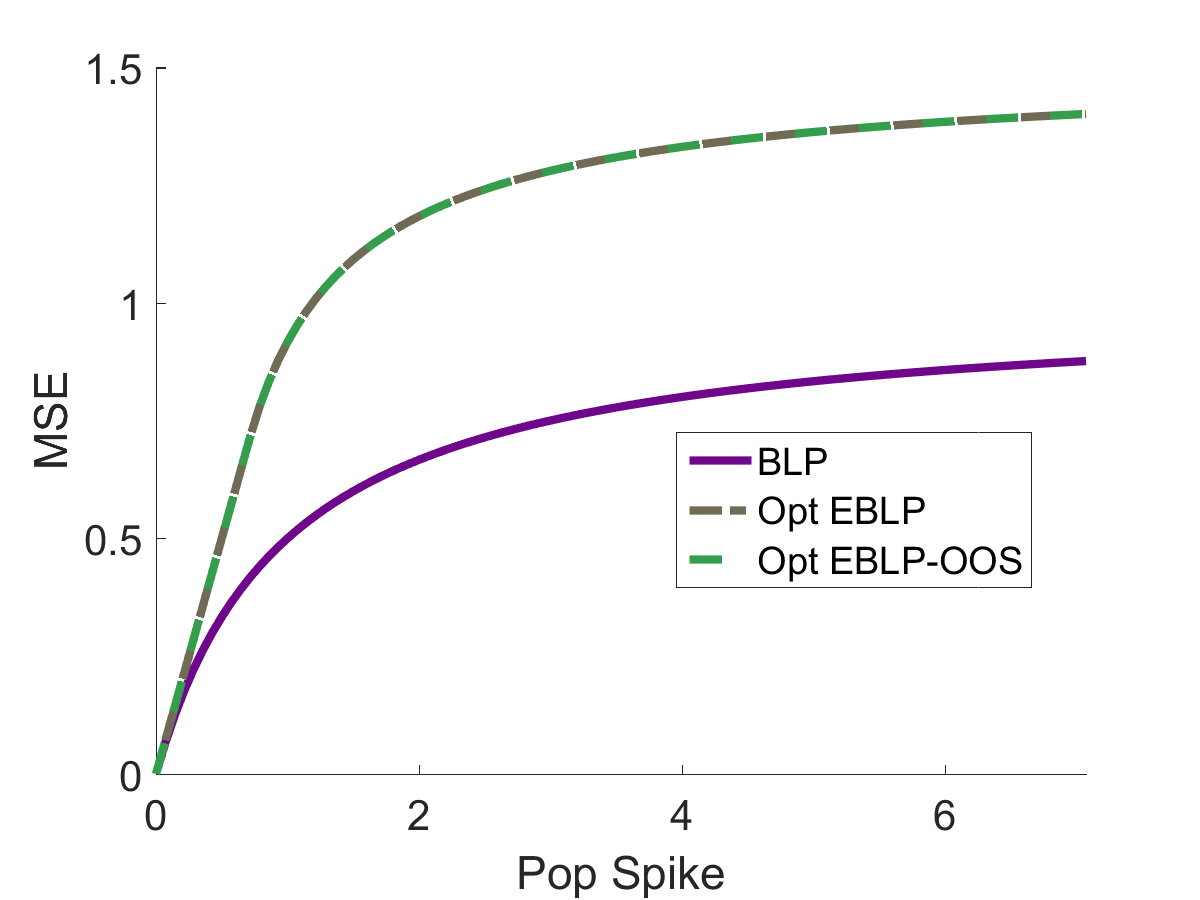}
\end{subfigure}
\caption{Optimal shrinkage coefficients (left) and MSE (right) for the three scenarios (BLP, in-sample optimal empirical BLP, and out-of-sample optimal empirical BLP) for $\gamma = 1/2$ and $\delta=1$.}
\label{blp_eblp_oos}
\end{figure}

\subsection{Unreduced-noise}
\label{den_po_def2}

We now study the denoising problem under the unreduced-noise model \eqref{po_def2}, where $Y_i  = D_i S_i +\ep_i$ under the assumptions of Cor.\ \ref{standard_spike_proj}. The analysis is similar to the reduced-noise model \eqref{po_def}. The key conclusions are summarized in Table \ref{denoise_tab_unproj noise}.

\begin{table}[]
\centering
\caption{Unreduced-noise: Denoising in the single-spiked case. BLP: population Best Linear Predictor using $u$. EBLP: in-sample empirical Best Linear Predictor using $\hat u$. EBLP-OOS: out-of-sample empirical Best Linear Predictor using $\hat u$.  Here we abbreviate $\lambda  = \lambda(\mu^2\ell;\gamma)$, $c^2 = c^2(\mu^2\ell;\gamma)$.}
\label{denoise_tab_unproj noise}
{\renewcommand{\arraystretch}{1.8}
\begin{tabular}{|l|l|l|l|l|l|}
\hline
Name & Definition & Asy MSE & Asy Opt $\eta$  & Asy Opt MSE\\ \hline
BLP & $\eta \cdot u u^\top Y_i $& 
$\ell + \eta^2 \cdot  (\mu^2 \ell + 1)  -  2\eta \cdot  \mu\ell$                           
& $\frac{\mu\ell}{\mu^2\ell+1}$ 
& $\frac{\ell}{\mu^2\ell+1}$ \\ \hline
EBLP & $\eta \cdot \hat u \hat u^\top Y_i$ 
& $\ell + \eta^2 \cdot \lambda 
-  2\eta \cdot  \mu\ell c^2 \cdot [1+\gamma/(\mu^2\ell)]$ 
& $\frac{\mu \ell c^2}{\mu^2\ell+1}$ 
& $ \ell -  \lambda\left(\frac{\mu \ell c^2}{\mu^2\ell+1}\right)^2$  \\ \hline
EBLP-OOS & $\eta \cdot \hat u \hat u^\top Y_i$ 
& $\ell + \eta^2  \cdot  (\mu^2\ell c^2+1) -  2\eta \cdot \mu \ell c^2$ 
& $\frac{\mu \ell c^2}{\mu^2\ell c^2+1}$ 
& $\ell \cdot  \frac{\mu^2\ell c^2 s^2+1}{\mu^2 \ell c^2+1}$ 
\\ \hline
\end{tabular}
}
\end{table}

The BLP of $S_i$ based on $Y_i$ is $\hat S_i^{BLP}  = \Cov{S_i,Y_i}  \Cov{Y_i,Y_i}^{-1} Y_i$. 
Under the conditions of Cor.\ \ref{standard_spike_proj}, we can show (see Sec.\ \ref{pf_unproj_blp}) that  this is asymptotically equivalent to
$
\hat S_i^\tau = \sum_{k=1}^r\tau_k  u_k u_{k}^\top Y_i,
$
where $\tau_k = \mu \ell_k/(\mu^2 \ell_k + 1)$. 

As before, the AMSE of $S_i^\tau$ with arbitrary $\tau$ decouples into the AMSEs over the different spikes $\ell_k$, and those are equal to the AMSEs for the single-spiked model with spikes equal to $\ell_k$. For a single-spiked model with spike $\ell$ we obtain in Sec.\ \ref{pf_unproj_blp} that
\begin{align*}
E\|S_i -\hat S_i^{\tau}\|^2 \to 
\ell + \tau^2  (\ell \mu^2 + 1)  -  2\tau \ell \mu.
\end{align*}
The optimal coefficient is $\tau^* = \ell\mu/(\ell\mu^2+1)$---as it should be, based on the above discussion---and it has an AMSE of $\ell/(\ell\mu^2+1)$. The advantage of this calculation is that it provides the MSE for any coefficient $\tau$.

Next, for the EBLP in the multispiked case, we use $\hat u_k$ as estimators of $u_k$. Since the form of the BLP is the same as before, the EBLP scaled by $\eta = (\eta_1,\ldots,\eta_r)$ have the form in \eqref{eblp}. To compute the AMSE, it is again not hard to see that it decouples into the corresponding single-spiked AMSEs. In the single-spiked case, we find in Sec.\ \ref{pf_unproj_eblp} that with $\lambda  = \lambda(\mu^2\ell;\gamma)$, $c^2 = c^2(\mu^2\ell;\gamma)$,
\[
\E\|S_i -\hat S_i^{\eta}\|^2 \to \ell + \eta^2 \cdot \lambda 
-  2\eta \cdot  \mu\ell c^2 \cdot [1+\gamma/(\mu^2\ell)].
\]
This shows that the optimal coefficient is $\eta^* = \mu \ell c^2/(\mu^2\ell+1)$.

Finally, for out-of-sample EBLP denoising, it is again not hard to see that the AMSE decouples over the different spikes, and each term equals the AMSE in the single-spiked case. For the AMSE in the single-spiked case, we let $(Y_0,D_0)$ be a new sample, and find (Sec.\ \ref{pf_unproj_eblp_oos})
$$\E \|S_0 -\eta \hat u \hat u^\top Y_{0}\|^2 \to 
 \ell + \eta^2 (1 + \mu^2\ell c^2) -  2\eta\mu \ell c^2.$$
The optimal coefficient is $\eta^* = \mu \ell c^2/(1 + \mu^2\ell c^2)  $, while the optimal MSE is $ \ell(1+\mu^2\ell c^2 s^2)/(1+\mu^2 \ell c^2)$. These findings are summarized in Table \ref{denoise_tab_unproj noise}.

 \subsubsection{Singular value shrinkage and the almost sure convergence of the error in the unreduced-noise model}
 \label{den_sv_shr_unre}
 As we discussed in Sec.\ \ref{den_sv_shr} for the reduced-noise model, in-sample EBLP applied to every column of the data matrix $Y = [Y_1,\dots,Y_n]^\top$ is equal to the asymptotically optimal singular value shrinkage estimator of the clean matrix $S = [S_1,\dots,S_n]^\top$. The same reasoning applies verbatim to the unreduced-noise model, replacing the almost sure limits of the empirical eigenvalues and angles with their counterparts for unreduced-noise model. 

 From Cor.\ \ref{standard_spike_proj}, the value of the asymptotic value of the cosine of the angle between the right empirical singular vector and the right population singular vector is $c_k^2$ is $c_k^2 = c^2(\mu^2 \ell;\gamma)$; and the same proof of this easily shows that the asymptotic cosine of the angle between the left singular vectors is $\tilde{c}^2 = \tilde{c}^2(\delta \ell_k;\gamma)$. 

 We now quote the formula from \cite{gavish2014optimal}, which says that the optimal singular value is equal to $\ell_k^{1/2}c_k\tilde{c}_k$; since formula \eqref{S_eta} shows that this is equal to $\eta^*_k \sigma_k(Y)$, the optimal coefficient $\eta_k^*$ is:
        \begin{align*}
        \eta^*_k = \frac{\ell_k^{1/2} c(\delta \ell_k  ; \gamma) 
                       \tilde{c}(\delta \ell_k  ; \gamma) }
                      {\sigma_k(Y)}
               = \ell_k^{1/2} c^2(\delta \ell_k  ; \gamma) 
                 \frac{\tilde{c}(\delta \ell_k  ; \gamma) }
                      {\ell_k^{1/2} \sigma_k(Y) c(\delta \ell_k  ; \gamma) }
        \end{align*}
 Substituting the asymptotic value for $\sigma_k(Y)$ given from Cor.\ \ref{standard_spike_proj}, it is easy to see that
        \begin{align*}
        \eta^*_k           = \ell_k^{1/2} c^2(\delta \ell_k  ; \gamma) 
                 \frac{\ell_k^{1/2}}{\delta \ell_k + 1} 
               = \frac{\ell_k c_k^2}{\delta \ell_k + 1}.
        \end{align*}

 This agrees with the value for the optimal coefficient $\eta_k$ for EBLP derived in Thm.\ \ref{den_thm}. In particular, the EBLP estimator we derive for the entire data matrix and the optimal singular value shrinkage estimator are identical. Furthermore, from the analysis of \cite{gavish2014optimal}, the error of the \textit{entire} denoised matrix converges \textit{almost surely} to the expression given in equation \eqref{err_insample}.

 We have proved the following theorem:

 \begin{thms}
 If $Y = [Y_1,\ldots,Y_n]^\top$ is the $n$-by-$p$ matrix of observations in the unreduced-noise model, then the asymptotically optimal singular value shrinkage estimator of the $n$-by-$p$ signal matrix $S = [S_1,\ldots, S_n]^\top$ is equal to the EBLP estimator for the unreduced-noise model defined in Sec.\ \ref{den_po_def2}. Furthermore, the squared Frobenius error of this estimator converges a.s.\ to the AMSE in the unreduced-noise model.
 \end{thms}

\subsection{Simulations}
\label{denoise_simu}
\begin{figure}
  \centering
  \includegraphics[scale=0.4]{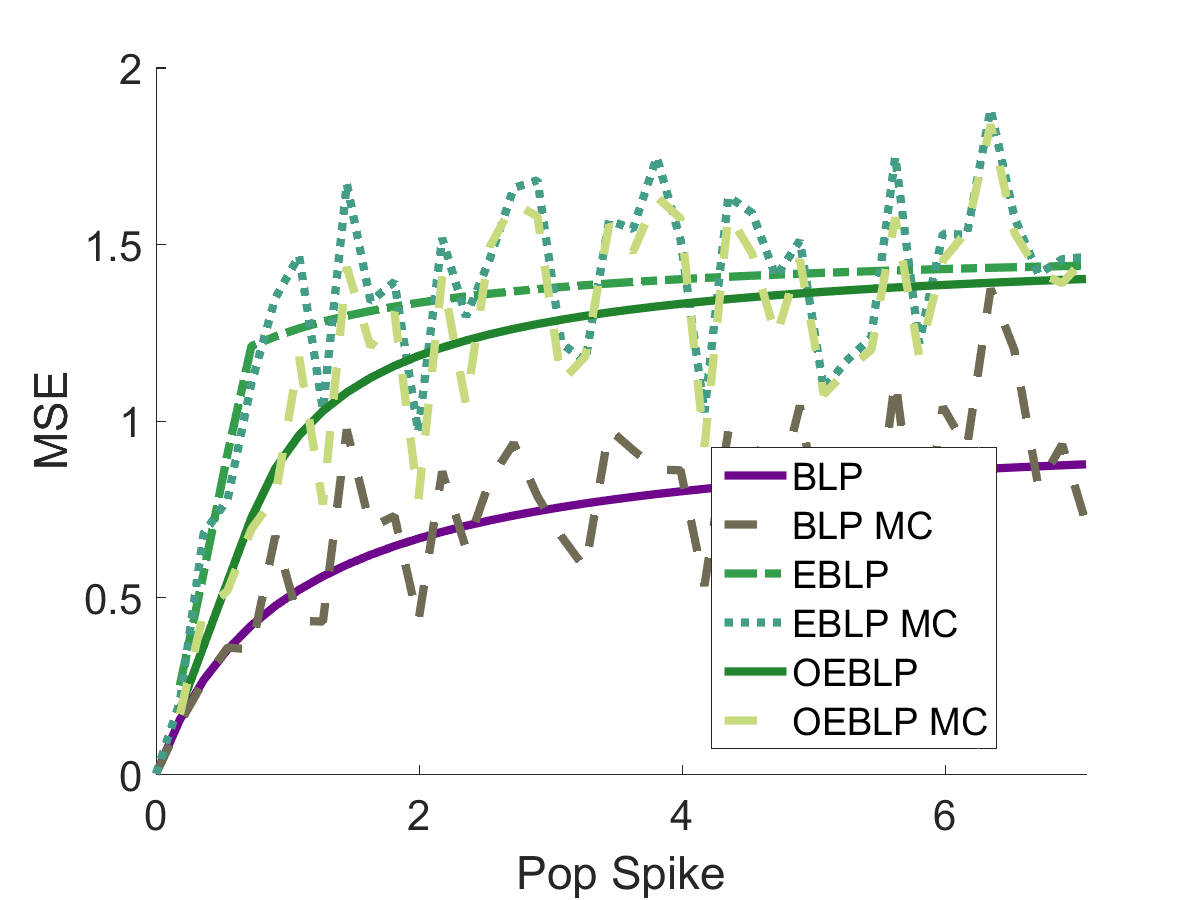}
\caption{In-sample MSE of denoising schemes: Theoretical results overlaid with Monte Carlo (MC) results. BLP: Best Linear Predictor assuming known population eigenvector $u$. Emp BLP: Empirical Best Linear Predictor, using the empirical eigenvector $\hat u$, with the sub-optimal shrinkage coefficient from the BLP. Opt Emp BLP: Empirical Best Linear Predictor using optimal shrinkage coefficient. No missing data ($\delta  = 1$). $\gamma = 1/2$.}
\label{denoise_mc}
\end{figure}

Next we perform a simulation to check the finite-sample accuracy of our formulas for denoising. In the reduced-noise model with $D_i = I_p$, we consider a single-spiked model with $n=1000$, $p=500$, and generate iid Gaussian random variables $z_i,\ep_i$, as well as a standardized iid Gaussian random vector $u$. We vary the spike strength on a grid, and compare our formulas for in-sample theoretical MSE to those obtained by averaging the denoising error in the first sample $\|\hat S_1  - S_1\|_2^2$ over 50 Monte Carlo simulations. The results in Fig. \ref{denoise_mc} show that the formulas are accurate up to the sampling error. This validates our results from Thm.\ \ref{den_thm}.

\section*{Acknowledgements}

The authors wish to thank Joakim Anden, Tejal Bhamre, Xiuyuan Cheng, David Donoho, and Iain Johnstone for helpful discussions. 

\section{Proofs}

\subsection{Proof of Thm.\ \ref{spike_proj_multi}}
\label{main_pf_proba}
The proof of Thm.\ \ref{spike_proj_multi} spans multiple sections, until Sec.\ \ref{pf_qf_lem3}.   The proof of the claims under observation models \eqref{po_def} and \eqref{po_def2} are very similar. Therefore, we present the proof of the result under model \eqref{po_def}, and outline the argument for model \eqref{po_def2} in Sec.\ \ref{sec_po_def2}. 

Moreover, to illustrate the idea of the proof, we first prove the single-spiked case, i.e., when $r=1$. The proof of the multispiked extension is provided in Sec.\ \ref{pf_multi}. The form $D_i = \mu + \Sigma^{1/2} E_i$ of the reduction matrices implies the following decomposition for the observations $Y_i$:
\begin{align*}
Y_i &=  (\mu + \Sigma^{1/2} E_i)X_i \\
&=\mu ( \ell^{1/2} z_i u + \ep_i) + \Sigma^{1/2} E_iX_i\\
&=\ell^{1/2} z_i \mu u + [\mu\ep_i +\Sigma^{1/2} E_iX_i].
\end{align*}
This suggests a ``signal+noise'' decomposition for the reduced vectors $Y_i$. Let us denote by $\nu=\mu u/\xi^{1/2}$ the normalized reduced signal, where $\xi = \|\mu  u\|^2 \to \tau$, and by $\ep^*_i = \mu\ep_i +\Sigma^{1/2} E_iX_i$ the noise component. The noise has two parts: $\mu\ep_i$ is due to sampling, while $\Sigma^{1/2} E_iX_i$ is due to projection. In matrix form, with the $n\times p$ matrix $Y$ having rows $Y_i^\top$:
\beq
\label{sig_noise}
Y = (\xi \ell)^{1/2} \tilde Z \nu^\top + \mathcal{E}^*.
\eeq

This suggests that after reduction, the signal strength $\ell$ changes to $\xi \ell$, while the noise structure changes from $\ep_i$ to $\ep_i^*$. This is not obvious, however, because the noise $\ep^*_i$ is functionally dependent on the signal $X_i$. Therefore we cannot rely on existing results. Instead, we will analyze the model from first principles, and show that the dependence is asymptotically negligible.  For non-diagonal reduction matrices $D_i$, the depencence may be asymptotically non-negligible; this explains why we currently need the diagonal assumption.

\subsubsection{Proof outline}
\label{setup}
 We will extend the technique of \cite{benaych2012singular} to characterize the spiked eigenvalues in the model \eqref{sig_noise}. We denote the normalized vector $Z = n^{-1/2}\tilde Z$, the normalized noise $N = n^{-1/2}\mathcal{E}^*$ and the normalized observable matrix $\tilde Y = n^{-1/2}Y$. Then, our model is 
\beq
\label{sig_noise_2}
\tilde Y = (\xi\ell)^{1/2}  \cdot Z \nu^\top + N. 
\eeq

 We will assume that $n,p \to \infty$ such that $p/n \to \gamma>0$. For simplicity of notation, we will first assume that $n \le p$, implying that $\gamma \ge 1$. It is easy to see that everything works when $n \ge p$.

 By Lemma 4.1 of \cite{benaych2012singular}, the singular values of $ \tilde Y $ that are not singular values of $N$ are the positive reals $t$ such that the 2-by-2 matrix
$$M_n(t) = 
\begin{bmatrix} 
t \cdot Z^\top (t^2 I_n - NN^\top)^{-1} Z & 
Z^\top (t^2 I_n - NN^\top)^{-1} N\nu \\ 
\nu^\top N^\top (t^2 I_n - NN^\top)^{-1} Z & 
t \cdot \nu^\top (t^2 I_p - N^\top N)^{-1} \nu
\end{bmatrix} 
-  
\begin{bmatrix} 0 & (\xi\ell)^{-1/2} \\  (\xi\ell)^{-1/2} & 0 \end{bmatrix}$$

is not invertible, i.e., $\det[M_n(t)]=0$. We will find almost sure limits of the entries of $M_n(t)$, to show that it converges to a deterministic matrix $M(t)$. Solving the equation $\det[M(t)] = 0$ will provide an equation for the almost sure limit of the spiked singular values of $\tilde Y$.  For this we will prove the following results: 

\begin{lem}[The noise matrix] 
\label{noise_lem}
The noise matrix $N$ has the following properties:
\benum
\item The eigenvalue distribution of $N^\top N$ converges almost surely (a.s.) to the Marchenko-Pastur distribution $F_{\gamma,H}$ with aspect ratio $\gamma \ge 1$. 
\item The top eigenvalue of $N^\top N$ converges a.s.\ to the upper edge $b_H^2$ of the support of $F_{\gamma,H}$. 
\eenum
\end{lem}

This is proved in Sec.\ \ref{pf_noise_lem}. For brevity we write $b  = b_H$. Compared to  \cite{benaych2012singular}, the key technical innovation here is to show that the contribution of the reduced signal component $\Sigma^{1/2} E_iS_i$ to $\mathcal{E}^*$ is negligible. This is accomplished by an ad-hoc bound on the operator norm of the contribution.

Since $\tilde Y$ is a rank-one perturbation of $N$, it follows that the eigenvalue distribution of $\tilde Y^\top \tilde Y$ also converges to the MP law $F_{\gamma,H}$. This proves the first claim of Thm \ref{spike_proj_multi}. 

Moreover, since $NN^\top$ has the same $n$ eigenvalues as the nonzero eigenvalues of $N^\top N$,  the two facts in Lemma \ref{noise_lem} imply that when $t>b$, $n^{-1}\tr (t^2 I_n - NN^\top)^{-1} \to \int (t^2-x)^{-1} d\underline F_{\gamma,H}(x) = -\underline m(t^2)$. Here $\underline F_{\gamma,H}(x)  = \gamma F_{\gamma,H}(x)  + (1-\gamma)\delta_0$ and $\underline m = \underline m_{\gamma,H}$ is the Stieltjes transform of $\underline F_{\gamma,H}$. Clearly this convergence is uniform in $t$. As a special note, when $t$ is a singular value of the random matrix $N$, we formally define $(t^2 I_p - N^\top N)^{-1} =0$ and $(t^2 I_n - NN^\top)^{-1} =0$. When $t>b$, the complement of this event happens a.s. In fact, from Lemma \ref{noise_lem} it follows that $(t^2 I_p - N^\top N)^{-1}$ has a.s. bounded operator norm. Next we control the quadratic forms in the matrix $M_n$.

\begin{lem}[The quadratic forms] 
\label{qf_lem}
When $t>b$, the quadratic forms in the matrix $M_n(t)$ have the following properties:
\benum
\item $Z^\top (t^2 I_n - NN^\top)^{-1} Z - n^{-1}\tr (t^2 I_n - NN^\top)^{-1}\to 0$ a.s.
\item $Z^\top (t^2 I_n - NN^\top)^{-1} N\nu \to 0$ a.s.
\item $\nu^\top (t^2 I_p - N^\top N)^{-1} \nu \to -m(t^2)$ a.s., where $m = m_{\gamma,H}$ is the Stieltjes transform of the Marchenko-Pastur distribution $F_{\gamma,H}$. 
\eenum

Moreover the convergence of all three terms is uniform in $t>b+c$, for any $c>0$.
\end{lem}

This is proved in Sec.\ \ref{pf_qf_lem}.  The key technical innovation is the proof of the third part. Most results for controlling quadratic forms $x^\top A x$ are concentration bounds for random $x$. Here $x=\nu$ is fixed, and matrix $A=(t^2 I_p - N^\top N)^{-1}$ is random instead. For this reason we adopt the ``deterministic equivalents'' technique of  \cite{bai2007asymptotics} for quantities $x^\top (z I_p - N^\top N)^{-1} x$, with the key novelty that we can take the imaginary part of the complex argument to zero. The latter observation is nontrivial, and mirrors similar techniques used recently in universality proofs in random matrix theory \citep[see e.g., the review by][]{erdos2012universality}.

Lemmas \ref{noise_lem} and \ref{qf_lem} will imply that for $t>b$, the limit of $M_n(t)$ is 
$$M(t) = 
\begin{bmatrix} 
-t \cdot \underline m(t^2) & 
-(\tau\ell)^{-1/2} \\ 
-(\tau\ell)^{-1/2} &
-t \cdot m(t^2) 
\end{bmatrix}.
$$

By the Weyl inequality, $\sigma_2(\tilde Y) \le \sigma_2( (\xi\ell)^{1/2}  \cdot Z \nu^\top) + \sigma_1(N) = \sigma_1(N)$. Since $\sigma_1(N) \to b$ a.s. by Lemma \ref{noise_lem}, we obtain that  $\sigma_2(\tilde Y) \to b$ a.s.  Therefore for any $\ep>0$, a.s. only $\sigma_1(\tilde Y)$ can be a singular value of $\tilde Y$ in $(b+\ep,\infty)$ that is not a singular value of $N$.

It is easy to check that $D(x) = x\cdot\underline m(x) m(x)$ is strictly decreasing on $(b^2,\infty)$.  Hence, denoting $h = \lim_{t\downarrow b} D(t^2)$, for $\tau\ell>h$, the equation $D(t^2)=1/(\tau \ell)$ has a unique solution $t\in (b,\infty)$. By Lemma A.1 of \cite{benaych2012singular}, we conclude that for $\tau\ell>h$, $\sigma_1(\tilde Y) \to t$ a.s., where $t$ solves the equation $\det[M(t)]=0$, or equivalently, 
$$
t^2 \cdot \underline m(t^2) m(t^2)  = \frac1{\tau\ell}.
$$
If $\tau\ell \le h$, then we note that $\det[M_n(t)]\to\det[M(t)]$ uniformly on $t>b+\ep$. Therefore, if $\det[M_n(t)]$ had a root $\sigma_1(\tilde Y)$ in $(b+\ep,\infty)$, $\det[M(t)]$ would also need to have a root there, which is a contradiction. Therefore, we conclude $\sigma_1(\tilde Y) \le b+\ep$ a.s., for any $\ep>0$. Since $\sigma_1(\tilde Y) \ge \sigma_2(\tilde Y) \to b$, we conclude that $\sigma_1(\tilde Y) \to b$ a.s., as desired. This finishes the spike convergence claim in Thm.\ \ref{spike_proj_multi}.

Next, we turn to proving the convergence of the angles between the population and sample eigenvectors. Let $\hat Z$ and $\hat u$ be the singular vectors associated with the top singular value $\sigma_1(\tilde Y)$ of $\tilde Y$. Then, by Lemma 5.1 of \cite{benaych2012singular}, if $\sigma_1(\tilde Y)$ is not a singular value of $X$, then the vector $\eta = (\eta_1,\eta_2)= (u^\top \hat u, Z^\top \hat Z)$ belongs to the kernel of the matrix $M_n(\sigma_1(\tilde Y))$. By the above discussion, this 2-by-2 matrix is of course singular, so this provides one linear equation for the vector $r$ (with $R=(t^2 I_n - NN^\top)^{-1}$)
$$t \eta_1 \cdot Z^\top R Z 
+ \eta_2[ Z^\top R N \nu- (\xi \ell)^{-1/2}]=0.$$

By the same lemma cited above, it follows that we have the \emph{norm identity} (with $t = \sigma_1(\tilde Y)$)
\beq
\label{norm_id}
t^2 \eta_1^2 \cdot Z^\top R^2 Z 
+ \eta_2^2 \cdot \nu^\top N^\top R^2 N \nu 
+ 2 t \eta_1 \eta_2  \cdot Z^\top R^2 N \nu 
= (\xi \ell)^{-1}.
\eeq

This follows from taking the norm of the equation $t \eta_1 \cdot R Z +\eta_2  \cdot R N \nu = (\xi \ell)^{-1/2} \hat Z$ (see Lemma 5.1 in \cite{benaych2012singular}). We will find the limits of the quadratic forms below.

\begin{lem}[More quadratic forms] 
\label{qf_lem2}
The quadratic forms in the norm identity have the following properties:
\benum
\item $Z^\top (t^2 I_n - NN^\top)^{-2} Z - n^{-1}\tr (t^2 I_n - NN^\top)^{-2}\to 0$ a.s.
\item $Z^\top (t^2 I_n - NN^\top)^{-2} N\nu \to 0$ a.s.
\item $\nu^\top N^\top (t^2 I_n - N N^\top)^{-2} N \nu \to  m(t^2) + t^2 m'(t^2)$ a.s., where $m$ is the Stieltjes transform of the Marchenko-Pastur distribution $ F_{\gamma,H}$. 
\eenum
\end{lem}

The proof is in Sec.\ \ref{pf_qf_lem2}. Again, the key novelty is the proof of the third claim. The standard concentration bounds do not apply, because $u$ is non-random. Instead, we use an argument from complex analysis constructing a sequence of functions $f_n(t)$ such that their derivatives are $f_n'(t) = \nu^\top N^\top (t^2 I_n - N N^\top)^{-2} N \nu$, and deducing the convergence of $f_n'(t)$ from that of $f_n(t)$. 

Lemma \ref{qf_lem2} implies that $n^{-1}\tr (t^2 I_n - NN^\top)^{-2} \to \int (t^2-x)^{-2} d\underline F_{\gamma,H}(x) = \underline m'(t^2)$ for $t>b$.  Solving for $\eta_1$ in terms of $\eta_2$ from the first equation, plugging in to the second, and taking the limit as $n \to \infty$, we obtain that $\eta_2^2 \to c_2$, where 

$$c_2  \left(\frac{\underline m'(t^2)}{\tau\ell\underline m(t^2)^2} + m(t^2)+t^2 m'(t^2)\right) = \frac{1}{\tau\ell}.$$

Using $D(x) = x\cdot m(x)\underline m(x)$, we find $c_2 =  \underline m(t^2)/[D'(t^2)\tau \ell]$, where
$t$ solves \eqref{sv_eq}.
From the first equation, we then obtain  $\eta_1^2 \to c_1$, where $c_1 =  m(t^2)/[D'(t^2)\tau \ell]$, where
$t$ is as above.
This finishes the proof of Thm.\ \ref{spike_proj_multi} in the single-spiked case. The proof of the multispiked case is a relatively simple extension of the previous argument, so we present it in Sec.\ \ref{pf_multi}.

\subsubsection{Proof of Lemma \ref{noise_lem}}
\label{pf_noise_lem}
 Recall that $N = n^{-1/2}\mathcal{E}^*$, where $\mathcal{E}^*$ has rows $\ep^*_i =\mu\ep_i +\Sigma^{1/2} E_iX_i$. 
Note 
$$\ep^*_{ij} 
=\mu_{j} \ep_{ij} + \sigma_j E_{ij}X_{ij} 
= [\mu_{j} + \sigma_j  E_{ij}]\ep_{ij} +\sigma_j  \ell^{1/2} E_{ij}z_i u_j.$$

Since $\ep_i = \Gamma^{1/2} \alpha_i$, the terms $a_{ij} =  [\mu_{j} + \sigma_j  E_{ij}]g_j \ep_{ij}$ are independent random variables with variance $g_j^2 \E D_{ij}^2 =g_j^2(\mu_j^2+\sigma_j^2)$. Recall that we assumed that the distribution $H_p$ of $g_j^2 \E D_{ij}^2$ converges weakly to the distribution $H$. 

Hence the eigenvalue distribution of the matrix $A^\top A$, where $A = (n^{-1/2}a_{ij})_{ij}$, converges to Marchenko-Pastur distribution $F_{\gamma,H}$ \citep[][Thm.\ 4.3]{bai2009spectral}. Moreover, since $\smash{\E{\alpha_{ij}^4}<\infty}$ and $\smash{\E D_{ij}^4<\infty}$, we have $\smash{\E{a_{ij}^4}<\infty}$. In addition, by assumption $\smash{\sup g_j^2 \E D_{ij}^2 \to \sup \,\text{supp}(H)}$. Thus the largest eigenvalue of $A^\top A$ converges a.s. to the upper edge $b^2$ of the support of $F_{\gamma,H}$, see \cite{bai1998no} and \citep[][Cor.\ 6.6]{bai2009spectral}. 

Therefore, since $\sigma_j$ are bounded, it is enough to show that the operator norm of the error matrix $E^{**}$ with entries $n^{-1/2}E_{ij}z_i u_j$ converges to zero a.s. This will ensure that $N=A+E^{**}$ has the same two properties as $A$ above, namely its ESD and operator norm converge.

 Now, denoting by $\odot$ elementwise products
$$\|E^{**}\| = \sup_{\|a\|=\|c\|=1} a^\top E^{**} c = n^{-1/2} \sup_{\|a\|=\|c\|=1}  (a \odot z)^\top E (c \odot u). $$
We have $\|a \odot z\| \le \|a\|  \max |z_i| = \max |z_i|$ and $\|c \odot u\| \le \|c\|  \max |u_i|=\max |u_i|$, hence
$$\|E^{**}\| \le  \|E\| \max |z_i| \max |u_i|.$$

Since $z$ has iid standardized entries, and $C:=\E z_i^{4+\phi}<\infty$, we can derive that 
$$Pr(\max |z_i|\ge a) \le \E \max |z_i|^{4+\phi}/a^{4+\phi} \le nC/a^{4+\phi}.$$
Taking $a = n^{1/2-\phi'}$ for $\phi'$ small enough, we obtain, $\max |z_i|\le n^{1/2-\phi'}$ a.s.

 However, since $E_{ij}$ are iid standardized random variables with bounded 4-th moment, $\|E\|\to 1+\sqrt{\gamma}$ a.s. \citep{bai2009spectral}. Since $\| u\|_{\infty} \le C \log(p)^B/p^{1/2}$, we obtain $\|E^{**}\| \le C(1+\sqrt{\gamma}) \cdot  \log(p)^B n^{1/2-\phi'} p^{-1/2}\to 0 $ a.s., as required.

\subsubsection{Proof of Lemma \ref{qf_lem}}
\label{pf_qf_lem}
Since $N=A+E^{**}$, and $\|E^{**}\|\to0$ a.s., it is enough to show the same concentration statements for $A$ instead of $N$. Indeed, it is easy to see that the error terms are all negligible.

\textbf{Part 1}: 
For $Z^\top (t^2 I_n - AA^\top)^{-1} Z$, note that $Z$ has iid entries ---with mean 0 and variance $1/n$---that are independent of $A$. We will use the following result: 

\begin{lem}[Concentration of quadratic forms, consequence of Lemma B.26 in \citet{bai2009spectral}] Let $x \in \RR^k$ be a random vector with i.i.d. entries and $\EE{x} = 0$, for which $\EE{(\sqrt{k}x_i)^2} = 1$ and $\smash{\sup_i \EE{(\sqrt{k}x_i)^{4+\phi}}}$ $ < C$ for some $\phi>0$ and $C <\infty$. Moreover, let $A_k$ be a sequence of random $k \times k$ symmetric matrices independent of $x$, with a.s. uniformly bounded eigenvalues. Then the quadratic forms $x^\top A_k x $ concentrate around their means:  $\smash{x^\top A_k x - k^{-1} \tr A_k \rightarrow_{a.s.} 0}$.
\label{quad_form}
\end{lem}

 We apply this lemma with $x = Z$, $k=p$ and $A_p =  (t^2 I_n - AA^\top)^{-1}$. To get almost sure convergence, here it is required that $z_i$ have finite $4+\phi$-th moment. This shows the concentration of $Z^\top (t^2 I_n - AA^\top)^{-1} Z$. 

\textbf{Part 2}: 
To show $Z^\top (t^2 I_n - AA^\top)^{-1} A\nu$ concentrates around  0, we note that $w=(t^2 I_n - AA^\top)^{-1} A\nu$ is a random vector independent of $Z$, with a.s. bounded norm. Hence, conditional on $w$: 
\begin{align*}
Pr(|Z^\top w|\ge a|w) &\le  a^{-4} \E|Z^\top w|^4 = a^{-4} [\sum_{i} \E Z_{ni}^4w_i^4 +\sum_{i\neq j} \E Z_{ni}^2  \E Z_{nj}^2w_i^2w_j^2] \\
&\le a^{-4}  \E Z_{n1}^4 (\sum_i w_i^2)^2 = a^{-4} n^{-2} \E Z_{1}^4 \cdot \|w\|_2^4
\end{align*}
For any $C$ we can write
$$Pr(|Z^\top w|\ge a) \le Pr(|Z^\top w|\ge a|\|w\|\le C) +Pr(\|w\|> C).$$
For sufficiently large $C$, the second term, $Pr(\|w\|>C)$ is summable in $n$. By the above bound, the first term is summable for any $C$. Hence, by the Borel-Cantelli lemma, we obtain $|Z^\top w|\to 0$ a.s. This shows the required concentration.

\textbf{Part 2}: 
Finally we need to show that $\nu^\top (t^2 I_p - A^\top A)^{-1} \nu$ concentrates around  a definite value. This is probably the most interesting part, because the vector $u$ is not random. Most results for controlling expressions of the above type are designed for random $u$; however here the matrix $A$ is random instead. For this reason we will adopt a different approach.
 
 Under our assumption we have $\nu^\top(\Gamma(\Sigma+\mu^2) -zI_p)^{-1} \nu \to m_H(z)$, for $z = t^2 + i v$ with $v>0$ fixed. Therefore, Thm 1 of \cite{bai2007asymptotics} shows that $\nu^\top (z I_p - A^\top A)^{-1} \nu \to - m(z)$
a.s., where $m(z)$ is the Stieltjes transform of the Marchenko-Pastur distribution $F_{\gamma,H}$. 

A close examination of their proofs reveals that their result holds when $v \to 0$ sufficiently slowly, for instance $v = n^{-\alpha}$ for $\alpha = 1/10$.  The reason is that all bounds in the proof have the rate $N^{-k}v^{-l}$ for some small $k,l>0$, and hence they converge to 0 for $v$ of the above form. 

 For instance, the very first bounds in the proof of Thm 1 of \cite{bai2007asymptotics} are in Eq. (2.2) on page 1543. The first one states a bound of order $O(1/N^r)$. The inequalities leading up to it show that the bound is in fact  $O(1/(N^r v^{2r}))$. Similarly, the second inequality, stated with a bound of order $O(1/N^{r/2})$ is in fact  $O(1/(N^{r/2} v^{r}))$. These bounds go to zero when $v = n^{-\alpha}$ with small $\alpha>0$. In a similar way, the remaining bounds in the theorem have the same property.

To get the convergence for real $t^2$ from the convergence for complex $z = t^2 + iv$, we note that 
\begin{align*}
|\nu^\top (z I_p - A^\top A)^{-1} \nu - \nu^\top (t^2 I_p - A^\top A)^{-1} \nu| & 
= v |\nu^\top (z I_p - A^\top A)^{-1} (t^2 I_p - A^\top A)^{-1} \nu| \le \\
&\le v \|(t^2 I_p - A^\top A)^{-1}\|^2 \cdot u^\top u.
\end{align*}
As discussed above, when $t>b$, the matrices  $(t^2 I_p - A^\top A)^{-1}$ have a.s. bounded operator norm. Hence, we conclude that if $v\to0$, then a.s.
$$\nu^\top (z I_p - A^\top A)^{-1} \nu - \nu^\top (t^2 I_p - A^\top A)^{-1} \nu \to0.$$

 Finally, $m(z) \to m(t^2)$ by the continuity of the Stieltjes transform for all $t^2>0$ \citep{bai2009spectral}. We conclude that $\nu^\top (t^2 I_p - A^\top A)^{-1} \nu\to -m(t^2)$ a.s. This finishes the analysis of the last quadratic form. 

\subsubsection{Proof of Lemma \ref{qf_lem2}}
\label{pf_qf_lem2}
As in Lemma \ref{qf_lem}, it is enough to show the same concentration statements for $A$ instead of $N$.

\textbf{Parts 1 and 2}: 
The proof of Part 1 and 2 are exactly analogous to those in Lemma \ref{qf_lem}. Indeed, the same arguments work despite the change from $(t^2 I_p - N^\top N)^{-1}$ to $(t^2 I_p - N^\top N)^{-2}$, because the only properties we used are its independence from $Z$, and its a.s. bounded operator norm. These also hold for $(t^2 I_p - N^\top N)^{-2}$, so the same proof works. 

\textbf{Part 3}: 
We start with the identity $\nu^\top N^\top (t^2 I_n - N N^\top)^{-2} N \nu  = - \nu^\top (t^2 I_p - N^\top N)^{-1} \nu + t^2 \nu^\top (t^2 I_p - N^\top N)^{-2} u$. Since in Lemma \ref{qf_lem} we have already established $\nu^\top (t^2 I_p - N^\top N)^{-1} \nu \to -m(t^2)$, we only need to show the convergence of  $\nu^\top (t^2 I_p - N^\top N)^{-2} u$. 

For this we will employ the following \emph{derivative trick} \citep[see e.g.,][]{dobriban2015high}. We will construct a function with two properties: (1) its derivative is the quantity $\nu^\top (t^2 I_p - N^\top N)^{-2} u$ that we want, and (2) its limit is convenient to obtain. The following lemma will allow us to get our answer by interchanging the order of limits: 
\begin{lem}[see Lemma 2.14 in \cite{bai2009spectral}] Let $f_1, f_2,\ldots $ be analytic on a domain $D$ in the complex plane, satisfying $|f_n(z)| \le M$ for every $n$ and $z$ in $D$. Suppose that there is an analytic function $f$ on $D$ such that $f_n(z) \to f(z)$ for all $z \in D$. Then it also holds that $f_n'(z) \to f'(z)$ for all $z \in D$.
\label{holo_derivative_conv}
\end{lem}

Accordingly, consider the function $f_p(r) = - \nu^\top (r I_p - N^\top N)^{-1} \nu$. Its derivative is $f_p'(r) = \nu^\top (r I_p - N^\top N)^{-2} u$. Let $\mathcal {S} : = \{x+ iv: x>b+\ep\}$ for a sufficiently small $\ep>0$, and let us work on the set of full measure where $\|N^\top N\|<b+\ep/2$ eventually, and where $f_p(r)\to m(r)$. By inspection, $f_p$ are analytic functions on $\mathcal {S}$ bounded as $|f_p|\le 2/\ep$. Hence, by Lemma \ref{holo_derivative_conv}, $f_p'(r)\to m'(r)$.

In conclusion, $\nu^\top N^\top (t^2 I_p - N^\top N)^{-2} N \nu \to m(t^2) + t^2 m'(t^2)$, finishing the proof.

\subsubsection{Proof of Thm.\ \ref{spike_proj_multi}: Model \eqref{po_def2}}
\label{sec_po_def2}

The proof for model \eqref{po_def2} is very similar to that for model \eqref{po_def}. Therefore, we only present the outline. Working again in the single-spiked case for simplicity, we have the following decomposition for the observations $Y_i$:
\begin{align*}
Y_i &=  (\mu + \Sigma^{1/2} E_i)S_i +\ep_i \\
&=\ell^{1/2} z_i \mu u + [\mu E_i S_i +\ep_i].
\end{align*}
Denoting $\nu=\mu u/\xi^{1/2}$, where $\xi = \|\mu  u\|^2 \to \tau$, and $\ep^*_i =\mu E_i S_i +\ep_i$, we have in matrix form
\beqs
Y = (\xi \ell)^{1/2} \tilde Z \nu^\top + \mathcal{E}^*.
\eeqs
As in the proof of Lemma \ref{noise_lem}, it is not hard to see that the operator norm $\|\mathcal{E}^*-\mathcal{E}\|_{\op}\to 0$. Therefore, the spectral properties of $Y$ are equivalent to those of $\tilde Y =  (\xi \ell)^{1/2} \tilde Z \nu^\top + \mathcal{E}$. However, this is now a spiked model where the signal component is independent of the noise component. 

It follows immediately from Thm 4.3 of \cite{bai2009spectral} that the singular value distribution of $\mathcal{E} = [\alpha_1,\ldots,\alpha_n]^\top \Gamma^{1/2}$ converges to the general Marchenko-Pastur distribution $F_{\gamma,G}$, where $G$ is the limit of the distributions $G_p$ of $g_j^2$, $j=1,\ldots,p$. Similarly the top singular value of $\mathcal{E}$ converges to the upper edge $b_G$ of $F_{\gamma,G}$.

Moreover, it also follows that the analogues of Lemmas \ref{qf_lem} and \ref{qf_lem2} hold in our case. Indeed, the same arguments carry through, because the same assumptions hold. This allows the entire argument from Sec.\ \ref{setup} to carry through, finishing the single-spiked case of Thm.\ \ref{spike_proj_multi}. The extension to the multispiked case is analogous to that in model \eqref{po_def}. 

\subsubsection{Numerical computation of the quantities from Thm.\ \ref{spike_proj_multi}}
\label{spectrode} 
{\sc Spectrode} computes the Marchenko-Pastur forward map: given an input limit population spectrum $H$ and an aspect ratio $\gamma$, it outputs an accurate numerical approximation to the limit empirical spectral distribution (ESD) $F_{\gamma,H}$. \cite{dobriban2015efficient} established the numerical convergence of the method, and showed in experiments that it is much faster than previous proposals. The method is publicly available at \url{http://github.com/dobriban/eigenedge}.

The output of {\sc Spectrode} includes a numerical approximation $\hat m$ to the Stieltjes transform $m$ of the limit ESD, computed over a dense grid $x_i$ on the real line. It also includes an approximation $\hat b^2$ of the upper edge $b^2$ of the ESD. From this, we compute an approximation of the $D$-transform as $\hat D(x_i) = x_i \cdot \hat m(x_i) \cdot \underline{\hat m}(x_i)$, where $\underline{\hat m}(x_i) =  \gamma \cdot \hat m(x_i)+(\gamma-1)/x_i$. Since $D$ is monotone decreasing on $(b^2,\infty)$, we find the smallest grid point $x_i$ such that $\hat D(x_i) \le 1/\ell$ to approximately compute $D^{-1}(1/\ell)$.

Finally, the derivative $\underline m'(x)$ can be expressed as a function $\underline  m'(x)  = \mathcal{F}(\underline m(x))$ by differentiating the Marchenko-Pastur fixed-point equation \citep[see e.g.,][]{dobriban2015efficient}. Therefore, we compute a numerical approximation to $D'(x) = m(x)\cdot \underline m(x) + x[m(x)\cdot \underline m'(x) + m'(x)\cdot \underline m(x)]$ by approximating $\underline{\hat m}'(x)$ via the same function $\underline{\hat m}'(x) = \mathcal{F}(\underline{\hat m}(x))$. Similarly we approximate $\hat m'$. With these steps, we obtain a full numerical implementation of Thm.\ \ref{spike_proj_multi}.

\subsection{Proof of Thm.\ \ref{spike_proj_multi} - Multispiked extension}
\label{pf_multi}

Let us denote by $u_i=\mu u_i/\xi_i^{1/2}$ the normalized reduced signals, where $\xi_i = \|\mu  u_i\|^2 \to \tau_i$, and by $\ep^*_i = \mu\ep_i +\Sigma^{1/2} E_iX_i$. For the proof we start as in Sec.\ \ref{setup}, obtaining  $Y_i =  \sum_{k=1}^r (\xi_k \ell_k)^{1/2} z_{ik} \nu_k +\ep^*_i.$  Defining the $r\times r$ diagonal matrices $L$, $\Delta$ with diagonal entries $\ell_k$, $\xi_k$ (respectively), and the $n \times r$, $p \times r$ matrices $Z ,\mathcal{V}$, with columns $Z_k = n^{-1/2}(z_{1k}, \ldots, z_{nk})^\top$ and $u_k$ respectively, we have  
$$\tilde Y =  Z (\Delta L)^{1/2}U^\top + \mathcal{E}^*.$$
The matrix $M_n(t)$ is now $2r\times 2r$, and has the form
$$M_n(t) = 
\begin{bmatrix} 
t \cdot Z^\top (t^2 I_n - NN^\top)^{-1} Z & 
Z^\top (t^2 I_n - NN^\top)^{-1} N\mathcal{V} \\ 
\mathcal{V}^\top N^\top (t^2 I_n - NN^\top)^{-1} Z & 
t \cdot \mathcal{V}^\top (t^2 I_p - N^\top N)^{-1} \mathcal{V}
\end{bmatrix} 
-  
\begin{bmatrix} 0_r & (\Delta L)^{-1/2} \\  (\Delta L)^{-1/2} & 0_r \end{bmatrix}.$$
It is easy to see that Lemma \ref{noise_lem} still holds in this case. To find the limits of the entries of $M_n$, we need the following additional statement. 

\begin{lem}[Multispiked quadratic forms] 
\label{qf_lem3}
The quadratic forms in the multispiked case have the following properties for $t>b$:
\benum
\item $Z_k^\top R^\alpha Z_j \to 0$ a.s. for $\alpha=1,2$, if $k\neq j$.
\item $\nu_k^\top (t^2 I_p - N^\top N)^{-\alpha} \nu_j \to 0$ a.s. for $\alpha=1,2$, if $k\neq j$.
\eenum
\end{lem}
This lemma is proved in Sec.\ \ref{pf_qf_lem3}, using similar techniques as those in Lemma \ref{noise_lem}.  Defining the $r\times r$ diagonal matrices $T$ with diagonal entries $\tau_k$, we conclude that for $t>b$, $M_n(t)\to M(t)$ a.s., where now
$$M(t) = 
\begin{bmatrix} 
-t \cdot \underline m(t^2) I_r & 
-(T L)^{-1/2} \\ 
-(T L)^{-1/2} &
-t \cdot m(t^2) I_r
\end{bmatrix}.
$$
As before,  by Lemma A.1 of \cite{benaych2012singular}, we get that for $\tau_k\ell_k>1/D(b^2)$, $\sigma_k(\tilde Y) \to t_k$ a.s., where $ t_k^2 \cdot \underline m(t_k^2) m(t_k^2)  = 1/(\tau_k\ell_k)$. This finishes the spike convergence proof.
 
To obtain the limit of the angles for $\hat u_k$ for a $k$ such that $\ell_k>\tau_k D(b^2)$, consider the left singular vectors $\hat Z_k$ associated to $\sigma_k(\tilde Y)$. Define the $2r$-vector 
$$\alpha= 
\begin{bmatrix} 
 \beta_1\\
 \beta_2
\end{bmatrix}
= 
\begin{bmatrix} 
 (\Delta L)^{1/2}\mathcal{V}^\top \hat u_k \\ 
 (\Delta L)^{1/2}Z^\top \hat Z_k 
\end{bmatrix}.
$$
The vector $\alpha$ belongs to the kernel of $M_n(\sigma_k(\tilde Y))$. As argued by \cite{benaych2012singular}, the fact that the projection of $\alpha$ into the orthogonal complement of $M(t_k)$ tends to zero, implies that $\alpha_j\to 0$ for all $j\notin \{k,k+r\}$. This proves that $\nu_j^\top \hat u_k \to 0$ for $j\neq k$, and the analogous claim for the left singular vectors. 

The linear equation $M_n(\sigma_k(\tilde Y))\alpha=0$ in the $k$-th coordinate, where $k \le r$, reads (with $t=\sigma_k(\tilde Y)$):
$$ t \alpha_k Z_k^\top R Z_k  - \alpha_{r+k} (\xi_k \ell_k)^{-1/2} 
+ \sum_{i\neq k} M_n(\sigma_k(\tilde Y))_{ik} \alpha_k=0.$$
Only the first two terms are non-negligible due to the behavior of $M_n$, so we obtain $ t \alpha_k Z_k^\top R Z_k  = \alpha_{r+k} (\xi_k \ell_k)^{-1/2} +o_p(1)$. 
Moreover taking the norm of the equation $\hat Z_k = R( t Z \beta_1 + N\mathcal{V} \beta_2)$ (see Lemma 5.1 in \cite{benaych2012singular}), we get
$$t^2 \sum_{i,j \le r} \alpha_i \alpha_j  Z_i^\top R^2 Z_j 
+ \sum_{i,j \le r}  \alpha_{k+i} \alpha_{k+j}  \nu_i^\top N^\top R^2 N \nu_j 
+ \sum_{i,j \le r}  \alpha_i \alpha_{k+j}  Z_i R^2 N \nu_j =1.$$
From Lemma \ref{qf_lem3} and the discussion above, only the terms $\alpha_k^2 Z_k^\top R^2 Z_k$ and $\alpha_{r+k}^2 \nu_k^\top N^\top R^2 N \nu_k$ are non-negligible, so we obtain 
$$t^2\alpha_k^2 Z_k^\top R^2 Z_k + 
\alpha_{r+k}^2 \nu_k^\top N^\top R^2 N \nu_k =1 + o_p(1).$$
Combining the two equations above, 
$$\alpha_{r+k}^2\left[ \frac{Z_k^\top R^2 Z_k}{\xi_k \ell_k (Z_k^\top R Z_k)^2} + \nu_k^\top N^\top R^2 N \nu_k \right]=1 + o_p(1).$$
Since this is the same equation as in the single-spiked case, we can take the limit in a completely analogous way. This  finishes the proof.

\subsubsection{Proof of Lemma \ref{qf_lem3}}
\label{pf_qf_lem3}
As in Lemma \ref{qf_lem}, it is enough to show the same concentration statements for $A$ instead of $N$.

\textbf{Part 1}: The convergence $Z_k^\top R^\alpha Z_j \to 0$ a.s. for $\alpha=1,2$, if $k\neq j$, follows directly from the following well-known lemma, cited from \cite{couillet2011random}:

\begin{lem}[Proposition 4.1 in \cite{couillet2011random}]
Let $x_n \in \mathbb{R}^n $ and $y_n \in \mathbb{R}^n $ be independent sequences of random vectors, such that for each $n$ the coordinates of $x_n$ and  $y_n$ are independent random variables. Moreover, suppose that the coordinates of $x_n$ are identically distributed with mean 0, variance $C/n$ for some $C>0$ and fourth moment of order $1/n^2$. Suppose the same conditions hold for $y_n$, where the distribution of the coordinates of $y_n$ can be different from those of $x_n$. Let $A_n$ be a sequence of $n \times n$ random matrices such that $\|A_n\|$ is uniformly bounded. Then $x_n^\top A_n y_n \to_{a.s.} 0$.
\end{lem}

\textbf{Part 2}: To show $\nu_k^\top (t^2 I_p - N^\top N)^{-\alpha} \nu_j \to 0$ a.s. for $\alpha=1,2$, if $k\neq j$, the same technique cannot be used, because the vectors $u_k$ are deterministic. However, it is straightforward to check that the method of \cite{bai2007asymptotics} that we adapted in proving Part 3 of Lemma \ref{qf_lem} extends to proving $\nu_k^\top (t^2 I_p - N^\top N)^{-1} \nu_j \to 0$. Indeed, it is easy to see that all their bounds hold unchanged. In the final step, as a deterministic equivalent for $\nu_k^\top (t^2 I_p - N^\top N)^{-1} \nu_j \to 0$,  one obtains $\nu_k^\top (t^2 I_p - t^2 m(t^2) \Sigma)^{-1} \nu_j$, which tends to 0 by our assumption, showing $\nu_k^\top (t^2 I_p - N^\top N)^{-1} \nu_j \to 0$. Then $\nu_k^\top (t^2 I_p - N^\top N)^{-2} \nu_j \to 0$ follows from the derivative trick employed in Part 3 of Lemma \ref{qf_lem2}. This finishes the proof.

\subsection{Proof of Corollary \ref{standard_spike_proj}}
\label{pf_standard_spike_proj}

This corollary is a special case of our previous results, so the convergence results hold in this case. We only need to check that the limits are given by the formulas provided. Since very similar analysis has been performed by \cite{benaych2012singular} and \cite{nadakuditi2014optshrink}, we only give part of the proof. 

Under model \eqref{po_def}, since we consider the singular values of the normalized matrix $m^{-1/2} \tilde Y$ instead of $\tilde Y$ as in Thm \ref{spike_proj_multi}, it is easy to see that the relevant equation for the limiting singular values in this case is $t^2 \cdot m_0(t^2) \underline m_0(t^2) = 1/\delta \ell$ instead of $t^2 \cdot m(t^2) \underline m(t^2) = 1/\tau \ell$. 

Note that $\underline m_0(t^2) = \int (x-t^2)^{-1} d\underline F_{\gamma,H}(x) = \gamma \cdot m_0(t^2)+(\gamma-1)t^{-2}$. Thus the equation for $t^2$ reads $(\gamma t^2 m_0(t^2) + \gamma -1) m_0(t^2)= (\delta \ell)^{-1}$. However, it is well-known that $m_0(t^2)$ obeys the equation $\gamma t^2 m_0^2 + (t^2+\gamma -1) m_0+1=0$. From these two relations we obtain $t^2 \cdot m_0(t^2) = -1-1/(\delta \ell)$. Plugging this back into the second equation and simplifying, we obtain $t^2 = (\delta \ell+1)(1+\gamma/(\delta \ell))$, as required. Finally, under model \eqref{po_def2}, the proof is analogous.

 \subsection{Proofs for covariance estimation}
 \subsubsection{Proof of Cor.\ \ref{cor:s_hat_unred}}
 \label{diag_control_pf}

 From the spectral analysis of $Y$ contained in Corollary \ref{standard_spike_proj} from Section \ref{proba} and the formula \eqref{sig_hat}, the proof of this corollary is immediate from the first part of the following lemma, proved below:

 \begin{lem}
 We have the following limits (in operator norm) of the diagonals:
 \begin{enumerate}
 \item 
 $\lim_{n\rightarrow\infty} \| \diag(\hat{\Sigma}_{Y}) - m \cdot  I_p \| = 0$ a.s.

 \item 
 $\lim_{n\rightarrow\infty} \| \diag(\hat{\Sigma}_S) - I_p \| = 0$ a.s.

 \end{enumerate}
 \end{lem}

 First, note that the second statement is an immediate corollary of the first. In the proof of the first statement, for notational convenience only, we will assume that $r=1$; however, an identical proof goes through for any fixed $r$. We will therefore drop the subscript $k$ for the proof.

 Decompose $X_i$ into signal and noise, writing $X_i = S_i + \ep_i$, 
 where $S_i = \ell^{1/2}z_i u$ and $\text{Cov}(\ep) = I_p$, and $S_i$ and $\ep_i$ are independent. So if $S_i = (s_{i1},\dots,s_{ip})^\top$, $X_i = (X_{i1},\dots,X_{ip})^\top$, $Y_i = (Y_{i1},\dots,Y_{ip})^\top$, and $\ep_i = (\ep_{i1},\dots,\ep_{ip})^\top$, we can write $Y_{ij} = D_{ij} X_{ij} = D_{ij} s_{ij} + D_{ij} \ep_{ij}$
 and consequently
$
        Y_{ij}^2 = D_{ij}^2 s_{ij}^2 
                    + 2 D_{ij}^2 s_{ij} \ep_{ij}
                    + D_{ij}^2 \ep_{ij}^2 
$.

 The $(i,i)^{th}$ element of $\diag(\hat{\Sigma}_Y)$ is then
        \begin{align*}
        \frac{1}{n}\sum_{j=1}^n Y_{ij}^2 = \frac{1}{n}\sum_{j=1}^n D_{ij}^2 s_{ij}^2 
                                        + \frac{1}{n}\sum_{j=1}^n 2 D_{ij}^2 s_{ij} \ep_{ij}
                                        + \frac{1}{n}\sum_{j=1}^n D_{ij}^2 \ep_{ij}^2 
        \end{align*}
 and we will control each of the three sums on the right side separately.

 Observe that $ \mathbb{E} \bigg[ 
                   \frac{1}{n} \sum_{j=1}^n D_{ij}^2 s_{ij}^2 
                   \bigg]
        = m  \ell  u_i^2$
 and so 
        \begin{align*}
        \text{Var} \bigg( \frac{1}{n} \sum_{j=1}^n D_{ij}^2 s_{ij}^2 \bigg)
        = \frac{1}{n}(\mathbb{E}[d^4] \mathbb{E}[z^4] \ell^2 u_i^4
                            - m^2 \ell u_i^4 )
        \le c \frac{\log^{4B}(n)}{n^3}
        \end{align*}
 where $c>0$ is a constant. Chebyshev's inequality then gives
        \begin{align*}
%
        \mathbb{P}\bigg\{ 
                \bigg| \frac{1}{n} \sum_{j=1}^n D_{ij}^2 s_{ij}^2 
                -  m \cdot \ell^2 \cdot u_i^2 \bigg| 
                \ge \ep \text{ for some } i \bigg\}
        \le c\frac{p \log^{4B}(n)}{\ep^2 n^3}.
        \end{align*}
 Since $u_i^2 \le C \log^B(p)/p$, this shows that $n^{-1} \sum_{j=1}^n D_{ij}^2 s_{ij}^2$ converges a.s. to 0 as $p,n \rightarrow \infty$ and $p/n \rightarrow \gamma$.

 For the sum of terms $D_{ij}^2 \ep_{ij}^2$, observe that $\mathbb{E} n^{-1} \sum_{j=1}^n 2 D_{ij}^2 s_{ij} \ep_{ij} = 0$, and 
        \begin{align*}
        \mathbb{E}| 2 D_{ij}^2 s_{ij} \ep_{ij}|^4
                = 16\mathbb{E}(d^8) \ell^2 u_i^4
                \le c \frac{\log^{4B}(n)}{n^2}
        \end{align*}
 where $c > 0$ is a constant and we have used the estimate $|u_i| \le C\log^{B}(p)/\sqrt{p}$. By using Markov's inequality for the fourth moment (see \cite{petrov2012sums}), we then obtain:
        \begin{align*}
%
        \mathbb{P}\bigg\{
            \bigg| \frac{1}{n} \sum_{j=1}^n 2 D_{ij}^2 s_{ij} \ep_{ij}\bigg| 
                \ge \ep \text{ for some } i \bigg\} 
        \le c \frac{\log^{4B}(n)}{\ep^4 n^2}.
        \end{align*}
 This proves that $n^{-1} \sum_{j=1}^n 2 D_{ij}^2 s_{ij} \ep_{ij}$ converges a.s. to 0 as $p,n \rightarrow \infty$ and $p/n \to \gamma$.

 Finally, to deal with $n^{-1}\sum_{j=1}^n D_{ij}^2 \ep_{ij}^2$, observe that $\mathbb{E} D_{ij}^2 \ep_{ij}^2 = m$, and that  $\mathbb{E} |D_{ij}^2 \ep_{ij}^2|^4 = \mathbb{E}d^8 \mathbb{E}{\ep^8}$
 which is assumed finite. Therefore using again Markov's inequality for the fourth moment,
        \begin{align*}
%
        \mathbb{P} \bigg\{ 
            \bigg| \frac{1}{n} \sum_{j=1}^n D_{ij}^2 \ep_{ij}^2 - m \bigg| 
                \ge \ep \text{ for some } i  \bigg\}
        \le c p \frac{ \mathbb{E}d^8 \mathbb{E}{\ep^8}}{\ep^4 n^4}.
        \end{align*}
 This proves that $n^{-1} \sum_{j=1}^n D_{ij}^2 \ep_{ij}^2$ converges a.s. to 0 as $p,n \to \infty$ and $p/n \to \gamma$. This finishes the proof.

 \subsubsection{Proof of \eqref{loss_infty}}
 \label{proof_eta_2d}
 The following analysis is adapted from that in \cite{donoho2013optimal}. We start with a fact from linear algebra.

 \begin{lem} 
 \label{simu_diag}
 Suppose $A$ and $B$ are two $p$-by-$p$ symmetric matrices, with eigenvectors $u_1,\dots,u_p$ and $v_1,\dots,v_p$. Suppose that each one has eigenvalue $0$ with multiplicity $p-r$ for some $1 \le r < p/2$, with corresponding eigenvectors $u_{r+1},\dots,u_p$ and $v_{r+1},\dots,v_p$. Suppose that no eigenvector $v_1,\dots,v_r$ lies in the span of $u_1,\dots,u_r$. Then there is an orthogonal matrix $W$ such that 
        \begin{align}
        \label{aa77}
        W^\top A W = \diag(\ell_1,\dots,\ell_r) \oplus \mathbf{0}_{(p-r)\times(p-r)} 
        \end{align}
 and 
        \begin{align}
        \label{bb77}
        W^\top B W = \tilde{B} \oplus \mathbf{0}_{(p-2r)\times(p-2r)}
        \end{align}
 where $\tilde{B}$ is a $2r$-by-$2r$ matrix whose entries are continuous functions of the inner products $u_i^\top v_j$, $1\le i,j \le r$, on $\mathbb{R}^{2r} \setminus \{\pm1\}^r$.

 Furthermore, if we assume in addition that $u_i^\top v_j = c_i \delta_{i=j}$ for some numbers $c_i \in (-1,1)$, and we let $s_i = \sqrt{1-c_i^2}$, then the matrix $\tilde{B}$ is of the form:
        \begin{align*}
        \left[
        \begin{array}{c c}
        \diag(\eta_1 c_1^2,\dots,\eta_r c_r^2) &
                \diag(\eta_1 c_1s_1,\dots,\eta_r c_rs_r) \\
        \diag(\eta_1 c_1s_1,\dots,\eta_r c_rs_r) &
                \diag(\eta_1 s_1^2,\dots,\eta_r s_r^2)
        \end{array}
        \right]
        \end{align*}
 \end{lem}

 \begin{proof}
 Form vectors $\tilde{v}_1,\dots,\tilde{v}_{p-2r}$, by performing Gram-Schmidt orthogonalization of $v_1,\dots,\tilde{v}_{p-2r}$ against the vectors $u_1,\dots,u_r$. If the columns of $W$ are the vectors in this basis, then clearly \eqref{aa77} holds, since $A u_k = \ell_k u_k $ for $k=1,\dots,r$, and $A \tilde{v}_k = 0$ for $k=1,\dots,p-2r$. Furthermore, since $B \tilde{v}_{k} = 0$ for $k=r+1,\dots,p-2r$, \eqref{bb77} holds for some $2r$-by-$2r$ block $\tilde{B}$. We need only check that the entries of $\tilde{B}$ are continuous functions of $u_i^\top v_j$, $1\le i,j \le r$, on $\mathbb{R}^{2r} \setminus \{\pm1\}^r$.

 Let $W_{2r}$ denote the $p$-by-$2r$ matrix containing the first $2r$ columns of $W$, $U_r$ denote the matrix with columns $u_1,\dots,u_r$, $V_r$ denote the matrix with columns $v_1,\dots,v_r$, and $\tilde{V}_r$ the matrix with columns $\tilde{v}_1,\dots,\tilde{v}_r$. We can therefore write $W_{2r} = [U_r \,\,\, \tilde{V}_r]$;

 By the Gram-Schmidt construction, every vector $v_k$, $k=1,\dots,r$, can be written as a linear combination of the vectors $u_1,\dots,u_r,\tilde{v}_1,\dots,\tilde{v}_k$; in matrix form, these linear combinations can be expressed as:
        \begin{align*}
        V_r = U_r (U_r^\top V_r) + \tilde{V}_r (\tilde{V}^\top V_r)
        \end{align*}

 Now, we observe that the $2r$-by-$2r$ block $\tilde{B}$ is
        \begin{align*}
        W_{2r}^\top B W_{2r} = 
        \begin{bmatrix} 
 U_r^\top V \\ 
 \tilde{V}_r^\top V
\end{bmatrix}
         \diag(\eta_1,\dots,\eta_r)
         [U_r^\top V \,\,\, \tilde{V}_r^\top V]
        \end{align*}

 We need only show that the entries of $\tilde{V}_r^\top V_r$ are continuous functions of the inner products $u_i^\top v_j$, $1 \le i,j \le r$. Let $R = \tilde{V}_r^\top V$. Then $R_{k,l} = 0$ whenever $k > l$. When $k=l$, we have:
        \begin{align*}
        R_{k,k} = \tilde{v}_k^\top v_k
                = \sqrt{1 - \sum_{i=1}^r (u_i^\top v_k)^2 
                          - \sum_{i=1}^{k-1} (\tilde{v}_i^\top v_k)^2 }
        \end{align*}
 and when $k < l$ we have
        \begin{align*}
        R_{k,l} = \frac{-\sum_{i=1}^r (u_i^\top v_l)(u_i^\top v_k)
                        + \sum_{i=1}^{k-1}(\tilde{v}_i^\top v_l)(\tilde{v}_i^\top v_k)}
                       {\| v_k - \sum_{i=1}^r (u_i^\top v_k) u_k\|}.
        \end{align*}
Since for any $l=2,\dots,r$, we have
        \begin{align*}
        R_{1,l} = \tilde{v}_1^\top v_l
                = \frac{- \sum_{i=1}^r (u_i^\top v_1)(u_i^\top v_l)}
                       {\|v_1 - \sum_{i=1}^r (u_i^\top v_1)u_i\|}
        \end{align*}
 which is obviously a continuous function of the inner products $v_1^\top u_i$ on $\mathbb{R}^{r}\setminus\{\pm1\}^r$, an induction argument easily shows all entries of $R$ are continuous in the inner products $u_k^\top v_l$.

 Finally, if $u_i^\top v_j = c_i \delta_{i=j}$, the final assertion follows immediately, finishing the proof.
 \end{proof}

If we apply the permutation $\pi_{2r} = (1,r+1,2,r+2,\dots,r,2r)$ to the rows and columns of the block matrix $\tilde{B}$ from Lemma \ref{simu_diag}, the corresponding matrix becomes
$
        \bigoplus_{i=1}^r B_2(\eta_i,c_i,s_i)
$
 where 
        \begin{align} \label{B_mx}
        B_2(\eta,c,s) = 
        \left(
        \begin{array}{c c}
         \eta c^2 & \eta  c s     \\
         \eta c s & \eta s^2        
        \end{array}
        \right).
        \end{align}
 Applying the same permutation to the top $2r$-by-$2r$ block of the matrix $A$ from Lemma \ref{simu_diag} turns this block into:
$   \bigoplus_{i=1}^r A_2(\ell_i)
$
 where 
$
        A_2(\ell) = 
        \left(
        \begin{array}{c c}
         \ell & 0     \\
         0 & 0        
        \end{array}
        \right).$

 Now, let $\tilde{B}$ be given by the formula \eqref{B_mx}, but where $c = c_i = c(\delta \ell_i)$ is the asymptotic inner product between the population eigenvector $u_i$ and the empirical eigenvector $\hat{u}_i$, and $\eta = \eta_i = \eta(\lambda_i)$ is the almost sure limit of the $i^{th}$ eigenvalue $\eta(\hat\lambda_i)$ of $\hat{\Sigma}^\eta_S$. Since the shrinkers $\eta$ collapse the vicinity of the bulk to 0, it follows from Cor.~\ref{cor:s_hat} that the shrunken estimator has rank at most $r$ a.s. Hence, by Lemma \ref{simu_diag} and Cor.~\ref{cor:s_hat}, if $L_p(A,B)$ is any orthogonally-invariant loss function that decomposes over blocks, we have the almost sure convergence of $\smash{L_p(\Sigma_S,\hat{\Sigma}^\eta_S)}$ to the quantity
        \begin{align*}
%
        L_\infty(\Sigma_S,\hat{\Sigma}^\eta_S)
        = L_{2r}\bigg( \bigoplus_{k=1}^r A_2(\ell_k),
                       \bigoplus_{k=1}^r B_2(\eta(\lambda_{k}),c_{k},s_{k})
                \bigg).
        \end{align*}
 This is the desired result.

\subsubsection{Proof of Props.\ \ref{rel_diff_prop}}
\label{rel_diff_prop_pf}
 We will only provide a proof for the reduced-noise model; the proof for the unreduced-model is even simpler. Define the operator $\mathcal{L}_p$ acting on $p$-by-$p$ matrices by
        \begin{math}
        \mathcal{L}_p(\Sigma) = \mu^2 \Sigma + \sigma^2 \diag(\Sigma)
        \end{math}
 and define the operator $L_p$ by 
        \begin{math}
        L_p(\Sigma) = \frac{1}{n} \sum_{k=1}^n D_k \Sigma D_k.
        \end{math}
 Also, define define the $p$-by-$p$ matrices $B_p = \hat{\Sigma}_Y - \frac{1}{n}\sum_{k=1}^n D_k^2$  and $\mathcal{B}_p = \hat{\Sigma}_Y - mI_p$.

 Let $\Delta B_p = \mathcal{B}_p - B_p$, and $\Delta L_p = \mathcal{L}_p - L_p$. The operator $\Delta L_p$ taking $p$-by-$p$ matrices to $p$-by-$p$ matrices is diagonal. In this proof, when we refer to the $(i,j)^{th}$ diagonal entry of a diagonal operator on $\mathbb{R}^{p \times p}$, we mean the entry that multiplies the $(i,j)^{th}$ coordinate of a matrix. When $i \ne j$, the $(i,j)^{th}$ diagonal entry of $\Delta L_p(\Sigma)$ is $\mu^2 - n^{-1}\sum_{k=1}^nD_{i,k} D_{j,k}\Sigma_{ij}$; and when $i=j$, the $(i,i)^{th}$ diagonal entry of $\Delta L_p$ is $m - n^{-1}\sum_{k=1}^nD_{i,k}^2$.

 A proof nearly identical to the proof of Cor.\ \ref{cor:s_hat} shows that the maximum of all the $p^2$ diagonal entries of $\Delta L_p$ converges almost surely to 0; in other words, if $\mathbb{R}^{p\times p}$ is equipped with Frobenius norm, then the operator norm of $\Delta L_p : \mathbb{R}^{p \times p} \to \mathbb{R}^{p \times p}$ converges to 0 almost surely.

 The $p$-by-$p$ matrix $\Delta B_p$ is diagonal, with $i^{th}$ diagonal entry equal to $m - n^{-1}\sum_{k=1}^n D_{k,i}$. Again, a proof like the proof of Cor.\ \ref{cor:s_hat} shows that the supremum of these elements converges to zero almost surely as $n,p\to\infty$.

 Also, note that the Frobenius norm of the matrix $\hat{\Sigma}_S$, which is the sum of squares of its eigenvalues, is of size $\approx\sqrt{n}$. To see this, observe that
        \begin{align*}
        \frac{1}{n}\|\hat{\Sigma}_S\|_F^2 
        = \frac{1}{n}\sum_{k=1}^{p} \sigma_k(\hat{\Sigma}_S)^2
        = \frac{1}{n}\sum_{k=1}^{p-r} \sigma_k(\hat{\Sigma}_S)^2
            + \frac{1}{n}\sum_{k=p-r+1}^{p} \sigma_k(\hat{\Sigma}_S)^2.
        \end{align*}
 The first term $\frac{1}{n}\sum_{k=1}^{p-r} \sigma_k(\hat{\Sigma}_S)^2$ converges almost surely to the second moment of the Marchenko-Pastur law, which is finite since the distribution has finite support. The second term $\frac{1}{n}\sum_{k=p-r+1}^{p} \sigma_k(\hat{\Sigma}_S)^2$ converges to 0.

 Observe that 
        \begin{align*}
        L_p(\hat{\Sigma}_S) - \Delta L_p(\hat{\Sigma}_S)
        = \mathcal{L}_p(\hat{\Sigma}_S) 
        = \mathcal{B}_p 
        = B_p + \Delta B_p
        = L_p(\hat{\Sigma}_S^\prime) + \Delta B_p
        \end{align*}
 or in other words,
        \begin{align*}
        \frac{\hat{\Sigma}_S - \hat{\Sigma}_S^\prime}{\|\hat{\Sigma}_S\|} 
        = \frac{L_p^{-1}(\Delta L_p\hat{\Sigma}_S)}{\|\hat{\Sigma}_S\|} 
            - \frac{L_p^{-1}(\Delta B_p)}{\|\hat{\Sigma}_S\|}.
        \end{align*}
 The result follows immediately.
%

%
%
%

\subsection{Proof that BLP is asymptotically diagonalized in PC basis (Sec. \ref{sec:denoise_setup})}
\label{sv_shr_equiv}

First, it is easy to check that $\Cov{S_i,Y_i} = \mu \sum_{k=1}^r\ell_k  u_k u_{k}^\top$, and  $M = \Cov{Y_i,Y_i}^{-1} = [\mu^2\sum_{j=1}^r\ell_j  u_{j}  u_{j}^\top+mI_p +E]^{-1}$, where $E = \sigma^2 \sum_{j=1}^r\ell_j  \diag (u_{j}\odot u_{j})$. Therefore the BLP equals  $\hat S_i^{BLP}=\mu \sum_{k=1}^r\ell_k  u_k u_{k}^\top M Y_{i}$. Moreover, since $u_j$ are delocalized, the operator norm $\|E\| \to 0$, so that it is easy to check that $\|\hat S_i^{BLP}-\hat S_i^0\|^2 \to 0$, where 
\[
\hat S_i^0 = \mu \sum_{k=1}^r\ell_k u_k u_{k}^\top M_0 Y_{i},
\]
and $M_0 = [\mu^2\sum_{j=1}^r\ell_j  u_{j}  u_{j}^\top+mI_p]^{-1}$. Therefore, it is enough to show that $\| \hat S_i^{0} - \hat S_i \|^2  \to 0$, where recall that $\hat S_i = \sum_{k=1}^{r} \mu\ell_k/(\mu^2 \ell_k + m) u_k u_k^\top Y_{i}$. We can write, with $m_k = u_k^\top(M_0 - I_p/(\mu^2 \ell_k+m)) Y_i$.
\begin{align*}
\| \hat S_i^{0} - \hat S_i \|^2  
&=  \| \mu \sum_{k=1}^r\ell_k u_k u_{k}^\top M_0 Y_{i}
-  \sum_{k=1}^{r} \mu\ell_k/(\mu^2 \ell_k + m) u_k u_k^\top Y_{i}\|^2\\
& =  \| \mu \sum_{k=1}^r\ell_k  u_k m_k  \|^2 
 = \mu^2 \sum_{k,j=1}^r \ell_k \ell_j m_k m_j u_k^\top u_j.
\end{align*}
Therefore, to show $\| \hat S_i^{0} - \hat S_i \|^2  \to0$, it is enough to show $m_k\to 0$. For this, using the formula $u^\top (uu^\top +T)^{-1} = u^\top T^{-1}/(1+ u^\top T^{-1}u)$, we can write
$$
u_{k}^\top \left[\mu^2\sum_{j=1}^r\ell_j  u_{j}  u_{j}^\top+mI_p \right]^{-1}
=u_{k}^\top \left[\mu^2\sum_{j\neq k}^r\ell_j  u_{j}  u_{j}^\top+mI_p \right]^{-1}
/\left(1+ \mu^2 u_{k}^\top \left[\mu^2\sum_{j\neq k}^r\ell_j  u_{j}  u_{j}^\top+mI_p \right]^{-1}u_{k}\right).
$$

Under the assumptions of Cor.\ \ref{standard_spike_proj}, we have $u_k^\top u_j \to \delta_{kj}$, hence it is easy to see that the denominator converges a.s. to $1+\mu^2/m$. Indeed by using the  formula $(VV^\top +mI)^{-1} =[I- V (V^\top V +m I)^{-1}V^\top]/m$, for $V=\mu [\ell_1^{1/2} u_{1}, \ldots, \ell_r^{1/2} u_{r}]$ (excluding $u_{k}$)
$$
u_{k}^\top \left[\mu^2\sum_{j\neq k}^r\ell_j  u_{j}  u_{j}^\top+mI \right]^{-1}u_{k} 
=  u_{k}^\top u_{k}/m -  u_{k}^\top V (V^\top V +m I)^{-1}V^\top u_{k}/m.
$$
Now the entries of the $r-1$-dimensional vector $v_k=V^\top u_{k}$ are  $\mu \ell^{1/2}u_{j}^\top u_{k}$ for $j\neq k$. Therefore, they converge to zero a.s. Since the operator norm of $(V^\top V +mI)^{-1}$ is bounded above by $1/m$, this shows that the second term converges to zero a.s. This shows that the denominator converges to $1+\mu^2/m$ a.s. Finally, by using the  formula $(VV^\top +mI)^{-1} = [I- V (V^\top V +mI)^{-1}V^\top]/m$ once again, we conclude similarly that 
$$
u_{k}^\top \left[\mu^2\sum_{j\neq k}^r\ell_j  u_{j}  u_{j}^\top+m I \right]^{-1}Y_{i} 
=  u_{k}^\top Y_{i}/m  -  u_{k}^\top V (V^\top V +mI)^{-1}V^\top Y_{i}/m.
$$

Denoting $M_k = \left[\mu^2\sum_{j\neq k}^r\ell_j  u_{j}  u_{j}^\top+I \right]^{-1}$, we have
$$m_k 
= \frac{u_{k}^\top M_k Y_{i}}{1+u_{k}^\top M_ku_{k}}  -  \frac{u_{k}^\top Y_{i}/m}{1+\mu^2/m \cdot \ell_k}
= \frac{u_{k}^\top Y_{i}}{m} \left(\frac{1}{1+u_{k}^\top M_ku_{k}}-\frac{1}{1+\mu^2/m \cdot\ell_k}\right)
- \frac{v_k^\top (V^\top V +mI)^{-1}v_k/m}{1+u_{k}^\top M_ku_{k}}.
$$
Based on our previous calculations, both terms tend to 0, implying $m_k\to 0$. Therefore, we have $\| \hat S_i^{BLP} - \hat S_i \|^2  \to0$.  By the triangle inequality, this immediately implies that all MSE properties for $ \hat S_i^{BLP}$ also hold for $ \hat S_i$. This proves the desired claim.

\subsection{Proof of Thm.\ \ref{den_thm}}
\label{den_pfs}

\subsubsection{BLP}
\label{blp_single_spike}

For the BLP, consider first the single-spiked case where the squared prediction error of $\hat S_i^{\tau} = \hat S_i^{\tau,B} = \tau u u^\top Y_i$, with $Y_i = D_i(S_i +\ep_i) = D_i(\ell^{1/2}z_i u +\ep_i)$ is 
\begin{align*}
\|S_i -\hat S_i^{\tau}\|^2 &
= \|\ell^{1/2}z_i u -  \tau u^\top Y_i u \|^2 
=  (\ell^{1/2}z_i -  \tau u^\top Y_i)^2  = [\ell^{1/2}z_i -  \tau( \ell^{1/2}z_i u^\top D_i u + u^\top D_i \ep_{i})]^2\\
&= (1-\tau u^\top D_i u)^2\ell z_i^2 + \tau^2 (u^\top D_i \ep_{i})^2 -  2\tau (1-\tau u^\top D_i \ep_{i}) \ell^{1/2} z_i u^\top D_i \ep_{i}.
\end{align*}

Now, $z_i$ and $\ep_i$ are independent of $D_i$, hence, conditional on $D_i$ we have 
$$\E [\|S_i -\hat S_i^{\tau}\|^2|D_i] = (1-\tau u^\top D_i u)^2\ell + \tau^2  u^\top D_i^2 u .$$
Moreover, $\E u^\top D_i u = \mu$, $\E (u^\top D_i u)^2 = \mu^2 +\sigma^2 \|u\|_4^4 \to \mu^2$---since $u$ is delocalized---and $\E u^\top D_i^2 u = m$, so the overall expectation converges to
$
\E \|S_i -\hat S_i^{\tau}\|^2 \to (1-2\tau \mu+\tau^2 \mu^2)\ell + \tau^2  m$.

This proves that the asymptotic MSE of BLP is $AMSE^{B}(\tau;\ell, \gamma) = (1-\tau\mu)^2\ell + \tau^2 m$. The minimum is achieved for $\tau^* = \mu \ell/(\mu^2 \ell + m)$, and equals  $m\ell/(\mu^2 \ell+m)$. 

In the multispiked case, the prediction error is 
\begin{align*}
\|S_i -\hat S_i^{\tau}\|^2 &
= \|\sum_{k=1}^r \ell_k^{1/2}z_{ik} u_k -  \tau_k u_{k}^\top Y_{i} u_k\|^2 \\
&= \sum_{k=1}^r (\ell_k^{1/2}z_{ik}-  \tau_k u_{k}^\top Y_{i})^2 
+\sum_{k\neq j}
 (\ell_k^{1/2}z_{ik}-  \tau_k u_{k}^\top Y_{i})
  (\ell_j^{1/2}z_{ij}-  \tau_j u_{j}^\top Y_{i})
  \cdot u_k^\top u_j
\end{align*}
Now, $z_{ik},z_{ij}$ are independent if $k\neq j$, hence we have if $k\neq j$
$$\E (\ell_k^{1/2}z_{ik}-  \tau_k u_{k}^\top Y_{i})
  (\ell_j^{1/2}z_{ij}-  \tau_j u_{j}^\top Y_{i}) = 
  \tau_k \tau_j \E u_{k}^\top Y_{i} u_{j}^\top Y_{i}.
  $$
Using $Y_i = D_i X_i$, $X_i = \sum_{k=1}^r\ell_k^{1/2} z_{ik}  u_k + \ep_i$, and the independence properties described above, we have  
\begin{align*}
&\E u_{k}^\top Y_{i} u_{j}^\top Y_{i}  
= \E u_{k}^\top D_i X_i X_i^\top D_i u_j 
= \E u_{k}^\top D_i (\sum_{m=1}^r\ell_m u_{m}u_{m}^\top + I_p)  D_i u_j\\
 &=\sum_{m=1}^r\ell_m \cdot \E (u_{k}^\top D_i u_{m} \cdot  u_{j}^\top D_i u_{m})  
 + \E u_{k}^\top D_i^2 u_{j} \\
  &=\sum_{m=1}^r\ell_m \cdot [\mu^2 (u_{k}^\top u_{m}) \cdot (u_{j}^\top u_{m}) +\sigma^2   (u_{k} \odot u_{m})^\top  \cdot (u_{j} \odot u_{m})]
 + m u_{k}^\top u_{j}.
\end{align*}
Above we denoted by $c  = a\odot b$ the vector with entries $c_j = a_j b_j$. Since the vectors $u_k$ are delocalized and we assumed $u_k^\top u_l \to 0$ for $k\neq l$, it is easy to see that the entire expression converges to zero. This shows that 
\[
\E \|S_i -\hat S_i^{\tau}\|^2 - \E \sum_{k=1}^r (\ell_k^{1/2}z_{ik}-  \tau_k u_{k}^\top Y_{i})^2 \to 0.
\]
Therefore, the asymptotic MSE decouples over the different PCs. Therefore, we can use our previous results about the asymptotic MSE in the single-spiked model, as soon as we can show that the same formulas for the MSE of specific coordinates given in Sec.\ \ref{blp_single_spike} hold in this setting. However, this is easy to see using a calculation similar to the one given above. Indeed, we have
\begin{align*}
\E (\ell_k^{1/2}z_{ik}-  \tau_k u_{k}^\top Y_{i})^2 
&=\E (\ell_k^{1/2}(1 - \tau_k   u_{k}^\top D_i u_{k}) z_{ik}-   \sum_{m\neq k}\tau_k \ell_m^{1/2} z_{im}  u_{k}^\top D_i u_{m} - \tau_k u_{k}^\top D_i \ep_{i})^2 \\
&= \ell_k \E (1 - \tau_k u_{k}^\top D_i u_k)^2 +  
\tau_k^2\sum_{m\neq k}\ell_m \E (u_{k}^\top D_i u_{m})^2 + \tau_k^2 \E u_{k}^\top D_i^2 u_k.
\end{align*}
For $k \neq m$, the term $\E (u_{k}^\top D_i u_{m})^2 = \mu^2 (u_{k}^\top u_{m})^2 +\sigma^2 \|u_k\odot u_m\|^2 \to 0$, while the other terms can be evaluated as before, leading to the same formulas. Therefore, the asymptotic limit of the MSE is the same as that in the single-spiked case. This finishes the claims about BLP.

\subsubsection{Empirical BLP}
The squared error of a general EBLP denoiser is
\begin{align*}
\|S_i -\hat S_i^{\eta}\|^2 &
= \|\sum_{k=1}^r \ell_k^{1/2}z_{ik} u_k  
- \sum_{k=1}^r \eta_k \hat u_k \hat u_{k}^\top Y_{i}\|^2 \\
&= \sum_{k=1}^r \|\ell_k^{1/2}z_{ik} u_k -\eta_k \hat u_k \hat u_{k}^\top Y_{i}\|^2 
+\sum_{k\neq j}
 (\ell_k^{1/2}z_{ik} u_k -\eta_k \hat u_k \hat u_{k}^\top Y_{i})^\top
  (\ell_j^{1/2}z_{ij} u_j -\eta_j \hat u_j \hat u_{j}^\top Y_{i})
\end{align*}
The following lemma, proved in Sec.\ \ref{denoise_pf_cross}, shows that the cross terms vanish:

\begin{lem}[Vanishing cross term in MSE]
\label{denoise_lim_multi}
The cross terms in the MSE of EBLP denoisers vanish, i.e., for all $k\neq j$,  
$\E(\ell_k^{1/2}z_{ik} u_k -\eta_k \hat u_k \hat u_{k}^\top Y_{i})^\top
  (\ell_j^{1/2}z_{ij} u_j -\eta_j \hat u_j \hat u_{j}^\top Y_{i}) \to 0$
\end{lem}

Therefore, the limit of the MSE in the multispiked case decouples into the MSEs for the individual spikes:
$$
\E\|S_i -\hat S_i^{\eta}\|^2 - \sum_{k=1}^r \E \|\ell_k^{1/2}z_{ik} u_k -\eta_k \hat u_k \hat u_{k}^\top Y_{i}\|^2 \to 0.
$$
To evaluate the limiting MSE, we rely on the following lemma (proved in Sec.\ \ref{denoise_pf}), which finds limiting expectations of inner products of the empirical singular vector $\hat u_B$ with the samples $Y_i$ and the population singular vector $u$. 

\begin{lem}[Denoising risk limit]
\label{denoise_lim}
We have the following convergences: 
\benum
\item $ \E  (\hat u_{k}^\top Y_{i})^2   \to m \lambda_k^2$, where $\lambda_k^2 = t^2(\delta \ell_k;\gamma)$ is defined in Eq.  \eqref{spike_eq_white}.
\item $ \E z_i (\hat u_{k}^\top Y_{i}) (u_k^\top \hat u_k)\to  \mu \ell_k^{1/2} c_k^2 \cdot \beta_k$, where $c_k^2 = c^2(\delta\ell_k;\gamma)$ is defined in Eq. \eqref{cos_eq_white}, and
\beq
\label{beta_eq}
\beta = 1 +\frac{\gamma}{\delta\ell_k}.
\eeq

\eenum
\end{lem}

The key technical innovation is the proof of the second part. The main argument is an extension of the technique introduced by \cite{benaych2012singular} for characterizing the limits of the inner products $u_j^\top \hat u_k$ between population and sample eigenvectors. For our proof, we need to extend this technique to characterize limits $\smash{w^\top \hat u_k}$ for \emph{arbitrary random vectors} $w$. Since this technique relies on an equation for the ``outliers'' among the eigenvalues of a finite-rank perturbation of a matrix in the terminology commonly used in random matrix theory \citep[see e.g.,][]{tao2013outliers}, we call this the \emph{outlier equation} method. 

Applying the outlier equation method is nontrivial in our case, because the random vectors $w$ to which we need to apply it---e.g., $X_i$---are dependent with $\hat u_k$. For this reason, we need to use rank-one perturbation formulas once again to show that the dependence is negligible. We envision that the proof method could have several other applications.

Going back to our main argument, based on Lemma \ref{denoise_lim}, the limit of the MSE is the following deterministic quantity:
$AMSE = \sum_{k=1}^r \ell_k +\eta_k^2 \cdot m t_k^2 - 2\eta_k \cdot \mu  \ell_k c_k^2 \cdot \beta_k.$
The optimal $\eta$ minimizing the AMSE has $\eta_k^* = \mu \ell_k  c_k^2 \cdot \beta_k/[m t_k^2]$ $= \mu \ell_k  c_k^2/[m (1+\delta \ell_k)]$. This finishes the EBLP analysis.

\subsubsection{Proof of Lemma \ref{denoise_lim_multi}}
\label{denoise_pf_cross}
We need to show that 
$\E(\ell_k^{1/2}z_{ik} u_k -\eta_k \hat u_k \hat u_{k}^\top Y_{i})^\top
  (\ell_j^{1/2}z_{ij} u_j -\eta_j \hat u_j \hat u_{j}^\top Y_{i}) \to 0.$
We expand the parantheses and note that the first term is $\E (\ell_k^{1/2}z_{ik} u_k)^\top
  \ell_j^{1/2}z_{ij} u_j = 0$, while the last term is a multiple of $\E \hat u_k ^\top \hat u_j \cdot \hat u_{k}^\top Y_{i} \cdot \hat u_{j}^\top Y_{i}=0$ (because $\hat u_k^\top \hat u_j = 0$ for $k \neq j$). Thus it is enough to show the following claim for $k \neq j$:
$$\E z_{ij} \hat u_k ^\top u_j \cdot \hat u_{k}^\top Y_{i} \to 0$$
For this, we have $u_k^\top  \hat u_j \to 0$ a.s., by Cor.\ \ref{standard_spike_proj}, thus also  $z_{ij} u_k^\top  \hat u_j \to 0$, so  by convergence reduction (Lemma \ref{conv_reduce}) it is enough to show that $\smash{\E (\hat u_{k}^\top Y_{i})^2}$ is uniformly bounded. However, as in the proof of Part 1 of Lemma \ref{denoise_lim} in Sec.\ \ref{denoise_lim_p2}, we obtain that $\smash{\E (\hat u_{k}^\top Y_{i})^2}$ is uniformly bounded. This finishes the proof of Lemma \ref{denoise_lim_multi}.

\subsection{Proof of Lemma \ref{denoise_lim}}
\label{denoise_pf}
For simplicity of exposition, we first prove the single-spiked case, when $r=1$. The extension to the multispiked case is presented in Sec.\ \ref{denoise_pf_multi}.
\subsubsection{Part 1}
\label{denoise_lim_p2}

Similar to \cite{lee2010convergence} in the unreduced case, we use the following exchangeability argument:
$$
\E  (\hat u^\top Y_i)^2  
=  \E  \hat u^\top Y_i  Y_i^\top \hat u 
= n^{-1} \sum_{i=1}^{n}  \E  \hat u^\top Y_i  Y_i^\top \hat u 
= n^{-1} \E  \hat u^\top Y^\top  Y \hat u   
= \E \sigma_1(\tilde Y)^2.$$
Now $\sigma_1(\tilde Y)^2 \to m \cdot \lambda(\delta \ell;\gamma)$ a.s. by Cor.\ \ref{standard_spike_proj}. Moreover, by a bound on the expectation of top eigenvalues of sample covariance matrices such as that in \cite{srivastava2013covariance}, it is easy to see that $\E \sigma_1(\tilde Y)^2$ is uniformly bounded. Hence, $\E\sigma_1(\tilde Y)^2 \to m \cdot \lambda(\delta \ell;\gamma)$, finishing the proof. 

\subsubsection{Part 2}

As a preliminary remark for Part 2, denoting $t^2 = \lambda(\delta\ell;\gamma)$ we note that $\beta = 1 + \gamma/(\delta\ell) = 1 -m \cdot \gamma m(t^2)/[1+m \cdot \gamma m(t^2)]$. By a simple calculation, this is equivalent to $m \cdot m(t^2) = -1/(\gamma + \delta \ell)$. This can be checked as in the proof of Cor.\ \ref{standard_spike_proj} in Sec.\ \ref{pf_standard_spike_proj}.

Therefore it is enough to show that 
\[ \E z_i (\hat u^\top Y_{i}) (u^\top \hat u)
\to   \ell^{1/2} c^2\left[\mu +\frac{-m \cdot \gamma m(t^2)}{1+m \cdot \gamma m(t^2)}\right]
 = \mu \ell^{1/2} c^2 + \ell^{1/2} c^2 \cdot \frac{-m \cdot  \gamma m(t^2)}{1+m \cdot \gamma m(t^2)}.\]

Expanding by using $Y_i = D_i X_i$, $X_i = \ell^{1/2} z_i u + \ep_i$, we have 

\beq
\label{expand_cross}
\E z_i (\hat u^\top Y_{i}) (u^\top \hat u) = \ell^{1/2} \E z_i^2 (\hat u^\top D_i u) (u^\top \hat u) +  \E z_i (\hat u^\top D_i \ep_{i}) (u^\top \hat u).
\eeq

Therefore, it is enough to show that $ \E z_i^2 (\hat u^\top D_i u) (u^\top \hat u) \to \mu c^2$ and $\E z_i (\hat u^\top D_i \ep_{i}) (u^\top \hat u) \to  \ell^{1/2} \cdot c^2 \cdot [-m \cdot \gamma m(t^2)]/[1+m \cdot \gamma m(t^2)].$ Since we already know that $(u^\top\hat u)^2 \to c^2$ a.s., our first step is to reduce these two claims by ``getting rid'' of the $u^\top\hat u$ term. For this and similar arguments, we will rely repeatedly on the following simple lemma:
\begin{lem}[Convergence Reduction]
\label{conv_reduce}
Suppose $W_n$,$Y_n$ are random variables such that $\E W_n \to D$ for a constant $D$, $ \E W_n^2$ is uniformly bounded, $Y_n \to C$ a.s. for some constant $C$, and $Y_n$ is a.s. uniformly  bounded. Then 
$\E W_n Y_n \to DC$.

In the special case when $C=0$, the statement $\E W_n \to D$ is not needed.
\end{lem}
\begin{proof}
We have $\E W_n Y_n  = \E W_n (Y_n-C) + C\E W_n$, so it is enough to show $\E W_n (Y_n-C) \to 0$.  Therefore we can assume without loss of generality that $C=0$. In that case, 
$$|\E W_n Y_n|\le|\E W_n^2|^{1/2} |\E Y_n^2|^{1/2} \le M_W |\E Y_n^2|^{1/2} \to 0,$$
because $ \E W_n^2$ are uniformly bounded; and because $Y_n \to 0$ a.s. and $Y_n$ are bounded, so $\E (Y_n)^2 \to 0$ by the dominated convergence theorem. Clearly, once we assumed $C=0$ we did not use $\E W_n \to D$. This finishes the proof.
\end{proof}

To use this lemma, let us choose the orientation of $\hat u$ such that $u^\top \hat u\ge 0$. Therefore, we know from our main results (e.g., Thm.\ \ref{spike_proj_multi}), that $u^\top \hat u\to c$ a.s. Moreover, since $\|\hat u\|=1$, it is easy to see that we can apply Lemma \ref{conv_reduce} to both terms in \eqref{expand_cross}. Specifically, for the first term, we apply it with $W_n = z_i^2 (\hat u^\top D_i u)$ and $Y_n = u^\top \hat u$; while for the second term we apply it with $W_n =z_i (\hat u^\top D_i \ep_{i})$ and $Y_n = u^\top \hat u$. As mentioned above, this effectively removes the $u^\top \hat u$ terms, and thus simplifies the analysis considerably. Indeed, by Lemma \ref{conv_reduce} we conclude that Part 2 follows from the following auxiliary convergence results, proved in Sec.\ \ref{denoise_aux_2_pf}: 

\begin{lem}[Denoising risk auxiliary limits: Part 2]
\label{denoise_aux_2}
We have the following convergence results: 
\benum
\item $ \E  z_i^2 (\hat u^\top D_i u)   \to \mu c$, where $c^2$ is defined in Eq. \eqref{cos_eq_white}.
\item $\E  z_i (\hat u^\top D_i \ep_{i})\to  \ell^{1/2} \cdot c \cdot \frac{-m \cdot \gamma m(t^2)}{1+m \cdot \gamma m(t^2)}$, where $t^2$ is defined in Eq. \eqref{spike_eq_white}, and $m$ is the Stieltjes transform of the standard Marchenko-Pastur law. 
\eenum
\end{lem}

Based on the above lemma, we completed the proof of Part 2 of Lemma \ref{denoise_lim}. We will prove Part 1 of Lemma \ref{denoise_lim} later in Sec.\ \ref{denoise_lim_p2}, because it re-uses many of the techniques and results established in the proof of Part 1. This finishes the proof of Lemma \ref{denoise_lim}.

\subsection{Proof of Lemma \ref{denoise_aux_2}}
\label{denoise_aux_2_pf}

In this lemma, we need to understand the asymptotics of inner products such as $u^\top D_i \hat u$. Previous results by \cite{benaych2012singular} characterized the asymptotics of the cosines $u^\top \hat u$. However, these results do not allow us to understand the asymptotics of the inner products $u^\top D_i \hat u$, due to the dependence of $\hat u$ and $D_i$. Instead, we must go back to first principles, and extend the \emph{outlier equation} technique introduced by \cite{benaych2012singular} to our setting. We will see that the current case is more challenging than the one handled in \cite{benaych2012singular}.

Let us recall the notation from Sec.\ \ref{setup}. According to \eqref{sig_noise_2}, the normalized data matrix $m^{-1/2}\tilde Y = (nm)^{-1/2}Y$ can be written as  $m^{-1/2} \tilde Y = \ell^{1/2}  \cdot Z u^\top + N$. Note here that $u = \nu$. Let here $t = t_p$ be the singular value of $m^{-1/2}\tilde Y$ with left and right singular vectors $\hat Z$, $\hat u$, and suppose that $t$ is not a singular value of $N$. As in Lemmas 4.1 and 5.1 of \cite{benaych2012singular}, we then have the following \emph{outlier equation} that provides an equation for the singular vectors of the perturbed matrix $m^{-1/2}\tilde Y$:
\beq
\label{outlier_eq}
\begin{bmatrix} 
t \cdot (t^2 I_n - NN^\top)^{-1}& 
 (t^2 I_n - NN^\top)^{-1} N \\ 
N^\top (t^2 I_n - NN^\top)^{-1} & 
t \cdot  (t^2 I_p - N^\top N)^{-1}
\end{bmatrix} 
\begin{bmatrix} 
(u^\top \hat u) \cdot Z \\
(Z^\top \hat Z)  \cdot u
\end{bmatrix} 
= 
\frac{1}{\ell^{1/2}}
\begin{bmatrix} \hat Z \\ \hat u \end{bmatrix}
\eeq

The idea is to \emph{take inner products} of the right hand side with the quantity we want to characterize (e.g., inner product with $D_iu$ to understand $u^\top D_i \hat u$), and then evaluate the limit of the quantities on the left hand side. This extends the technique of \cite{benaych2012singular}, who took inner products only with the population and sample singular vectors ($u$ and $\hat u$). In contrast, by using other vectors, such as $D_iu$ and $D_i\ep_{i}$, we are able to extend vastly the reach of their technique. 

To formalize this, let $w$ be a $p$-dimensional vector. By taking the inner product of the outlier equation's last $p$ coordinates with $w$, we obtain the following scalar equation:

\beq
\label{outlier_eq_2}
 (u^\top \hat u) \cdot w^\top N^\top (t^2 I_n - NN^\top)^{-1} Z +
t (Z^\top \hat Z) \cdot  w^\top (t^2 I_p - N^\top N)^{-1}u
=  \frac{1}{\ell^{1/2}} w^\top \hat u.
\eeq

Therefore, we can formalize the outlier equation technique as follows:

\begin{lem}[Outlier equation method]
\label{outlier_eq_lem}
Under the assumptions of Cor. \ref{standard_spike_proj}, suppose that $\hat u$, $\hat Z$ are chosen so that $\hat u^\top u\ge 0$, $\hat Z^\top Z \ge 0$. Suppose moreover that $w=w_p$ is a sequence of random vectors such that the assumptions of Lemma \ref{conv_reduce} for $W_n$ apply to $W_{n1} = w^\top N^\top (t^2 I_n - NN^\top)^{-1} Z$ and $W_{n2} =  w^\top (t^2 I_p - N^\top N)^{-1}u$. Specifically, suppose that $\E W_{n1} \to w^*_1$ and $\E W_{n2} \to w^*_2$. Then the random variables $w^\top \hat u$ also converge in expectation, namely
\beq
\label{outlier_eq_3}
\E w^\top \hat u \to\ell^{1/2} (c(\delta \ell;\gamma) w^*_1+ t  \cdot  \tilde c(\delta \ell;\gamma) w^*_2).
\eeq
Here $c\ge0$ where $c^2$ is defined in Eq. \eqref{cos_eq_white}, while $\tilde c(\delta \ell;\gamma) \ge 0$ is the limit cosine between $Z,\hat Z$, $\tilde c(\ell;\gamma)^2= (1-\gamma/\ell^2)/(1+1\ell)$  if $\ell>\gamma^{1/2}$ and $\tilde c^2=0$ otherwise.

As an important special case, suppose that $w_2^* = -\kappa m(t^2)$, while $w_1^* = 0$. Then 
\beq
\label{outlier_eq_4}
\E w^\top \hat u \to \kappa c.
\eeq
\end{lem}
\begin{proof}
The first part follows from the discussion before the lemma. For the second part, from Lemma \ref{qf_lem}, it folows that for $w = u$, we have $\E w^\top (t^2 I_p - N^\top N)^{-1}u \to - m(t^2)$, so that in this case the claim holds with $\kappa=1$. For other vectors $w$, the result follows by linearity.
\end{proof}

An analogous version of Lemma \ref{outlier_eq_lem} holds with convergence in expectation replaced by convergence a.s. Next, we will use this lemma to prove the two parts of  Lemma \ref{denoise_aux_2}. 

\subsubsection{Part 1}
\label{denoise_aux_lem_2_part_1}
Recall the notation $R=(t^2 I_n - NN^\top)^{-1}$
and define $\tilde R = (t^2 I_p - N^\top N)^{-1}$.
To show $ \E  z_i^2 (\hat u^\top D_i u)   \to \mu c$, we use Lemma \ref{outlier_eq_lem} with the sequence of vectors $w=z_i^2 D_i u $. To conclude our result by the second part of Lemma \ref{outlier_eq_lem}, it is enough to establish the following claims
\benum
\item $\E z_i^2 u^\top D_i N^\top R Z \to 0$
\item $\E z_i^2 u^\top D_i \tilde R u \to -\mu \cdot m(t^2)$.
\eenum

\emph{Claim 1}: As in the Proof of Lemma \ref{noise_lem} in Sec.\ \ref{pf_noise_lem}, we can write $N = A +E^{**}$, where $E^{**}$ is independent of $z$ and $\|E^{**}\| \to 0$. Therefore, it is easy to see that it is enough to prove the claim with $N$ replaced by $A$. Let us define $R(A) = (t^2 I_n - AA^\top)^{-1}$. Then, as in the proof of Part 2 of  Lemma \ref{qf_lem} in Sec.\ \ref{pf_qf_lem}, we obtain  $ u^\top D_i A^\top R(A) Z \to 0$ a.s. Clearly, these random variables are bounded a.s., since by the proof of Lemma \ref{noise_lem}, $R(A)$ has a.s. uniformly bounded operator norm. Therefore, by the convergence reduction Lemma \ref{conv_reduce} applied to $W_n=z_i$ and $Y_n= u^\top D_i A^\top R(A) Z $---valid since $\E z_i^2=1$ and $\E z_i^4$ is uniformly bounded---we conclude that $\E z_i^2 u^\top D_i A^\top R(A) Z \to 0$. As discussed, this implies the desired Claim 1.

\emph{Claim 2}: As in Claim 1, it is enough to prove the result with $N$ replaced by $A$. From now on in this claim, we will only work with $A$, not $N$; therefore, we denote for brevity $R=R(A)$, $\tilde R=\tilde R(A) = (t^2 I_p -A^\top A)^{-1}$ (and no confusion will arise). Also similarly to Claim 1, it will be enough to prove that $u^\top D_i \tilde R u \to -\mu m(t^2)$ a.s. 

For this, note that $D_i$ of course depends on the $i$-th data vector $Y_i=D_iX_i$. Therefore, we cannot use the concentration of quadratic forms directly. However, since the only dependence occurs in the $i$-th sample, we can use a rank-one perturbation formula to separate the $i$-th sample, and then control the two resulting terms separately.

For this, we define $a_j =n^{-1/2} D_j \ep_j$ to be the rows of $A$, so that $\tilde R = (t^2I_p -\sum_j a_j a_j^\top)^{-1}$. We also define the perturbed matrix $\tilde R_i = (t^2I_p -\sum_{j\neq i} a_j a_j^\top)^{-1}$. By the matrix inversion formula $(M+aa^\top)^{-1} = M^{-1} - M^{-1}aa^\top M^{-1}/(1+a^\top M^{-1} a)$, we have 
\beq
\label{rank_one_pert_mx}
\tilde R = \tilde R_i + \frac{ \tilde R_i a_i  a_i ^\top \tilde R_i}{1- a_i^\top \tilde R_i a_i}
\eeq
Therefore, 
\beq
\label{rank_one_pert}
u^\top D_i \tilde R u = u^\top D_i \tilde R_{i} u + \frac{(u^\top D_i \tilde R_i a_i)\cdot (u^\top \tilde R_i a_i)}{1- a_i^\top \tilde R_i a_i}.
\eeq
As in the proof of Part 3 of  Lemma \ref{qf_lem} in Sec.\ \ref{pf_qf_lem}, we obtain  $ u^\top D_i \tilde R_i u +u^\top D_i u \cdot m(t^2) \to 0$ a.s. However, $u^\top D_i u = \sum_{j} u_{j}^2 F_j$, where $F_j$ are iid random variables with mean $\mu=\E D_{ij}$. Hence, it is easy to see that $u^\top D_i u  \to \mu $ a.s.  This shows that the first term in \eqref{rank_one_pert} converges to $-\mu \cdot m(t^2)$ a.s. and in expectation. 

It remains to show that the second term in \eqref{rank_one_pert} converges to $0$ in expectation. 
Since $1/(1- a_i^\top \tilde R_i a_i)$ is uniformly bounded a.s., it is enough to show that $\E(u^\top D_i \tilde R_i a_i)\cdot (u^\top \tilde R_i a_i) \to 0$. However, by independence of $\ep_i$ from $D_i, \tilde R_i $, we have $\E(u^\top D_i \tilde R_i a_i)\cdot (u^\top \tilde R_i a_i) = n^{-1} \E u^\top D_i  \tilde R_i D_i^2 \tilde R_i u = O(n^{-1})$ a.s., showing the desired claim. This finishes the proof of Claim 2, and hence that of Part 1 of Lemma \ref{denoise_aux_2}. 

\subsubsection{Part 2}
\label{denoise_aux_2_part_2}
For Part 2 of Lemma \ref{denoise_aux_2}, we need to show that $\E  z_i (\hat u^\top D_i \ep_{i})\to  -\ell^{1/2} \cdot c m\cdot \gamma m(t^2)/[1+m \cdot \gamma m(t^2)]$. Using the Outlier Equation method, Lemma \ref{outlier_eq_lem}, as in Part 1, with the sequence of vectors $w=z_i D_i\ep_i$, it is enough to establish the following claims
\benum
\item $\E z_i \ep_{i}^\top D_i N^\top R Z \to \frac{-m \cdot \gamma m(t^2)}{1+m \cdot \gamma m(t^2)}$
\item $\E z_i \ep_{i}^\top D_i \tilde R u \to 0$.
\eenum

As before, we can work with $A$ instead of $N$, and with $R,\tilde R$ being the resolvents of $A$.  The second claim follows immediately, because $z_i$ is independent of $ \ep_{i}^\top D_i \tilde R u $, since $\tilde R$ does not depend on $z_i$. Therefore,  $\E z_i \ep_{i}^\top D_i \tilde R u = \E z_i \E \ep_{i}^\top D_i \tilde R u =0$.  

For the first claim, noting that $A^\top R = \tilde R A^\top$, and expanding $A^\top Z = \sum_j z_j a_j/n^{1/2}$, where $a_j =n^{-1/2} D_j \ep_j$, we can write 
\begin{align*}
\E z_i \ep_{i}^\top D_i A^\top R Z  &= 
\E z_i \ep_{i}^\top D_i \tilde R A^\top Z 
= \E z_i^2 \ep_{i}^\top D_i \tilde R a_i  +  \sum_{j \neq i} \E z_i z_j \ep_{i}^\top D_i \tilde R a_j \\
&= n^{-1} \E  \ep_{i}^\top D_i \tilde R D_i \ep_i 
= \E  a_{i}^\top \tilde R a_i
\end{align*}
since $z_i$ are independent of $ \ep, \tilde R$, and have mean 0. To control this last term, we use a technique similar to that in Sec.\ \ref{denoise_aux_lem_2_part_1}. Using the rank one perturbation formula \eqref{rank_one_pert_mx}, and denoting $\beta_i = a_i^\top \tilde R_i a_i = n^{-1}\ep_i^\top D_i \tilde R_i D_i \ep_i$, we find
\beqs
a_i^\top \tilde R a_i
= a_i^\top \tilde R_{i} a_i  + \frac{(a_i^\top \tilde R_{i} a_i)^2}{1- \beta_i} 
= \frac{\beta_i}{1- \beta_i}.
\eeqs
By independence and the concentration of quadratic forms, $\beta_i -n^{-1}m\cdot\tr(\tilde R_i ) \to 0$, while by the Marchenko-Pastur law, $n^{-1}\tr(\tilde R_i) \to -\gamma \cdot m(t^2)$. By the convergence reduction lemma,
$
\E  a_{i}^\top \tilde R a_i \to 
-m \cdot \gamma m(t^2)/[1+m \cdot \gamma m(t^2)]
$
This finishes the proof of Part 2 of Lemma \ref{denoise_aux_2}. Therefore, the proof of Lemma \ref{denoise_aux_2} is complete. 

\subsection{Multispiked EBLP MSE}
\label{denoise_pf_multi}
We now show that the formulas from Lemma \ref{denoise_lim} are also valid in the multispiked case. Indeed, 
$$
\E \|\ell_k^{1/2}z_{ik} u_k -\eta_k \hat u_k \hat u_{k}^\top Y_{i}\|^2 
= \ell_k + \eta_k^2 \E (\hat u_{k}^\top Y_{i})^2 -  2\eta_k\ell_k^{1/2} \E z_{ik} (\hat u_{k}^\top Y_{i}) (u_k^\top \hat u_k) .
$$ 
We have already showed in the proof of Lemma \ref{denoise_lim_multi} that the limit of the second term is the same as in the single-spiked case. It remains to characterize the third term. For this, we follow a similar pattern to that used in the proof of Lemma \ref{denoise_lim_multi}  in Sec.\ \ref{denoise_pf}, using the multispiked outlier equation. We sketch the argument below.

As in the proof of Thm.\ \ref{spike_proj_multi} in Sec.\ \ref{pf_multi}, the normalized data matrix $m^{-1/2}\tilde Y = (nm)^{-1/2}Y$ can be written as  $m^{-1/2} \tilde Y =Z L^{1/2} U^\top + N$, where $Z$ is the $n\times r$ matrix with entries $z_{ik}$, $U$ is the $n \times r$ matrix with columns $u_k$, while and $L$ is the $r\times r$ diagonal matrix with entries $\ell_k$. Let $t_k$ be the singular value of $m^{-1/2}\tilde Y$ with left and right singular vectors $\hat Z_k$, $\hat u_k$, and suppose that $t_k$ is not a singular value of $N$. The multispiked outlier equation is now:
\beq
\label{outlier_eq_ms}
\begin{bmatrix} 
t_k \cdot (t_k^2 I_n - NN^\top)^{-1}& 
 (t_k^2 I_n - NN^\top)^{-1} N \\ 
N^\top (t_k^2 I_n - NN^\top)^{-1} & 
t_k \cdot  (t_k^2 I_p - N^\top N)^{-1}
\end{bmatrix} 
\begin{bmatrix} 
Z \cdot L^{1/2} \cdot (U^\top \hat u_k) \\
U  \cdot L^{1/2} \cdot (Z^\top \hat Z_k) 
\end{bmatrix} 
=
\begin{bmatrix} \hat Z_k  \\ \hat u_k \end{bmatrix}
\eeq

For a sequence of random vectors $w=w_p$, denoting the $r$-vectors $W_{n1} = w^\top N^\top (t_k^2 I_n - NN^\top)^{-1} Z$ and $W_{n2} =  w^\top (t_k^2 I_p - N^\top N)^{-1}u$, the analogue of Eq. \eqref{outlier_eq_2} is obtained by taking the inner product of the last $p$ coordinates of the multispiked equation with $w$:
$$
W_{n1} \cdot L^{1/2} \cdot (U^\top \hat u_k) 
+t_k \cdot  W_{n2}  \cdot L^{1/2} \cdot (Z^\top \hat Z_k) 
=w^\top \hat u_k
$$
The outlier equation method formalized in Lemma \ref{outlier_eq_lem} also extends in the natural way, by taking the limits of the above equation. 

Going back to our main argument about EBLP risks, it remains only to characterize the  limit of $\smash{\E z_{ik} (\hat u_{k}^\top Y_{i}) (u_k^\top \hat u_k)}$. By the convergence reduction lemma, it is enough to show that $\smash{\E z_{ik} (\hat u_{k}^\top Y_{i}) \to \mu \ell_k^{1/2} c_k \beta_k}$, where $c_k \ge 0$ is the cosine, while $\beta_k = 1 +\gamma/(\delta\ell_k)$. Examining closely the proof of Lemma \ref{denoise_aux_2}, we see that all claims hold unchanged, and the proofs go through either unchanged or with minimal modifications (similar to the extension of the spike behavior to the multispiked case). We omit the details.

\subsection{Proof of Thm.\ \ref{oos-prop}}
\label{oos-prop-pf}

Similarly to our previous calculations, the MSE equals
\begin{align*}
E_n^{\eta,o} = \E \|S_0 -\hat S_0^\eta\|^2
&=\sum_{k=1}^r\E \|\ell_k^{1/2}z_{0k}u_k -\eta_k \hat u_k \hat u_{k}^\top Y_{0}\|^2 \\
&+\sum_{k\neq j}
 \E (\ell_k^{1/2}z_{0k} u_k -\eta_k \hat u_k \hat u_{k}^\top Y_{0})^\top
  (\ell_j^{1/2}z_{0j} u_j -\eta_j \hat u_j \hat u_{j}^\top Y_{0}).
\end{align*}

The cross terms vanish because $z_{0k}$ are independent zero-mean random variables, and $\hat u_k^\top \hat u_j = 0$. To find the limit, we expand the main term as follows:
$$
\E \|\ell_k^{1/2}z_{0k}u_k -\eta_k \hat u_k \hat u_{k}^\top Y_{0}\|^2
= \ell_k + \eta_k^2 \E (\hat u_{k}^\top Y_{0})^2 -  2\eta_k\ell_k^{1/2} \E z_{0k} (\hat u_{k}^\top Y_{0}) (u_k^\top \hat u_k)
$$
Now, because $Y_{0}$ are independent of $\hat u_j$, we can take expectation over the randomess in $Y_{0}$ to see that
\begin{align*}
&\E  \hat u_{k}^\top Y_{0}  \cdot 
\hat u_{k}^\top Y_{0} 
= \E  \hat u_{k}^\top  \left[ \mu^2\sum_{j=1}^r\ell_j  u_{j}  u_{j}^\top+mI_p +\sigma^2 \sum_{j=1}^r\ell_j  \diag (u_{j}\odot u_{j}) \right] \hat u_{k} \\
&=m
+  \mu^2 \sum_{j=1}^r \ell_j \E (\hat u_{k}^\top u_{j})^2
+\sigma^2 \sum_{j=1}^r\ell_j  \E  \hat u_{k}^\top \diag (u_{j}\odot u_{j})  \hat u_{k}
\to m +  \mu^2 \ell_k c_k^2 .
\end{align*}
On the last line we used the results of Lemma \ref{standard_spike_proj}, as well as the delocalization of $u_k$, and denoted $c_k^2 = c_{kk}^2$ the asymptotic cosine of the $k$-th singular vectors. Moreover  
\begin{align*}
\E z_{0k} (\hat u_{k}^\top Y_{0}) (u_k^\top \hat u_k)
&= \E z_{0k}\left(\sum_{m=1}^r\ell_m^{1/2} z_{0m}  \hat u_{k}^\top D_0 u_{m} + \hat u_{k}^\top D_0 \ep_0 \right)\cdot (u_k^\top \hat u_k) \\
&=  \ell_k^{1/2} \E \hat u_{k}^\top D_0 u_{k}\cdot (u_k^\top \hat u_k) 
\to \mu\ell_k^{1/2}c_k^2.
\end{align*}
Hence,  we have the following convergence: 
$$\E \|\ell_k^{1/2}z_{0k}u_k -\eta_k \hat u_k \hat u_{k}^\top Y_{0}\|^2 \to E^{\eta_k,o}(\ell_k) 
= \ell_k + \eta_k^2 \cdot (m +  \mu^2 \ell_k c_k^2) -  2\eta_k\cdot(\mu \ell_k c_k^2)$$
Therefore, the limit prediction error is $E^{\eta,o}  = \sum_{k=1}^r E^{\eta_k,o}(\ell_k)$, finishing the proof of Thm.\ \ref{oos-prop}.

\subsection{Proof of Prop. \ref{is-vs-oos-prop}}
\label{mse_equal}
First, clearly the optimal shrinkage for in-sample denoising is smaller, because $\mu\ell c^2/(\mu^2 \ell +m) \le \mu\ell c^2)/(\mu^2 \ell c^2+m).$ 
To check that the MSE of in-sample and out-of sample EBLP agree, we verify  that 
$\ell - \mu (\mu \ell  c^2  \cdot \beta)^2/(m t^2) = \ell - \mu (\mu \ell c^2)^2/(\mu^2 \ell c^2 +m).$
Clearly this holds whenever $c^2=0$. Considering the case $c^2>0$, it is enough to show that 
$ \beta^2/(m t^2) =1/(\mu^2 \ell c^2 +m),$
or also that
$(\delta\ell+\gamma)^2/(\delta\ell)^2=t^2/(\delta\ell c^2 + 1).$
Since $t^2 = (\delta\ell+1)[1+\gamma/(\delta\ell)]$, using the formula for $c^2$, this follows immediately, finishing the proof.

\subsection{Proofs for denoising in the unreduced-noise model (Sec.\ \ref{den_po_def2})}
\label{pf_unproj}
The proofs for denoising in the unreduced-noise model from \eqref{po_def2} are very similar to those in the reduced-noise model from \eqref{po_def}. Therefore, while we give the full outline of the proofs, we skip some of the more technical parts that are essentially identical to those in the proof of Thms. \ref{den_thm} and \ref{oos-prop}.

\subsubsection{BLP}
\label{pf_unproj_blp}

Since $S_i = \sum_{k=1}^{r} \ell_k^{1/2} z_{ik} u_k$ and $D_i$ has iid entries with mean $\mu$, we have $\Cov{S_i,Y_i} = \mu  \sum_{k=1}^r\ell_k  u_k u_{k}^\top $. Moreover, 
\[
\Cov{Y_i,Y_i} = \E Y_i Y_i^\top = \E D_i S_i S_i^\top  D_i + \E \ep_i  \ep_i^\top = 
\E D_i  \sum_{k=1}^r\ell_k  u_k u_{k}^\top  D_i + I_p.
\]
Since the entries of $D_i$ have variance $\sigma^2$, the $a,b$-th entry of the first component equals 
$
\sum_{k=1}^r \ell_k u_{ka}u_{kb} \cdot  \E D_{ia} D_{ib} 
= \sum_{k=1}^r \ell_k u_{ka}u_{kb} [\mu^2 + \sigma^2 I(a=b)].
$
Therefore, denoting $\diag(u_{k}\odot u_k)$ the diagonal matrix with entries $\{u_{ki}^2\}$, $i=1,\ldots,p$,
\[
\Cov{Y_i,Y_i} = \mu^2 \sum_{k=1}^r\ell_k  u_k u_{k}^\top + I_p + \sigma^2  \sum_{k=1}^r\ell_k  \diag(u_{k}\odot u_k).
\]
Since $u_k$ are delocalized, the operator norm of the last term converges to zero, $\|\diag(u_{k}\odot u_k)\|_{op} \le C^2 \log^2B(p)/p \to 0$. Hence by using the matrix inversion difference formula $A^{-1}-B^{-1} = A^{-1}( B-A)B^{-1}$, for $A = \Cov{Y_i,Y_i}$ and $B = \mu^2 \sum_{k=1}^r\ell_k  u_k u_{k}^\top + I_p$, so that $\|A-B\|\to 0$, we see that $\|\Cov{S_i,Y_i}A^{-1} Y_i - \Cov{S_i,Y_i}B^{-1} Y_i\|\to 0$. Thus the BLP is asymptotically equivalent to 
$
\mu \sum_{k=1}^r\ell_k  u_k u_{k}^\top \left[ \mu^2 \sum_{k=1}^r\ell_k  u_k u_{k}^\top + I_p \right]^{-1} Y_i.
$
As in the reduced-noise model in Sec.\ \ref{sv_shr_equiv},  it is not hard to show that this is asymptotically equivalent to
$\hat S_i^\tau = \sum_{k=1}^r\tau_k  u_k u_{k}^\top Y_i$,
where $\tau_k = \mu \ell_k/(\mu^2 \ell_k + 1)$.  

To calculate the MSE of $S_i^\tau$ for general $\tau$ in the single-spiked case, we write as in the proof of Thm \ref{den_thm} in Sec.\ \ref{den_pfs}
\[
E\|S_i -\hat S_i^{\tau}\|^2 
= \E(\ell^{1/2}z_i -  \tau u^\top Y_{i})^2  = 
\ell + \tau^2 \E (u^\top Y_{i})^2  -  2\tau \ell^{1/2} \E z_i u^\top Y_{i}.
\]
But $Y_i = D_i S_i +\ep_i$ and in the single-spiked case $S_i = \ell^{1/2} z_i u$, so 
$ \E z_i u^\top Y_{i} = \E \ell^{1/2} z_i^2 u^\top D_i u + \E z_i u^\top\ep_i = \ell^{1/2} \mu$.
Moreover, 
\[
\E (u^\top Y_{i})^2 = \E \ell z_i^2 (u^\top D_i u)^2 + 2\ell^{1/2} \E z_i u^\top D_i u u^\top \ep_i + \E (u^\top\ep_i)^2 = \ell \E (u^\top D_i u)^2 + 1.
\]
Now, denoting by $u(k)$ the entries of $u$
\[
\E (u^\top D_i u)^2 = \E (\sum_k u(k)^2 D_{ik}) (\sum_l u(l)^2 D_{il}) = \mu^2 \sum_{k\neq l} u(k)^2 u(l)^2 +(\mu^2+ \sigma^2) \sum_k u(k)^4.
\]
Since $u$ is delocalized, $\sum_k u(k)^4 \to 0$, so that $\E (u^\top D_i u)^2 \to \mu^2$. In conclusion, we obtain as required $ E\|S_i -\hat S_i^{\tau}\|^2 \to 
\ell + \tau^2  (\ell \mu^2 + 1)  -  2\tau \ell \mu$.

\subsubsection{EBLP}
\label{pf_unproj_eblp}

To find the MSE of EBLP in the single-spiked case, we have 
\[
\E\|S_i -\hat S_i^{\eta}\|^2 
= \E \|\ell^{1/2}z_{i} u -\eta \hat u \hat u^\top Y_{i}\|^2
= \ell + \eta^2 \E(\hat u^\top Y_{i})^2 -  2\eta \ell^{1/2} \E z_i \hat u^\top u \cdot \hat u^\top Y_{i}.
\]
As in Lemma \ref{denoise_lim}, we find that $\E(\hat u^\top Y_{i})^2 \to t^2(\mu^2\ell;\gamma)$, where $t^2(\mu^2\ell;\gamma)$ is defined in Eq. \eqref{spike_eq_white}. Also, by the convergence reduction Lemma \ref{conv_reduce}, choosing the orientation of $\hat u$ such that $\hat u^\top u \ge 0$, for the last term it is enough to characterize the limit of  $\E z_i \hat u^\top Y_{i}$. 

By the same argument as in Lemma \ref{denoise_lim}, we find that 
$
\E z_i \cdot \hat u^\top Y_{i}  \to \mu\ell^{1/2} \cdot  c(\mu^2\ell;\gamma) \cdot [1+\gamma/(\mu^2\ell)].
$
Hence,
\[
\E\|S_i -\hat S_i^{\eta}\|^2 \to \ell + \eta^2 \cdot t^2(\mu^2\ell;\gamma) 
-  2\eta \cdot  \mu\ell \cdot  c^2(\mu^2\ell;\gamma) \cdot [1+\gamma/(\mu^2\ell)],
\]
as claimed. This also shows that the optimal coefficient is $\eta^* = \mu \ell \cdot c^2(\mu^2\ell;\gamma)/[\mu^2\ell+1]$.

\subsubsection{EBLP Out-of-sample}
\label{pf_unproj_eblp_oos}

Finally, for out-of-sample EBLP AMSE in the single-spiked case, we let $(Y_0,D_0)$ be a new sample, and expand
$$
\E \|\ell^{1/2}z_{0}u -\eta \hat u \hat u^\top Y_{0}\|^2
= \ell + \eta^2 \E (\hat u^\top Y_{0})^2 -  2\eta\ell^{1/2} \E z_{0} (\hat u^\top Y_{0}) (u^\top \hat u)
$$
Now, because $Y_{0}$ is independent of $\hat u$, we can take expectation over the randomess in $Y_{0}$ to see that
\begin{align*}
&\E  \hat u^\top Y_{0}  \cdot 
\hat u^\top Y_{0} 
= \E  \hat u^\top   [I_p+\mu^2\ell u u^\top + \sigma^2 \diag(u\odot u)] \hat u \\
&=1 + \mu^2\ell \E (u^\top \hat u)^2 +\sigma^2\E  \hat u^\top \diag(u\odot u) \hat u
\to 1 + \mu^2\ell c^2 .
\end{align*}
On the last line we used that $u$ is delocalized, and the results of Lemma \ref{standard_spike_proj}. Next,  
\begin{align*}
\E z_{0} (\hat u^\top Y_{0}) (u^\top \hat u)
&= \E z_{0}\left(\ell^{1/2} z_{0}  \hat u^\top D_0 u + \hat u^\top\ep_{0}\right)\cdot (u^\top \hat u) \\
&=  \ell^{1/2} \E \hat u^\top D_0 u\cdot (u^\top \hat u) 
\to \ell^{1/2} \mu c^2.
\end{align*}
Hence,  we have the following convergence to the desired answer: 
$\E \|\ell^{1/2}z_{0}u -\eta \hat u \hat u^\top Y_{0}\|^2 \to 
 \ell + \eta^2 (1 + \mu^2\ell c^2) -  2\eta\mu \ell c^2$.
As claimed, the optimal coefficient is $\eta^* = \mu \ell c^2/(1 + \mu^2\ell c^2)  $, while the optimal MSE is $ \ell(1+\mu^2\ell c^2 s^2)/(1+\mu^2 \ell c^2)$.

{\small
\setlength{\bibsep}{0.2pt plus 0.3ex}
\bibliographystyle{plainnat-abbrev}
\bibliography{references}
}

\end{document}